\documentclass[11pt,reqno]{amsart}
\usepackage{amsaddr}
\usepackage{amsfonts}
\usepackage{amsthm}
\usepackage{amsmath}
\usepackage{amssymb}
\usepackage{pictex}
\usepackage{comment}
\usepackage{anyfontsize}

\usepackage{hyperref}

\usepackage[enableskew]{youngtab}

\usepackage{graphicx}
\usepackage{enumitem}

\usepackage{pgf,tikz}
\usetikzlibrary{arrows}
\usetikzlibrary{arrows.meta}
\usetikzlibrary{decorations.pathreplacing}
\usepackage{dynkin-diagrams}

\usepackage[acronym,symbols]{glossaries}

\makeglossaries

\newglossaryentry{sample}{name={sample},description={an example}}

\newglossaryentry{bn}{name={\ensuremath{B_n}},
  description={The hyperoctahedral group, i.e. the free group generated by $r_1, \dots, r_{n-1},r_{n}$ subject to the relations
  \ref{hyper1}-\ref{hyper4}  }}

\newglossaryentry{symptab}{name={\ensuremath{\textsf{KN}(\lambda,n)}},
  description={The $U(\mathfrak{sp}(2n,\mathbb{C}))$ crystal of Kashiwara--Nakashima tableaux of shape $\lambda$ in the alphabet
  $\mathcal{C}_n$}}

\newglossaryentry{r3}{name={\ensuremath{\bf R3}},
  description={The symplectic contraction/dilation relation in the symplectic plactic monoid $\mathcal{C}_{n}^\ast/\sim$}}

  \newglossaryentry{branchedsymptab}{name={\ensuremath{
\textsf{KN}_J(\lambda, n)}},
  description={The Levi branched crystal of Kashiwara--Nakashima tableaux
obtained by deleting in $\textsf{KN}(\lambda,n)$ all the arrows not
labeled in $J \subset I $}}

\newglossaryentry{Jn}{name={\ensuremath{J_n}},
  description={The cactus group $J_{\mathfrak{sl}(n,\mathbb{C})}$}}

  \newglossaryentry{vJ2n}{name={\ensuremath{ \widetilde{J}_{2n}}},
  description={The virtual symplectic cactus group}}

\newglossaryentry{Jsp}{name={\ensuremath{J_{\mathfrak{sp}(2n,\mathbb{C})}
}},  description={The symplectic cactus group  with generators $s_J$ for $J$
any connected sub-diagram of the $C_n$ Dynkin diagram subject to the
relations in Lemma~\ref{symplecticact}}}

\newglossaryentry{ssyt}{name={\ensuremath{ \textsf{SSYT}(\lambda,n)}},
  description={The $U_{q}(\mathfrak{sl}(n,\mathbb{C}))$-crystal of
semi-standard Young tableaux of shape $\lambda$ and entries in $[n]$}}

\newglossaryentry{ssytv}{name={\ensuremath{\textsf{SSYT}(\lambda^{A},n,\bar
n) }},  description={The $U_{q}(\mathfrak{sl}(2n,\mathbb{C}))$-crystal
of semi-standard Young tableaux of shape $\lambda^{A}$ and entries in
$\mathcal{C}_{n}$}}

\newglossaryentry{B}{name={\ensuremath{\textsf{B}}}, description={A normal crystal}}

  \newglossaryentry{cncrystal}{name={\ensuremath{\textsf{B}(\lambda)}},
  description={The normal $\mathfrak{g}$-crystal with highest weight
$\lambda$}}

   \newglossaryentry{Bj}{name={\ensuremath{\textsf{B}_J}},
  description={The Levi branched normal crystal $\textsf{B}_J$, the
restriction of $\textsf{B}$ to the sub-diagram $J$ of $I$}}

\newglossaryentry{Corderedletters}{name={\ensuremath{\mathcal{C}_{n}}},
  description={$\{1<\cdots<n<\bar n<\cdots<\bar 1\}$}}

    \newglossaryentry{Cstar}{name={\ensuremath{\mathcal{C}^{*}_{n}}},
  description={The monoid of words in the alphabet $\mathcal{C}_{n}$}}

\newglossaryentry{branchedssyt}{name={\ensuremath{\textsf{SSYT}_{J}(\lambda,
n) }},
  description={The Levi branched crystal, the restriction of $\textsf{SSYT}(\lambda, n)$  to  $J \subseteq
[n-1] $}}

\newglossaryentry{BK}{name={\ensuremath{ \mathcal{BK}}},
  description={The Berenstein--Kirillov group (or Gelfand--Tsetlin group)
\cite{bk95}}}

\newglossaryentry{BKc}{name={\ensuremath{ \mathcal{BK}^{C_n}}},
description={The type $C_n$ symplectic Berenstein--Kirillov group}}

\newglossaryentry{BKn}{name={\ensuremath{ \mathcal{BK}_n}},
  description={The subgroup of $\mathcal{BK}$ generated by the first
$n-1$ Bender--Knuth involutions $t_1,\dots,t_{n-1}$}}

\newglossaryentry{vBK2n}{name={\ensuremath{ \widetilde
{\mathcal{BK}}_{2n}}},
  description={The virtual symplectic Berenstein--Kirillov group, a
subgroup of $\mathcal{BK}_{2n}$ satisfying the relations of the virtual
symplectic cactus group $\widetilde{J}_{2n}$}}

  \newglossaryentry{E}{name={\ensuremath{E}},
  description={The virtualization map defined   by Baker
\cite[Proposition 2.2, Proposition 2.3]{ba00a} on type $C_n$ Kashiwara--Nakashima tableaux}}

  \newglossaryentry{xib}{name={\ensuremath{\xi_{\textsf{B}}}},
  description={The Sch\"utzenberger--Lusztig involution on the normal
crystal $\textsf{B}$ }}

  \newglossaryentry{xij}{name={\ensuremath{\xi_{\textsf{J}}}},
  description={The partial Sch\"utzenberger--Lusztig involution to the
sub-diagram $J\subseteq I$ is the Sch\"utzenberger--Lusztig involution
$\xi_{B_J}$ on the normal crystal $\gls{Bj}$}}

   \newglossaryentry{xi}{name={\ensuremath{\xi}},
  description={The  Sch\"utzenberger--Lusztig involution on
$\gls{cncrystal}$}}

\newglossaryentry{g}{name={\ensuremath{\mathfrak{g}}}, description={Finite dimensional, complex, semi-simple Lie algebra}}

\newglossaryentry{rev}{name={\ensuremath{\textsf{reversal}}}, description={Combinatorial procedure to compute the Sch\"utzenberger involution
$\xi$  on $\textsf{SSYT}(\lambda,n)$}}

\newglossaryentry{symprev}{name={\ensuremath{\textsf{reversal}^{C_n}}}, description={Combinatorial procedure to compute the Sch\"utzenberger
involution
$\xi$  on $\textsf{KN}(\lambda,n)$}}

\newglossaryentry{jrev}{name={\ensuremath{\textsf{reversal}_J}}, description={$J$-partial reversal, the reversal on
$\textsf{SSYT}_J(\lambda,n)$ with $J\subseteq [n-1]$}}

\newglossaryentry{jsymprev}{name={\ensuremath{\textsf{reversal}^{C_n}_J}}, description={$J$-partial symplectic reversal, the symplectic
reversal $\textsf{KN}_J(\lambda,n)$ with $J\subseteq [n]$ a connected sub-diagram containing the node $n$}}

\newacronym{ex}{EX}{example}

\newtheorem{theorem}{Theorem}
\newtheorem{proposition}{Proposition}
\newtheorem{remark}{Remark}

\newtheorem{corollary}{Corollary}
\newtheorem{example}{Example}
\newtheorem{definition}{Definition}

\newtheorem{lemma}{Lemma}

\newcommand{\op}{\operatorname}
\newcommand{\re}{\color{red}}
\newcommand{\pur}{\color{violet}}
\newcommand{\green}{\color{green}}

\newcommand{\mt}{\color{red}}
\newcommand{\oo}{\color{brown}}
\newcommand{\YT}[3]{
	\vcenter{\hbox{
			\begin{tikzpicture}[x={(0in,-#1)},y={(#1,0in)}] 
			\foreach \rowi [count=\i] in {#3} {
				\foreach \e [count=\j] in \rowi {
					\draw (\i,\j) rectangle +(-1,-1);
					\draw (\i-0.5,\j-0.5) node {$#2\e$};
				}
			}
			\end{tikzpicture}
	}}
}

\newcount\cols
{\catcode`,=\active\catcode`|=\active
 \gdef\Young#1{\hbox{$\vcenter
 {\mathcode`,="8000\mathcode`|="8000
  \def,{\global\advance\cols by 1 &}%
  \def|{\cr
        \multispan{\the\cols}\hrulefill\cr
        &\global\cols=2 }%
  \offinterlineskip\everycr{}\tabskip=0pt
  \dimen0=\ht\strutbox \advance\dimen0 by \dp\strutbox
  \halign
   {\vrule height \ht\strutbox depth \dp\strutbox##
    &&\hbox to \dimen0{\hss$##$\hss}\vrule\cr
    \noalign{\hrule}&\global\cols=2 #1\crcr
    \multispan{\the\cols}\hrulefill\cr%
   }
 }$}}
\gdef\Skew(#1:#2){\hbox{$\vcenter
{\mathcode`,="8000\mathcode`|="8000
  \dimen0=\ht\strutbox \advance\dimen0 by \dp\strutbox
  \def\boxbeg{\vbox
    \bgroup\hrule\kern-0.4pt\hbox to\dimen0\bgroup\strut\vrule\hss$}%
  \def\boxend{$\hss\egroup\hrule\egroup}%
  \def,{\boxend\boxbeg}%
  \def|##1:{\boxend\vrule\egroup\nointerlineskip\kern-0.4pt
    \moveright##1\dimen0\hbox\bgroup\boxbeg}%
  \def\\##1\\##2:{\boxend\vrule\egroup\nointerlineskip\kern-0.4pt
    \kern ##1\dimen0\moveright##2\dimen0\hbox\bgroup\boxbeg}%
  \moveright#1\dimen0\hbox\bgroup\boxbeg#2\boxend\vrule\egroup
 }$}}
}

\title[Symplectic cacti, virtualization and Berenstein--Kirillov
groups  ]{Symplectic cacti, virtualization and Berenstein--Kirillov
groups  }

\author{Olga Azenhas}
\address{University of Coimbra, CMUC, Department of Mathematics}
\email{\href{oazenhas@mat.uc.pt}{oazenhas@mat.uc.pt}}

\author{Mojdeh Tarighat Feller}
\address{Department of Mathematics, University of Virginia}
\email{\href{mt3cb@virginia.edu}{mt3cb@virginia.edu}}

\author{Jacinta Torres}
\address{Institute of Mathematics, Jagiellonian University in Krakow}
\email{\href{jacinta.torres@uj.edu.pl}{jacinta.torres@uj.edu.pl}}

\begin{document}

\begin{abstract}
We explicitly realize an internal action of the \textit{symplectic}
cactus group, recently defined by Halacheva for any complex, reductive, finite-dimensional Lie algebra, on crystals of Kashiwara--Nakashima
tableaux. Our methods include a symplectic version of jeu de taquin due to Sheats
and Lecouvey, symplectic reversal, and virtualization due to Baker. As
an application, we define and study a symplectic version of the
Berenstein--Kirillov group and show that it is a
quotient of the symplectic cactus group. In addition two relations for symplectic
Berenstein--Kirillov group are given that do not follow from the defining relations of the symplectic cactus group.
\end{abstract}


\maketitle
\tableofcontents

\section{Introduction}
The \textit{cactus group} was originally defined by Henriques--Kamnitzer
  \cite{hk06} in the context of coboundary categories defined  by
Drinfeld \cite{Drinfeld1990}.  It has appeared in connection with the study of moduli spaces of rational curves with $n+1$ marked
points \cite{dev99,ehkr10,w18,khwi19} and has been generalized to other Coxeter types under the name
 \textit{mock reflection group} in \cite{djs03}. Coboundary categories are monoidal
categories equipped with a \emph{commutor}, that is, a collection of
natural isomorphisms $\sigma_{A,B}: A\otimes B \rightarrow B\otimes A$
satisfying certain properties. The idea of studying the cactus group was
originally due to A. Berenstein and was taken up by Henriques--Kamnitzer
in \cite{hk06}, who defined it and further showed that it can be
realized as the fundamental group of the moduli space of marked real
genus zero stable curves. The original idea of Berenstein was to
construct a commutor in the category of crystals of a complex, reductive,
finite-dimensional Lie algebra, by first defining  an involution
$\gls{xib}: \gls{B} \rightarrow \gls{B}$ for each crystal
$\gls{B}$ which flips the crystal by exchanging highest weight
elements with lowest weight elements.  In the case of
$ \mathfrak{sl}(n,\mathbb{C})$  with the tableau model for the highest weight
crystal $\gls{cncrystal}$ it was known that
$\xi_{\gls{cncrystal}}$ coincides with the Sch\"utzenberger
involution on semi-standard Young tableaux of shape $\lambda$
\cite{bz96}.   See  \cite[Sections 4.3, 14.3.3]{BSch17} and the
references therein.

Let $\gls{g}$ be a complex, semi-simple Lie
algebra with Dynkin diagram $X$. There is a Dynkin diagram automorphism
$\theta: X \rightarrow X$ defined by $\alpha_{\theta(i)} =
-w_{0}\alpha_{i}$, where $w_{0}$ is the longest element of the Weyl
group  $W$ of $\gls{g}$. The \textit{cactus group $ J_{\gls{g}}$}, defined by Halacheva in  \cite{ha20, ha16}, is
the group generated by $\sigma_{I}$, where $I$ runs over all  connected
sub-Dynkin diagrams of $X$, subject to the following relations:
\begin{align}
\label{cactusrelations}
\sigma_{I}^{2} &= 1,\\ \label{cactusrelations2}
\sigma_{I} \sigma_{J} & = \sigma_{J}\sigma_{I} \hbox{ if }  J\subseteq
X,\, J\sqcup I \hbox{ is disconnected } \\ \label{cactusrelations3}
\sigma_{I} \sigma_{J} & = \sigma_{\theta_{I}(J)}\sigma_{I} \hbox{ if }
J\subset I
\end{align}
\noindent
where $\theta_{I}$ is the automorphism on $I$ defined by the longest
element of the parabolic group $W^I$. Halacheva has defined an
internal action of the cactus group $J_{\gls{g}}$ on a normal
$\gls{g}$--crystal by partial Sch\"utzenberger--Lusztig
involutions $\xi_I$.  From this action we know that partial
Sch\"utzenberger--Lusztig involutions satisfy the cactus group $
J_{\gls{g}}$ relations \cite{hakarywe20}.  Halacheva \cite{ha20}
initiated a combinatorial study of the cactus group for $\gls{g}=
\mathfrak{sl}(n,\mathbb{C})$ by comparing the action of $
\gls{Jn}=J_{\mathfrak{sl}(n,\mathbb{C})}$ on a normal $
\mathfrak{sl}(n,\mathbb{C})$-crystal  with that of the
Berenstein--Kirillov group on Gelfand--Tsetlin patterns (or semi-standard
Young tableaux) \cite{bk95}. Using a different approach, Chmutov, Glick
and Pylyavskyy \cite{cgp16} have also found relationships between those
two groups.

Our results compose a combinatorial  study of the cactus group for the
symplectic Lie algebra $\gls{g}=\mathfrak{sp}(2n,\mathbb{C})$.
There are many combinatorial models for $\mathfrak{sp}(2n,\mathbb{C})$-crystals: De Concini tableaux
\cite{deconcini}, King  tableaux \cite{King75},  Lakshmibai–-Seshadri
\cite{lakseshadri} and  Littelmann paths \cite{lit95,lit97}, the alcove path model of Gaussent--Littelmann \cite{GL05}
and the one of Lenart--Postnikov \cite{lenpos08}, but we work
with Kashiwara--Nakashima tableaux, for which a rich combinatorial
structure exists \cite{kn91,honkang,lec02,lec07}.  We review the basics in  Sections \ref{sec:basics} and \ref{sec:basicsKN}. For
each connected sub-Dynkin diagram $I$ of $X$, we define the explicit action of $\xi_{I}$ on a given Kashiwara--Nakashima tableau. The
algorithmic procedure for that action is given by virtualization. In the case when $I$ forms a Dynkin diagram of type $C_{n-k}$, it is also
given by the \emph{$I$-partial symplectic reversal}, a symplectic analogue of partial reversal on $A_{n-1}$ semi-standard Young tableaux.
Thereby we provide a
combinatorial action of the cactus generators  $\sigma_{I}$ on the set of Kashiwara--Nakashima tableaux on the alphabet
$\gls{Corderedletters}$. This is addressed in Sections \ref{sec:internalpartialschutzaction} and \ref{sec:partialschutzalg}.  The case of $I =
X$ has already been developed by Santos in \cite{sa21a}, where he defines an operation
on  straight shaped  Kashiwara--Nakashima tableaux which is a
\textit{symplectic analogue} of the Sch\"utzenberger involution
operation, also known as evacuation, on straight shaped $A_{n-1}$
semi-standard Young tableaux. This procedure includes the symplectic
jeu de taquin defined by Sheats in \cite{she99}, and  further developed
by Lecouvey \cite{lec02} using crystal isomorphisms. This is the content
of Section  \ref{sec:fullschutz}.

For $\ I \subseteq X$ such that $I$  forms a Dynkin diagram of type
$C_{n-k}$, we define an algorithm for \emph{ $I$-partial symplectic
reversal} which generalizes Santos' algorithm in the sense that, when $I
= X$, our algorithm is exactly the same. The symplectic $C_{n-k}$
reversal extends symplectic $C_{n-k}$ evacuation to arbitrary
semi-standard skew tableaux on the alphabet $\mathcal{C}_{n-k}$ whose
shift of the entries by $k$  are   admissible on the alphabet
$\gls{Corderedletters}$.  The $C_{n-k}$ reversal of such a
semi-standard skew tableau $P$ on the alphabet $\mathcal{C}_{n-k}$,
is  characterized to be the unique skew tableau coplactic equivalent to
$P$ and plactic  equivalent to the $C_{n-k}$ evacuation of the
symplectic rectification of $P$.

An important inspiration behind our generalization is the operation of \textit{tableau-switching} \cite{bss} of Benkart, Sottile and Stroomer
on $A_{n-1}$ semi-standard Young tableaux. Given an admissible tableau on the alphabet $\gls{Corderedletters}$, we start off by freezing the
entries corresponding to nodes not appearing in $I$, creating at the same time a new Young tableau $U$ with shape defined by the positive
frozen entries as well as a skew tableau $P$ consisting of the non-frozen entries.  The  tableau pair $(U,P)$, sharing a common border,  pass
through each other via \emph{symplectic jeu de taquin} (SJDT for short).
After performing this procedure, a new pair $(R,V)$ arises with $R$ the symplectic rectification of $P$ and $V$ consisting of the entries of
$U$ as well as some new, \emph{colored} letters. Each color records a precise instance of the symplectic rectification of $P$. Our
\emph{symplectic  colorful tableau switching} is  reversible since SJDT is reversible. It reduces to the $A_{n-1}$  tableau switching on
tableaux in the alphabet $[n]$. This work is carried out in detail in Subsection \ref{partialsymplecticreversal0} of this paper, yielding
Formula \eqref{formula}, and illustrated in Subsection \ref{ex:partialsymplecticreversal}.

For the general case we use the virtualization map defined by Baker
\cite{ba00a}, that is, an injective map
  \begin{align*}
E: \gls{symptab} \longrightarrow \gls{ssytv}
  \end{align*}
\noindent which assigns to the $\mathfrak{sp}(2n,\mathbb{C})$-crystal $\gls{symptab}$, a
subset of the $\mathfrak{sl}(2n, \mathbb{C})$-crystal $\textsf{SSYT}( \lambda^A,n,\,\bar
n)$ in a reversible way. This is discussed in Section
\ref{virtualization}. We show that one may apply the map $\gls{E}$, then
perform a certain partial Sch\"utzenberger--Lusztig involution in the
type $\mathfrak{sl}(2n, \mathbb{C})$-crystal without leaving the image of $\gls{E}$, reverse the virtualization map $\gls{E}$ and obtain our
desired result. In Subsection \ref{sec:embedsymplecticreversal} we show Theorem \ref{goingandbacR1} and Theorem \ref{goingandbacR2} using such
algorithmic procedures. Additionally, in Definition \ref{foldAcact},  Section \ref{sec:cacti}, we define the \textit{virtual symplectic cactus
group} $\gls{vJ2n}$  and show that it is a subgroup of $J_{2n}$ isomorphic to the symplectic cactus group $\gls{Jsp}$. In Theorem
\ref{cactusj2n}, Section {\ref{sec:internalpartialschutzaction}}, an action of the virtual symplectic cactus group on the set  $\gls{ssytv}$
is defined. The subset $\gls{E}(\gls{symptab})$ is preserved under this action as shown in Subsections \ref{sec:embedsymplecticreversal} and
\ref{sec:virtaction}. In particular, in Subsection \ref{sec:virtaction}, we realize such an action of the virtual symplectic cactus group on the
virtual images of Kashiwara--Nakashima tableaux and show that it \textit{virtualizes} the action of the symplectic cactus group on
Kashiwara--Nakashima tableaux.
This work is  illustrated in Section \ref{ex:virt}.

As an application, in Section \ref{sec:symplecticbk}, we define
\textit{symplectic} Bender--Knuth involutions combinatorially
(Definition  \ref{def:symplecticbenderknuth}). We start off by defining the type $C_n$
Berenstein--Kirillov group  $\gls{BKc}$ as the free group
generated by the  partial symplectic  \linebreak 
Sch\"utzenberger--Lusztig
involutions with respect to connected sub-diagrams of the type $C_n$
Dynkin diagram of the form $I=[n]$ modulo the relations they satisfy on
Kashiwara--Nakashima tableaux of any straight shape.  These generators of $\gls{BKc}$ satisfy the
relations of the symplectic cactus group (Theorem \ref{symplecticbk}). We show that symplectic
Bender--Knuth involutions are also generators of $\gls{BKc}$.

We study relations for $\gls{BKc}$ under the virtualization map $\gls{E}$. More precisely, for each generator of the symplectic Berenstein--Kirillov group, one can associate a corresponding product of generators in the Berenstein--Kirillov group of type $A_{2n-1}$ in a way analogous to the definition of the virtual symplectic cactus group. We call the group generated by them the \emph{virtual symplectic Berenstein--Kirillov group} $\gls{vBK2n}$ (see Definition \ref{def:virtualbk}). 
It is a subgroup of the type
$A_{2n-1}$ Berenstein--Kirillov group $\mathcal{BK}_{2n}$ satisfying, in particular, the relations of the virtual cactus group $\gls{vJ2n}$
(Theorem \ref{virtualsymplecticbk}). In Proposition \ref{prop:virtualbender-knuth} the  virtual symplectic Bender--Knuth involutions are defined, in an analogous way to how the 
generators of $\gls{vBK2n}$ are defined, as products of certain Bender--Knuth involutions in the type $A_{2n-1}$ Berenstein--Kirillov group. These are shown to be themselves also generators of $\gls{vBK2n}$. In Theorem \ref{thm:virtualbender-knuth2}, they are shown to be the virtualization of the symplectic Bender--Knuth
involutions, that is, they commute with the virtualization map $E$. The  virtual image of the group $ \gls{BKc}$ satisfies the relations of $\gls{vBK2n}$. More precisely, it is shown that the corresponding map $\gls{BKc} \rightarrow \gls{vBK2n}$ is a group isomorphism. Some of the relations listed in
Proposition \ref{rels:bkc} are obtained by applying the partial inverse to the virtualization map. Relations
\eqref{eq:symplecticbkrelationspecial} and \eqref{eq:symplecticbkrelationspecialbn} in Proposition \ref{rels:bkc} are the only ones that do
not follow from the defining relations of the symplectic cactus group $\gls{Jsp}$. They are instead equivalent to the braid relations of the
type $C_n$ Weyl group.

\section{Acknowledgements}
 This collaboration was undertaken within  the project \emph{The A, C, shifted Berenstein--Kirillov groups and cacti} in the framework of the
 ICERM program ``Research Community in Algebraic Combinatorics." All three authors were supported by the aforementioned ICERM program.  O.A.
 was also supported by the Centre for Mathematics of the University of Coimbra- UIDB/00324/2020, funded by the Portuguese Government through
 FCT/MCTES. M.T.F. was also supported by the grant NSF/DMS 1855804. J.T. was also supported by the grant SONATA NCN  UMO-2021/43/D/ST1/02290
 and partially supported by the grant MAESTRO NCN UMO-2019/34/A/ST1/00263. This  work benefited from computations using SageMath \cite{sagemath}.

The authors are grateful to the anonymous referees for their   
 careful reading.

\section{Basics}\label{sec:basics}
Let $\gls{g}$ be a finite-dimensional, complex, semi-simple Lie algebra. Let $I$ be the Dynkin diagram associated to the root system of
$\gls{g}$,  $\Delta=\{\alpha_i:i\in I\}$ the set of simple roots, $W$ its Weyl group, generated by the simple reflections $\{ r_{i}: i \in I
\}$, and $w_{0} \in W$ the longest Weyl group element. We will use the numbering of the vertices of $I$ given by \cite{bourbakirootsystems}.
The Dynkin diagram has an automorphism, a permutation of its nodes which leaves the diagram invariant,  $\theta:I\rightarrow I$   defined by
$\alpha_{\theta(i)} = -w_0\alpha_i$, for any node $i\in I$, where $w_0$ is the longest element of $W$. We will also denote by $\Lambda$ the
integral weight lattice associated to the root system of $\gls{g}$. It is generated by the fundamental weights $\omega_{i}$, $i \in I$. For a
connected sub-diagram of $I$, $J\subseteq I$, denote by $\theta_J:J\rightarrow J$ the Dynkin diagram automorphism that satisfies
$\alpha_{\theta_J (j)} = -w_0^J\alpha_j$, for any node $j\in J$, where $w_0^J$ is the longest element of the parabolic subgroup $W^J\subseteq
W$  (the Weyl group for  $\gls{g}$ restricted to $J$) \cite{BB05}. When $J=I$ one has the original notation $\theta_I=\theta$. We focus on the
cases where $\gls{g}=\mathfrak{sl}(n,\mathbb{C}), \mathfrak{sp}(2n, \mathbb{C})$. We will often abuse notation and write a Dynkin diagram $I$
with $n$ nodes as the interval $[n] =\{1<\dots <n\}$. The corresponding Weyl groups are the symmetric group $\mathfrak{S}_n$ on $n$
letters and the
hyperoctahedral group $\gls{bn}$ respectively, where  $\gls{bn}$ is the free group generated by $r_1, \dots, r_{n-1},r_{n}$ subject to the
relations
\begin{align} r_i^2 &= 1, \, 1 \le i \le n,\label{hyper1}\\
 (r_ir_j)^2 &= 1, \,1 \le i < j \le n,\, |i - j| > 1,\label{hyper2}\\
   (r_ir_{i+1})^3 &= 1, 1 \le i \le n - 2,\label{hyper3}\\
 (r_{n-1}r_n)^4&=1.\label{hyper4}
 \end{align}
The free group generated by $r_1, \dots, r_{n-1}$, subject to the relations above, for $1\le i, j<n$, is $\mathfrak{S}_n$ realized  by the
simple transpositions $r_i=(i,i+1)$ on the set $[n]$. The group $\gls{bn}$ has $2^nn!$ elements and is realized by the signed transpositions
$r_i=(i,i+1)(\overline{i},\overline{i+1})$, $i=1,\ldots,n-1$, and $r_n=(n,\overline{n})$ on the set $\{1<\cdots <n<\bar n<\cdots<\bar 2<\bar
1\}$.  That is, we may see $\gls{bn}$  embedded in  $\mathfrak{S}_{2n}$  by folding $\{1<\cdots <n<\bar n<\cdots<\bar 2<\bar 1\}$ through a
central
symmetry. The long element of $\gls{bn}$, $(r_n \cdots r_2 r_1)^n = (r_1 r_2 \cdots r_n)^n$, has length $n^2$, while the long element of
$\mathfrak{S}_n$ has length $n(n-1)/2$ \cite{BB05}.

Occasionally, for the sake of clarity, we write $w_0^A$ and $w_0^C$ for the corresponding longest elements of $\mathfrak{S}_n$ and $\gls{bn}$
respectively, or simply $w_{0}$ when there is no room for confusion. Given a vector $v\in \mathbb{Z}^n$, we have that $r_i$, with $i\in
[n-1]$, acts on $v$ by swapping the $i$-th and the $(i+1)$-st entries, and $r_n$ acts on $v$ by changing the sign of the last
entry. Henceforth, $w_0^A$ reverses $v$, $w_0^A(v_1,\ldots, v_n)=(v_n,\dots,v_1)$, and $w_0^C$ changes the sign of the entries of $v$,
$w_0^Cv=-v$.

Recall the $ \mathfrak{sl}(n,\mathbb{C})$ simple roots $\alpha_i=\textbf{e}_i-\textbf{e}_{i+1}$, $i\in [n-1]$, and the
$\mathfrak{sp}(2n,\mathbb{C})$ simple roots $\alpha_i=\textbf{e}_i-\textbf{e}_{i+1}$, $i\in [n-1]$ and $\alpha_n=2\textbf{e}_n$, where
$\textbf{e}_i$, $i\in [n]$, is the $\mathbb{R}^n$ standard basis. The $A_{n-1}$ Dynkin diagram automorphism above is given by $\theta
(i)=n-i$, since $-w_0\alpha_i=-(-\alpha_{n-i})=\alpha_{n-i}$, with $i\in I=[n-1]$.
 For instance,
 $$n=5\qquad A_4\qquad
\begin{tikzpicture}[xscale=0.6,yscale=0.6]
\node at (0,0) (1) {$\bullet$};
\node at (0,-0.4) {1};
\node at (2,0) (2) {$\bullet$};
\node at (2, -0.4) {2};
\node at (4,0) (3) {$\bullet$};
\node at (4, -0.4) {3};
\node at (6,0) (4) {$\bullet$};
\node at (6, -0.4) {4};

\draw[thick] (0,0) to (6,0);

\draw[stealth-stealth, thick] (1) to[bend left=22] node[above] {} (4);
\draw[stealth-stealth, thick] (2) to[bend left=22] node[above] {} (3);

\end{tikzpicture}$$
 $$n=6\qquad A_5\qquad
\begin{tikzpicture}[xscale=0.6,yscale=0.6]
\node at (0,0) (1) {$\bullet$};
\node at (0,-0.4) {1};
\node at (2,0) (2) {$\bullet$};
\node at (2, -0.4) {2};
\node at (4,0) (3) {$\bullet$};
\node at (4, -0.4) {3};
\node at (6,0) (4) {$\bullet$};
\node at (6, -0.4) {4};
\node at (8,0) (5) {$\bullet$};
\node at (8, -0.4) {5};

\draw[thick] (0,0) to (8,0);

\draw[stealth-stealth, thick] (1) to[bend left=22] node[above] {} (5);
\draw[stealth-stealth, thick] (2) to[bend left=22] node[above] {} (4);

\end{tikzpicture}$$

The $C_{n}$ Dynkin  diagram automorphism above is given by $\theta(i)=i$, since for $ w_0$ $\in \gls{bn}$, $-w_0 \alpha_i=
-(-\alpha_i)=\alpha_i$, with $i\in I=[n]$. The weight lattices are $\Lambda= \mathbb{Z}^n$ for $\mathfrak{sp}(2n,\mathbb{C})$ and $\Lambda =
\mathbb{Z}^n /
\langle(1,...,1)\rangle
$ for  $\mathfrak{sl}(n, \mathbb{C})$.  We will often work with representatives in the case of  $\mathfrak{sl}(n, \mathbb{C})$. The
fundamental weights are $\omega_{i} = \sum_{j=1}^{i}e_{i}$, $1 \leq
i \leq n$ and respectively have representatives $\omega_{i}$, $1 \leq i \leq n-1$.

\subsection{Levi sub-diagrams}

Let $I$ be a finite Dynkin diagram. A Levi sub-diagram $J$ of $I$ obtained by deleting from $I$ a subset of its nodes is the Dynkin diagram of
a semi-simple Lie algebra $\mathfrak{g}_{J} \subset \gls{g}$ known as a \textit{Levi sub-algebra} which is the Levi component of  the
parabolic Lie sub-algebra of $\gls{g}$ generated by the Chevalley generators associated to the nodes of $J$.

\begin{example}
If we remove the last node (the one labeled by $n$) from the Dynkin diagram of type $C_{n}$, we obtain a Dynkin diagram of type $A_{n-1}$
 which corresponds to the Levi sub-algebra $\mathfrak{sl}(n,\mathbb{C})$ of $\mathfrak{sp}(2n,\mathbb{C})$.

$$C_n \qquad  \dynkin[edge length=1cm,labels*={1,2,3,n-1,n}] C{***.**}$$

$$A_{n-1} \qquad  \dynkin[edge length=1cm,labels*={1,2,3,n-1}] A{***.*}$$

\end{example}

\begin{example}
 The semi-simple Lie algebra $\mathfrak{sl}(3,\mathbb{C} ) \times \mathfrak{sp}(4,\mathbb{C})$ is a Levi sub-algebra of
 $\mathfrak{sp}(12,\mathbb{C} )$. Note that the semi-simple Lie algebra $\mathfrak{sl}(n,\mathbb{C}) \times \mathfrak{sl}(2,\mathbb{C})$ is not
 a Levi sub-algebra of $\mathfrak{sp}(2n,\mathbb{C})$, as its Dynkin diagram of type $A_{n-1}\times A_1$ cannot be obtained from the type
 $C_n$ diagram by deleting some of its vertices.
\end{example}

\section{Normal crystals and Levi restrictions}\label{sec:basicsKN}
 Crystals corresponding to finite-dimensional (quantum group) $U_q(\gls{g})$-representa-\break tions belong to a family of crystals called
 \textit{normal crystals} \cite{BSch17,hakarywe20}. In classical types, these crystals may be  realized  by a tableau model  \cite{kn91} and
 have nice combinatorial properties.
 These crystals decompose into connected components, one for each
 irreducible component to the representation at hand. The Levi restriction of a normal crystal is still a normal crystal, and the union of
 some connected components of a normal crystal is also a normal crystal \cite{BSch17,hakarywe20}. The crystals that we deal with are tableau
 crystals for finite-dimensional representations of $ U_q(\mathfrak{sl}(n,\mathbb{C}))$ and $U_q({\mathfrak{sp}(2n,\mathbb{C})})$.

  A \emph{$\gls{g}$-crystal} is a finite set $\gls{B}$ along with maps
$$ \textsf{wt} :\gls{B}\rightarrow \Lambda, \quad e_i, f_i :\gls{B}\rightarrow \gls{B} \sqcup \{0\},\,
\varepsilon_i,\varphi_i:\gls{B}\rightarrow \mathbb{Z},\,$$
satisfying the following axioms for any $b, b'\in \gls{B}$ and $i \in I$,

\begin{itemize}	
\item $b' = e_i(b)$ if and only if $b = f_i(b')$,
\item if $f_i(b) \neq 0 $ then $\textsf{wt}(f_i(b)) = \textsf{wt}(b)-\alpha_i$;\\
if $e_i(b) \neq 0$, then
$\textsf{wt}(e_i(b)) = \textsf{wt}(b)+\alpha_i$, and
\item $\varepsilon_i(b)=\max\{a \in \mathbb{Z}_{\geq 0} :e_i^a(b)\neq 0\}$ and
$\varphi_i(b)=\max\{a \in \mathbb{Z}_{\geq 0 }:f_i^a(b)\neq 0\},$
\item  $\varphi_i(b)-\varepsilon_i(b)=  \langle \textsf{wt}(b),\alpha_i^\vee  \rangle$,
\end{itemize}

\noindent
where $\alpha_i^\vee= \frac{2\alpha_i}{ \langle \alpha_i,\alpha_i \rangle}$ are the simple coroots.

\begin{remark}
Our abstract $\gls{g}$-crystals are defined with the additional condition that they are seminormal \cite{BSch17}.
\end{remark}

The \emph{crystal graph} of $\gls{B}$  is the  directed graph with vertices in $\gls{B}$ and edges labeled by $i\in I$. If $f_i(b)=b'$ for
$b, b'\in \gls{B} $, then we draw an edge $b\stackrel{i}{\rightarrow} b'$.  See Example \ref{graph}. Given an arbitrary subset $J\subseteq I$, $\gls{Bj}$ is
defined to be the crystal $\gls{B}$ restricted to the sub-diagram $J$ of $I$, the \emph{Levi branched crystal}. The crystal graph  of $\gls{Bj}$ has
the same vertices as $\gls{B}$, but the arrows are only those labeled in $J$; that is, we forget the maps $e_i,f_i, \varphi_i,$ and
$\varepsilon_i$, for $i\not\in J$ \cite{BSch17}. The weight map, which we denote by $\textsf{wt}_J$, is $\gls{B}\overset
{\textsf{wt}}\rightarrow \Lambda \overset{can}{\rightarrow} \Lambda_{J}$, where $\textsf{wt}$ is the weight map of $\gls{B}$, $\Lambda$ is the
weight lattice of $\gls{g}$, $\Lambda_{J} = \Lambda/ \langle\omega_{i} : i \notin J \rangle$ is the weight lattice of $\mathfrak{g}_{J}$, and $ \Lambda
\overset{can}{\rightarrow} \Lambda_{J}$ is the canonical projection. If $\gls{g} = \mathfrak{sp}(2n,\mathbb{C})$ and we restrict to $J =
[n-1]$, then we obtain an $\mathfrak{sl}(n,\mathbb{C})$-crystal. Given $b\in \gls{B}$, $\gls{B}(b)$ denotes the connected component of
$\gls{B}$ containing $b$.

A $\gls{g}$-crystal is \emph{normal} if it is isomorphic to a disjoint union of the crystals $\gls{cncrystal}$, where $\gls{cncrystal}$ is the
crystal associated to an irreducible, finite-dimensional \linebreak $\gls{g}$-representation \emph{of highest weight} $\lambda$, where
$\lambda \in \Lambda$ is a \emph{dominant weight}. In this work, where we focus on $\gls{g} = \mathfrak{sp}(2n,\mathbb{C})$, respectively
$\gls{g} = \mathfrak{sl}(n,\mathbb{C})$, dominant weights in $\mathbb{Z}^{n}$, respectively in $ \mathbb{Z}^{n}/\langle(1,...,1)\rangle$, correspond
precisely to \emph{partitions} with at most $n$ parts, that is, weakly decreasing vectors in $\mathbb{Z}^{n}$ with non-negative entries,
respectively to weakly
decreasing vectors in $\mathbb{Z}^{n}$, and each such representative is equivalent to a unique partition in $\mathbb{Z}^{n-1} \hookrightarrow
\mathbb{Z}^{n}$, where the last entry is fixed as zero. An important property of normal crystals \gls{B} is the existence of a unique highest
weight vertex for each connected component of \gls{B}, that is, an element which is a source in the corresponding crystal graph, whose weight
is dominant. Note that we work solely with
\emph{highest weight crystals}, namely, crystals $\gls{B}$ such that for each $b\in \gls{B}$, there exists a finite sequence $a_1, a_2,\dots,
a_l\in I$ and a highest weight element $u_{b} \in \gls{B}(b)$  such that $b = f_{a_l}\cdots f_{a_2}f_{a_1}(u_{b})$. For $b, b' \in \gls{B}$,
we have $B(b) = B(b')$ if and only if $u_{b} = u_{b'}$.
From now on, we will refer to $\mathfrak{sp}(2n,\mathbb{C})$-crystals by $C_n$-crystals, and $\mathfrak{sl}(n,\mathbb{C})$-crystals by
$A_{n-1}$-crystals.

\subsection{Kashiwara--Nakashima tableaux}
Let \gls{cncrystal} be the irreducible $C_n$-crystal  with highest weight a partition $\lambda$ of at most $n$ parts.
We realize \gls{cncrystal} as the crystal \gls{symptab} of Kashiwara--Nakashima tableaux \cite{kn91} of shape $\lambda$ on the alphabet

\[ \gls{Corderedletters} =\{1<\cdots<n<\bar n<\cdots<\bar 1\}. \]

The irreducible $A_{n-1}$-crystal  with highest weight a partition $\lambda$ of at most $n$ parts is realized as the crystal $\gls{ssyt}$ of
semi-standard tableaux of  shape $\lambda$ on the alphabet $ [n]$. We also will refer to these tableaux as the $A_{n-1}$ tableaux of shape
$\lambda$. The crystal $\gls{ssyt}$ is a connected sub-crystal of  $\gls{symptab}$. The \emph{weight} of an $A_{n-1}$ tableau $T$,
respectively a Kashiwara--Nakashima tableau $U$, is represented by, respectively is, the vector $(\mu_{1},...,\mu_{n}) \in \mathbb{Z}^{n}$,
where $\mu_{i}$ denotes the number of $i$'s in $T$, respectively the number of $i$'s minus the number of $\bar i$'s in $U$.

Kashiwara--Nakashima tableaux (KN for short) are semi-standard Young tableaux in the alphabet \gls{Corderedletters} which satisfy some extra
conditions. They are a variation of De Concini symplectic tableaux \cite{deconcini}. A semi-standard Young tableau of any shape (skew or
straight) with entries in \gls{Corderedletters} is KN if and only if the following two conditions hold:

\begin{itemize}
\item Each one of its columns is admissible.
\item Its splitting is a semi-standard Young tableau.
\end{itemize}

\begin{definition} \label{def:admissible}
Let $C$ be a semi-standard column in the alphabet \gls{Corderedletters} of length at most $n$. Let $Z = \left\{z_{1} > ... > z_{m} \right\}$
be the set of non-barred letters $z$ in $\gls{Corderedletters}$ such that both $z$ and $\bar z$ appear in $C$. We say that the column $C$ is
admissible if there exists a set $T = \left\{t_{1} > ... > t_{m}  \right\}$ of unbarred letters $t$ that satisfies:
\begin{itemize}
    \item $t, \bar t \notin C$;
    \item $t_{1} < z_{1}$ and is maximal with this property;
    \item $t_{i} < min(t_{i-1}, z_{i})$ and is maximal with this property.
\end{itemize}
The split of a column is the two-column tableau $lC rC$ where $lC$ is the column obtained from $C$ by replacing $ z_{i}$ by $ t_{i}$ and
possibly re-ordering, and $rC$ is obtained from $C$ by replacing $ \bar z_{i}$ by $\bar t_{i}$ and possibly re-ordering. The splitting of a
tableau consisting of admissible columns is the concatenation of the splits of its columns.
\end{definition}
Given $\mu\subseteq \lambda$ partitions with at most $n$ parts,  $\textsf{KN}(\lambda/\mu,n)$ denotes the normal $C_n$-crystal  of KN
tableaux of skew shape $\lambda/\mu$ on the alphabet \gls{Corderedletters}  \cite[Lemma 6.1.3, Corollary 6.3.9]{lec02}.

\begin{example}
Let $n = 2$. The column $\Skew(0:\hbox{\tiny{$2$}}|0: \hbox{\tiny{$\bar 2$}})$ is admissible, however, $\Skew(0:\hbox{\tiny{$1$}}|0:
\hbox{\tiny{$\bar 1$}})$ is not. Notice that although each one of its columns is admissible, the tableau
$\Skew(0:\hbox{\tiny{$2$}},\hbox{\tiny{$2$}}|0: \hbox{\tiny{$\bar 2$}}, \hbox{\tiny{$\bar 2$}} )$ is not KN, because its split,
$\Skew(0:\hbox{\tiny{$1$}},\hbox{\tiny{$2$}}, \hbox{\tiny{$1$}},\hbox{\tiny{$2$}} |0: \hbox{\tiny{$\bar 2$}}, \hbox{\tiny{$\bar 1$}},
\hbox{\tiny{$\bar 2$}}, \hbox{\tiny{$\bar 1$}} )$, is not semi-standard.
\end{example}

 We will mostly use the notation and definitions from \cite{lec02,lec07}. We also refer the reader to the references therein.

\begin{remark}\cite[Remark 2.2.2]{lec02}
\label{lecouvey2.2.2}
The maximal height of an admissible column is $n$. Moreover, a column $C$ is admissible if and only if, for any $m \in [n]$, the number $N(m)$
of letters $x$ in $C$ such that either $x \leq m$ or $x \geq \bar m$ satisfies $N(m) \leq m$. Moreover, if there exists in $C$
a letter $m \leq n$ such that $N(m) >m$, then $C$ contains a pair $(z, \bar z)$ satisfying $N(z) > z$.

In \cite{lec02}, \textbf{coadmissible} columns are defined as well (see \cite[p.301]{lec02}). We will not delve into the details here,
however, we remark that there exists a bijection between admissible and coadmissible columns given by filling in the shape of the given
admissible column $C$ with the unbarred letters of $lC$ from top to bottom in increasing order, followed by the barred letters of $rC$ in the
same fashion. We will denote this bijection by $\Phi$ and use it in Section \ref{sjdt}.
\end{remark}
The \textit{reading word} of a KN tableau is the word in the alphabet  $\gls{Corderedletters}$ given by  the Chinese/Japanese reading of its columns (from
 top to bottom and right to left). Recall the $C_n$ \textit{signature rule} \cite{kn91,lec02,BSch17} to compute the action of the Kashiwara crystal operators on a word in the alphabet
$\gls{Corderedletters}$. (We refer to Section \ref{sec:solenodeweyl} for more details.)
\begin{example}\label{graph} The $C_2$ crystal $\textsf{KN}(\lambda,2)$ of shape $\lambda = (2,1)$. Each node in the graph represents an
 element of the crystal. There is a blue, respectively red, arrow  connecting an element $a$ to an element $b$ whenever $f_{1}(a) = b$,
 respectively $f_{2}(a) = b$ where $b$ is computed by $f_1$ respectively $f_2$ according to the $C_2$ signature rule on the reading word of $a$.
$${ \begin{array}{cccc}
 \begin{tikzpicture}
  [scale= 1.4,  auto=left]
\node (n0) at (4,10) {  {$\Skew(0:\mbox{\tiny{$1$}},\mbox{\tiny{$1$}} | 0:\mbox{\tiny{$2$}})$}};
  \node (n1l) at (3,9)  { {$ \Skew(0:\mbox{\tiny{$1$}},\mbox{\tiny{$2$}} | 0:\mbox{\tiny{$2$}})$}};
\node (n1r) at (5,9)  { {$ \Skew(0:\mbox{\tiny{$1$}},\mbox{\tiny{$1$}} | 0:\mbox{\tiny{$\bar 2$}})$}};
\node (n2l) at (3,7.5)  { {$ \Skew(0:\mbox{\tiny{$1$}},\mbox{\tiny{$\bar 2$}} | 0:\mbox{\tiny{$2$}})$}};
 \node (n2r) at (5,7.5) { {$ \Skew(0:\mbox{\tiny{$1$}}, \mbox{\tiny{$2$}} | 0: \mbox{\tiny{$\bar 2$}})$}};
\node (n3r) at (5,6)  { {$ \Skew(0:\mbox{\tiny{$2$}}, \mbox{\tiny{$2$}} | 0: \mbox{\tiny{$\bar 2$}})$}};
 \node (n4rr) at (6,5) { {$\Skew(0:\mbox{\tiny{$2$}},\mbox{\tiny{$2$}}|0:\mbox{\tiny{$\bar 1$}})$}};
\node (n4r) at (4,5) { {$\Skew(0: \mbox{\tiny{$2$}},\mbox{\tiny{$\bar 2$}} |0: \mbox{\tiny{$\bar 2$}})$}};
\node (n5r) at (5,4) { {\Skew(0:\mbox{\tiny{$2$}},\mbox{\tiny{$\bar 2$}}|0: \mbox{\tiny{$\bar 1$}})}};
\node (n6r) at (5,2.5) { {\Skew(0:\mbox{\tiny{$\bar 2$}},\mbox{\tiny{$\bar 2$}}|0:\mbox{\tiny{$\bar 1$}})}};
\node (n7) at (4,1.5) { { \Skew(0:\mbox{\tiny{$\bar 2$}}, \mbox{\tiny{$\bar 1$}} | 0: \mbox{\tiny{$\bar 1$}}) }};
\node (n3ll) at (2,6.5)  { {\Skew(0:\mbox{\tiny{$1$}}, \mbox{\tiny{$\bar2$}} | 0: \mbox{\tiny{$\bar 2$}})}};
\node (n4l) at (3,5.5) { { \Skew(0:\mbox{\tiny{$1$}}, \mbox{\tiny{$\bar 1$}} | 0: \mbox{\tiny{$\bar 2$}})}};
\node (n3l) at (4,6.5)  { { \Skew(0:\mbox{\tiny{$1$}}, \mbox{\tiny{$\bar 1$}} | 0: \mbox{\tiny{$2$}})}};
\node (n5l) at (3,4) { {\Skew(0:\mbox{\tiny{$2$}}, \mbox{\tiny{$\bar 1$}} | 0: \mbox{\tiny{$\bar 2$}})}};
\node (n6l) at (3,2.5)  { { \Skew(0:\mbox{\tiny{$2$}}, \mbox{\tiny{$\bar 1$}} | 0: \mbox{\tiny{$\bar 1$}}) }};

    \foreach \from/\to in {n0/n1r,n6l/n7,n3l/n4l,  n2l/n3ll,n1l/n2l,n3r/n4r,n4rr/n5r,n5r/n6r}
    \draw[->] [draw=red,  thick](\from) -- (\to);

    \foreach \from/\to in {n0/n1l,n1r/n2r,n2r/n3r,n3r/n4rr,n4r/n5r,n6r/n7,n3ll/n4l,n5l/n6l,n4l/n5l,n2l/n3l}
 \draw [->] [draw=blue,  thick](\from) -- (\to);
\end{tikzpicture}
\end{array}}$$
\end{example}

\subsubsection{Levi branching of KN tableau crystals }\label{subsubsec:signature}
For $J\subseteq I, \hbox{   } \gls{branchedsymptab}$  is the restriction of \gls{symptab} to the sub-diagram $J$ of $I$: as a crystal graph it
has the same set of vertices as \gls{symptab} but only contains the arrows labeled by $J$, and it is also a normal crystal. The highest
weight elements  of \gls{branchedsymptab} are those $C_n$ tableaux in \gls{symptab} where the only incoming edges are colored in $[n]\setminus
J$.

If $ J=[p, q]$, $1\le p\le q \leq n$, the crystal graph  \gls{branchedsymptab} consists of the KN tableaux of \gls{symptab}  with arrows
colored
in $J$.

 If $q<n$,  the Levi branched crystal $\textsf{KN}_{[p,q]}(\lambda,n)$  is a type
 $A_{q-p+1}$ normal crystal. The Weyl group  is  $W^J=\mathfrak{S}_{[p,q+1]}$,  the symmetric group on the letters $\{p,\dots,q+1\}$ and
 generators $r_j=(j, j+1) (\bar j ,\overline {j+1}), \hbox{   } j\in J.$
We say that the entries outside of
 \begin{align*}[\pm p, q+1] := \{p<\dots<q+1\} \sqcup \{ \overline{q+1}<\cdots<\overline {p}\} \end{align*} are \emph{frozen}, which amounts to saying that the KN
 tableaux of the set \gls{symptab} in the same connected component of $\textsf{KN}_{[p,q]}(\lambda,n)$ are stable in the entries over $
 \gls{Corderedletters}
 \setminus [\pm p, q+1]$ under the action of the Kashiwara operators $f_i$, $e_i$, $i\in [p,q]$. That is, if $q<n$, in the same
 connected component of $\textsf{KN}_{[p,q]}(\lambda,n)$, the subtableaux consisting of the letters $[1,p-1]=\{1<\cdots<p-1\}$, $[\overline{p-1},\bar 1] =\{
 \overline{p-1}<\cdots<\overline {1}\}$ or  $ [\pm (q+2),n]=\{ {q+2}<\cdots <n<\bar n<\cdots<\overline {q+2}\}$ are the same.

 If $q = n,$ the Levi
 branched crystal
 $\textsf{KN}_{[p,n]}(\lambda,n)$  is isomorphic to a type $C_{n-p+1}$ normal crystal. The Weyl group  is  $W^J=B_{[ p,n]}$ generated
 by the signed permutations on the subset $\{p<\dots<n<\bar n<\dots< \bar p\}$. The entries outside of
 \begin{align*}[\pm p, n] = \{p<\dots<n< \overline{n}<\cdots<\overline {p}\} \end{align*} are \emph{frozen}; within the same connected component of
 \gls{branchedsymptab}, the subtableaux either consisting of the letters $\{1<\cdots<p-1\}$ or $ \{
 \overline{p-1}<\cdots<\overline {1}\}$ are the same.
 In Example \ref{ex:levi}, since $\mathfrak{sl}(2,\mathbb{C}) =
 \mathfrak{sp}(2,\mathbb{C})$, we get two crystals of types $A_{1} = C_{1}$.

\begin{example}\label{ex:levi}
\begin{figure}[ht]
  \centering
$$\begin{array}{cccc}
 \begin{tikzpicture}
  [scale= 1.35,  auto=left]
\node (n0) at (4,10) { {$\Skew(0:\mbox{\tiny{$1$}},\mbox{\tiny{$1$}} | 0:\mbox{\tiny{$2$}})$}};
  \node (n1l) at (3,9)  {{$ \Skew(0:\mbox{\tiny{$1$}},\mbox{\tiny{$2$}} | 0:\mbox{\tiny{$2$}})$}};
\node (n1r) at (5,9)  {{$ \Skew(0:\mbox{\tiny{$1$}},\mbox{\tiny{$1$}} | 0:\mbox{\tiny{$\bar 2$}})$}};
\node (n2l) at (3,7.5)  {{$ \Skew(0:\mbox{\tiny{$1$}},\mbox{\tiny{$\bar 2$}} | 0:\mbox{\tiny{$2$}})$}};
 \node (n2r) at (5,7.5) { {$ \Skew(0:\mbox{\tiny{$1$}}, \mbox{\tiny{$2$}} | 0: \mbox{\tiny{$\bar 2$}})$}};
\node (n3r) at (5,6)  { {$ \Skew(0:\mbox{\tiny{$2$}}, \mbox{\tiny{$2$}} | 0: \mbox{\tiny{$\bar 2$}})$}};
 \node (n4rr) at (6,5) { {$\Skew(0:\mbox{\tiny{$2$}},\mbox{\tiny{$2$}}|0:\mbox{\tiny{$\bar 1$}})$}};
\node (n4r) at (4,5) { {$\Skew(0: \mbox{\tiny{$2$}},\mbox{\tiny{$\bar 2$}} |0: \mbox{\tiny{$\bar 2$}})$}};
\node (n5r) at (5,4) { {\Skew(0:\mbox{\tiny{$2$}},\mbox{\tiny{$\bar 2$}}|0: \mbox{\tiny{$\bar 1$}})}};
\node (n6r) at (5,2.5) { {\Skew(0:\mbox{\tiny{$\bar 2$}},\mbox{\tiny{$\bar 2$}}|0:\mbox{\tiny{$\bar 1$}})}};
\node (n7) at (4,1.5) { { \Skew(0:\mbox{\tiny{$\bar 2$}}, \mbox{\tiny{$\bar 1$}} | 0: \mbox{\tiny{$\bar 1$}}) }};
\node (n3ll) at (2,6.5)  { {\Skew(0:\mbox{\tiny{$1$}}, \mbox{\tiny{$\bar2$}} | 0: \mbox{\tiny{$\bar 2$}})}};
\node (n4l) at (3,5.5) { { \Skew(0:\mbox{\tiny{$1$}}, \mbox{\tiny{$\bar 1$}} | 0: \mbox{\tiny{$\bar 2$}})}};
\node (n3l) at (4,6.5)  { { \Skew(0:\mbox{\tiny{$1$}}, \mbox{\tiny{$\bar 1$}} | 0: \mbox{\tiny{$2$}})}};
\node (n5l) at (3,4) { {\Skew(0:\mbox{\tiny{$2$}}, \mbox{\tiny{$\bar 1$}} | 0: \mbox{\tiny{$\bar 2$}})}};
\node (n6l) at (3,2.5)  { { \Skew(0:\mbox{\tiny{$2$}}, \mbox{\tiny{$\bar 1$}} | 0: \mbox{\tiny{$\bar 1$}}) }};

  \foreach \from/\to in {n0/n1r,n6l/n7,n3l/n4l,  n2l/n3ll, n1l/n2l,n3r/n4r, n4rr/n5r,n5r/n6r}
    \draw[->] [draw=red,  thick](\from) -- (\to);
\end{tikzpicture}
\vline
&
 \begin{tikzpicture}
  [scale= 1.35,  auto=left]
\node (n0) at (4,10) {  {$\Skew(0:\mbox{\tiny{$1$}},\mbox{\tiny{$1$}} | 0:\mbox{\tiny{$2$}})$}};
  \node (n1l) at (3,9)  { {$ \Skew(0:\mbox{\tiny{$1$}},\mbox{\tiny{$2$}} | 0:\mbox{\tiny{$2$}})$}};
\node (n1r) at (5,9)  { {$ \Skew(0:\mbox{\tiny{$1$}},\mbox{\tiny{$1$}} | 0:\mbox{\tiny{$\bar 2$}})$}};
\node (n2l) at (3,7.5)  { {$ \Skew(0:\mbox{\tiny{$1$}},\mbox{\tiny{$\bar 2$}} | 0:\mbox{\tiny{$2$}})$}};
 \node (n2r) at (5,7.5) { {$ \Skew(0:\mbox{\tiny{$1$}}, \mbox{\tiny{$2$}} | 0: \mbox{\tiny{$\bar 2$}})$}};
\node (n3r) at (5,6)  { {$ \Skew(0:\mbox{\tiny{$2$}}, \mbox{\tiny{$2$}} | 0: \mbox{\tiny{$\bar 2$}})$}};
 \node (n4rr) at (6,5) { {$\Skew(0:\mbox{\tiny{$2$}},\mbox{\tiny{$2$}}|0:\mbox{\tiny{$\bar 1$}})$}};
\node (n4r) at (4,5) { {$\Skew(0: \mbox{\tiny{$2$}},\mbox{\tiny{$\bar 2$}} |0: \mbox{\tiny{$\bar 2$}})$}};
\node (n5r) at (5,4) { {\Skew(0:\mbox{\tiny{$2$}},\mbox{\tiny{$\bar 2$}}|0: \mbox{\tiny{$\bar 1$}})}};
\node (n6r) at (5,2.5) { {\Skew(0:\mbox{\tiny{$\bar 2$}},\mbox{\tiny{$\bar 2$}}|0:\mbox{\tiny{$\bar 1$}})}};
\node (n7) at (4,1.5) { { \Skew(0:\mbox{\tiny{$\bar 2$}}, \mbox{\tiny{$\bar 1$}} | 0: \mbox{\tiny{$\bar 1$}}) }};
\node (n3ll) at (2,6.5)  { {\Skew(0:\mbox{\tiny{$1$}}, \mbox{\tiny{$\bar2$}} | 0: \mbox{\tiny{$\bar 2$}})}};
\node (n4l) at (3,5.5) { { \Skew(0:\mbox{\tiny{$1$}}, \mbox{\tiny{$\bar 1$}} | 0: \mbox{\tiny{$\bar 2$}})}};
\node (n3l) at (4,6.5)  { { \Skew(0:\mbox{\tiny{$1$}}, \mbox{\tiny{$\bar 1$}} | 0: \mbox{\tiny{$2$}})}};
\node (n5l) at (3,4) { {\Skew(0:\mbox{\tiny{$2$}}, \mbox{\tiny{$\bar 1$}} | 0: \mbox{\tiny{$\bar 2$}})}};
\node (n6l) at (3,2.5)  { { \Skew(0:\mbox{\tiny{$2$}}, \mbox{\tiny{$\bar 1$}} | 0: \mbox{\tiny{$\bar 1$}}) }};

 \foreach \from/\to in {n0/n1l,n1r/n2r,n2r/n3r,n3r/n4rr,n4r/n5r,n6r/n7,n3ll/n4l,n5l/n6l,n4l/n5l,n2l/n3l}
 \draw [->] [draw=blue,  thick](\from) -- (\to);
\end{tikzpicture}
\end{array}
$$
\caption{The Levi branched crystals $\textsf{KN}_{\left\{
2\right\}}(\lambda,2)$ and $\textsf{KN}_{\left\{ 1\right\}}(\lambda,2)$ respectively from left to right for $\lambda = (2,1)$. \label{fig:levicrystal}}
\end{figure}

 In Figure \ref{fig:levicrystal}, we have the Levi branched crystals $\textsf{KN}_{\left\{
2\right\}}(\lambda,2)$ and \linebreak$\textsf{KN}_{\left\{ 1\right\}}(\lambda,2)$ respectively from left to right for $\lambda = (2,1)$. Both are
$A_1$-crystals.

The highest weights, with multiplicity, in the LHS have representatives $(2,1),$ $ (1,2),$ $  (1,0)$, $(0,1),$ $ (0,1),$ $(-1,2)$,
$  (-1,0),$ $
(-2,1) $. In the quotient of $\mathbb{Z}^{2}$ by the fundamental weight $\omega_{1} = (1,0)$, these are equivalent to  the vectors
$(0,1),(0,2), (0,0)$, $(0,1), (0,1), (0,2),$ $ (0,0), (0,1)$, respectively. In practice, this means that we have ignored the multiplicity of
the
letters $\{1, \bar 1\}$ in the tableaux of the LHS to compute the highest weights. On the RHS, we consider another embedding of $\mathbb{Z}
\hookrightarrow \mathbb{Z}^{2}$ given by the quotient $ \mathbb{Z}^{2}/ \langle (1,1)\rangle$,  since $\omega_{2} = (1,1)$. The computation of the highest
weights on the RHS is similar to that of the LHS, and we thus leave it as an exercise for the reader.

\end{example}

\section{Virtualization}
\label{virtualization}
In this section we closely follow Baker \cite[Section 2]{ba00a} and adopt the notation used there. In Example \ref{ex:virt}, we present a
detailed example of the content in this section. We include it later rather than earlier because it includes some more information which is
not
yet presented up to the end of this section.
\subsection{Baker embedding and Baker recording tableau}\label{bakerembedding} Let
$\lambda = \lambda_{1}\omega_{1}+ \cdots +\lambda_{n}\omega_{n} \in \mathbb{Z}^{n}$  with $ \omega_{j} =  \sum^{j}_{i=1} e_{i} \in
\mathbb{Z}^{n}$, $1 \leq j \leq n,$ the fundamental weights of type $C_n$. Let
\begin{align*}
\omega^{A}_{j} &= \sum^{j}_{i=1} e_{i} \in \mathbb{Z}^{2n} \hbox{ for } 1 \leq j \leq n, \hbox{ and  }\\
\omega^{A}_{\bar j}  = \omega^{A}_{2n-j+1} &= \sum^{2n - j +1}_{i=1} e_{i} \in \mathbb{Z}^{2n} \hbox{ for } 1 < j \leq n,
\end{align*}
be the $A_{2n-1}$ fundamental weights, and consider as well the $\mathbb{Z}^{2n}$ partition
\begin{align}\label{eq:lambdaA}
    \lambda^{A} = 2 \lambda_{n} \omega^{A}_{n} + \sum_{i = 1}^{n-1} \lambda_{i} (\omega^{A}_{i} + \omega^{A}_{\overline {i+1}}).
\end{align}
 Let \gls{ssytv} be the $A_{2n-1}$ crystal of semi-standard Young tableaux in the alphabet $\gls{Corderedletters}$ of shape
 $\lambda^{A}$. We will denote the corresponding crystal operators
   by $f^{A}_{i}$ for $i \in \gls{Corderedletters}$ and consider, for $1 \leq i \leq n$, the operators
    $f^{E}_{i} = f^{A}_{i}    f^{A}_{\overline{i+1}}$, $i< n$, and $f^{E}_{n} =(f^{A}_{n})^{2} $.
    Let $\gls{E}$ denote the \emph{virtualization map} defined on type $C_n$ Kashiwara--Nakashima tableaux defined by Baker  \cite[Proposition
    2.2, Proposition 2.3]{ba00a}.
    More precisely, $\gls{E}$ is an injective map
 \begin{align}  \label{ssytv}
 \gls{E}: \gls{symptab} \hookrightarrow  \gls{ssytv}
 \end{align}

 \noindent
 such that $\gls{E}(f_{i}(T)) =f^{E}_{i}(\gls{E}(T)) $ for $T \in   \gls{symptab}, 1\le i\le n$. We will denote by $E^{-1}$ the restriction of
 any left inverse of $\gls{E}$ to the image of $\gls{symptab}$ under $\gls{E}$.

\noindent
 Given an admissible column  $C$ in the alphabet $\gls{Corderedletters}$ of shape $\omega_i$, $1\le i\le n$, denote by $\psi(C)$ its
  \textit{Baker virtual split} \cite[Proposition 2.2]{ba00a}, a two column type $A_{2n-1}$ tableau of shape $\omega^{A}_i+\omega^{A}_{2n-i}$.
  The map $\psi$ is injective and embeds admissible columns of length $i$,  in the alphabet $\gls{Corderedletters}$, into
  $\textsf{SSYT}(\omega^{A}_{i} + \omega^{A}_{2n-i}, n, \bar n), 1 \leq i \leq n$. We define $\psi^{-1}$ analogously to $E^{-1}$. From
  \cite[Proposition 2.3]{ba00a} we know that, if we write $T$ as a concatenation of its columns, that is, $T = C_{k}\cdots C_{1}$,
  then
\[ \gls{E}(T)=[\emptyset \leftarrow w(\psi(C_{1})) \leftarrow \cdots \leftarrow w(\psi(C_{k}))], \]

  \noindent
 where the word  $w(\psi(C))$ of the type $A_{2n-1}$ two-column tableau $\psi(C)$  is given by  the Chinese/Japanese reading of its two columns (from
 top to bottom and right to left), and
$  P \leftarrow w$ is the Schensted column insertion of a word $w$ into a type $A_{2n-1}$ semi-standard Young tableau $P$  in the alphabet
$\gls{Corderedletters}$ \cite{Fu,Sta}.

Let $T_{\lambda} \in \gls{symptab}$ be the highest weight element; that is, $T_\lambda$ is the Yamanouchi tableau of shape and weight
$\lambda$ on the alphabet $[n]$ (each row $i$ is  solely filled with the letter $i$). Then $\gls{E}(T_{\lambda})=T_{\lambda^A}$ is the highest
weight element of \gls{ssytv}, that is, the $A_{2n-1}$ Yamanouchi tableau of shape and weight $\lambda^A$ in the alphabet
\gls{Corderedletters}. The  image  of \gls{symptab} under $\gls{E}$ in \gls{ssytv} is the crystal  generated by acting with the lowering
operators $f^{E}_{i}$ on the highest weight element $T_\lambda^{A}$ of \gls{ssytv}. For $T \in \gls{symptab}$, where $T = C_{k} \cdots C_{1}$,
we write
\[w_{T} = w(\psi(C_{1})) \cdots  w(\psi(C_{k})).\]

\noindent Then $w_{T}$ is a word in  $\gls{Cstar}$, the monoid of words in the alphabet $\gls{Corderedletters}$, and $\gls{E}(T) = [ \emptyset
\leftarrow w_{T}]$. We will call the \emph{recording
tableau of the column insertion of $w_{T}$}, $ Q(w_T)$, the \textit{Baker recording tableau} associated to $T$.

\begin{proposition}
For $T \in \gls{symptab}$, the Baker recording tableau  $Q(w_{T})$ depends only on $\lambda$. From now on, we will denote by $Q_{\lambda}$ the
Baker recording tableau associated to $\lambda$.
\end{proposition}

  \begin{proof}
By abuse of notation, we will denote by the same symbols the type $A_{2n-1}$ crystal operators on  the $A_{2n-1}$  crystal
$\gls{Cstar}$ of words and those on semi-standard Young tableaux in the same alphabet. Now, we know that there exists
a sequence $1 \leq i_{1},..., i_{k} \leq n$ such that $f_{i_{k}}    \cdots    f_{i_{1}}(T_{\lambda}) = T$. Therefore  $f^{E}_{i_{k}}    \cdots
f^{E}_{i_{1}}(\gls{E}(T_{\lambda})) = \gls{E}(T)$,
where $\gls{E}(T_{\lambda})=T_{\lambda^A}$, the highest weight element of $\gls{ssytv}$, and so
\[f^{E}_{i_k}    \cdots    f^{E}_{i_1}  (w_{T_{\lambda}}) = w_{T}\]
 because the connected components of the
crystal $\gls{Cstar}$ of words of type $A_{2n-1}$ with highest weight elements $w_{T_\lambda}$ and $w(\gls{E}(T_\lambda))=w(T_{\lambda^A})$
have the same highest weight $\lambda^A$ and are hence isomorphic. In particular, both $w_T$ and  $w_{T_\lambda}$ belong to the same connected
component of the crystal $\gls{Cstar}$ of words of type $A_{2n-1}$, namely, the connected component containing the Yamanouchi word
$w_{T_{\lambda}}$ of weight $\lambda^A$ (recall that all words $w_{T}$ have the same rectification shape $\lambda^{A}$ and that all $A_{2n-1}$
crystal operators commute with \textit{jeu de taquin}). Now, we consider a version of the RSK correspondence \cite{Fu,Sta,hoon,BSch17}
which is a bijection
\begin{align}
\label{rsk}
\gls{Cstar}  &\overset{1:1}{\longleftrightarrow} \underset{\begin{smallmatrix} \mu \\ \ell(\mu)\leq 2n\end{smallmatrix}}{\bigsqcup}
\textsf{SSYT}(\mu,n,\bar n) \times \textsf{SYT}(\mu),&
w &\overset{\operatorname{RSK}}{\mapsto} (P(w), Q(w)).
\end{align}
\noindent
where $\textsf{SYT}(\mu) $ is the set of standard Young tableaux of shape $\mu, P(w) = [ \emptyset \leftarrow w]$ and $Q(w)$ is the
corresponding recording tableau which encodes the sequence of shapes produced by the column insertion of $w$. In particular for each standard
Young tableau $Q$ of shape $\mu$ the pre-image $\operatorname{RSK}^{-1}(\textsf{SSYT}(\mu,n,\bar n) \times \{Q\})$ is a crystal isomorphic to
$\textsf{SSYT}(\mu,n,\bar n)$, and all of these pre-images are disjoint and cover $\gls{Cstar}$. In particular this means that all the words
$w_{T}$ for $T \in \gls{symptab}$ are contained in the same connected component of $\gls{Cstar}$ defined by:
\[\operatorname{RSK}^{-1}( \hbox{\gls{ssytv}} \times \{Q(w_{T_{\lambda}})\}).\]

 \noindent Thereby, $Q(w_{T}) = Q(w_{T_{\lambda}}) $ for all $T \in \gls{symptab}$.
 \end{proof}

\begin{corollary}  Let  $\lambda=\omega_{m_1}+\cdots+\omega_{m_k}$, $1\le m_1\le\cdots\le m_k\le n$, and let
\[
\lambda^A=\omega^A_{2n-m_1}+\omega^A_{m_1}+\cdots+\omega^A_{2n-m_k}+\omega^A_{m_k}\in \mathbb{Z}^{2n}.
\]
Then  $Q_\lambda$ can be written out of the shape $\lambda^A$ as a sequence of shapes by adding successively the columns
$\omega^A_{m_1},\omega^A_{2n-m_1},\dots, \omega^A_{m_k},\omega^A_{2n-m_k}$, whose boxes are filled along columns, top to bottom with
consecutive
numbers from 1 to $|\lambda^A|$:
\begin{align*} \emptyset
&\subset
 \omega^A_{m_1}\\
 &\subset\omega^A_{2n-m_1}+{\omega^A_{m_1}}\\
 &\subset \omega^A_{m_{2}}+\omega^A_{2n-m_1}+\omega^A_{m_1}\\
 &\subset
 \omega^A_{2n-m_{2}}+\omega^A_{m_{2}}+\omega^A_{2n-m_1}+\omega^A_{m_1
 }\\
 & \;\;\vdots\\
 &\subset \omega^A_{m_k}+\cdots+\omega^A_{2n-m_{2}}+\omega^A_{m_{2}}+\omega^A_{2n-m_1}+\omega^A_{m_1} \\
  &\subset \omega^A_{2n-m_k}+\omega^A_{m_k}+\cdots+\omega^A_{2n-m_1}+\omega^A_{m_1} =\lambda^A.
  \end{align*}
 \end{corollary}

Given a partition $\lambda$ with at most $n$ parts, and $T=C_k\cdots C_1\in \gls{symptab}$,  let $\displaystyle\Psi(T)=(w(\psi(C_1)),$
$\ldots, w(\psi(C_k)))\in \gls{Cstar}$ (here the word is presented as a $k$-tuple) and $\Psi^{-1}=\underbrace{(\psi^{-1},\dots,\psi^{-1})}_k$.
Then $(\gls{E}(T) ,Q_\lambda)=\textrm{RSK}\Psi(T)=(P(w_T),Q_\lambda)$ and
$$E^{-1}=\Psi^{-1}\textrm{RSK}^{-1}_{\mid \gls{E}(\gls{symptab}) \times \{Q_\lambda\}}$$
where $\textrm{RSK}^{-1}_{\mid\gls{symptab} \times \{Q_\lambda\}}$ denotes the inverse of $\textrm{RSK}$ restricted to $
\gls{E}(\gls{symptab}) \times \{Q_\lambda\}$.

 The computation of $\textrm{RSK}^{-1}_{\mid\gls{symptab} \times
\{Q_\lambda\}}$ uses $Q_\lambda$ to perform the inverse of column Schensted insertion. See  example on Section \ref{ex:virt}.

\subsection{The Levi branched crystal and virtualization} \label{branchvirt}
Recall that a  Levi branched crystal $\gls{Bj}$, $J\subseteq
I$, $I$ a Dynkin diagram,
is obtained by ignoring the maps $f_i,e_i,\varphi_i,\varepsilon_i$, for $i \notin J$.  Let $I$ be the $ A_{2n-1}$ Dynkin diagram with nodes
$\{1,\dots,n,\bar n,\dots, \bar 2\}$.
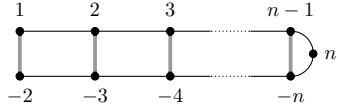
 $$\\~\\ \dynkin[edge length=1cm,fold,labels={ 1,2,3,n-1,n,    \overline{n},  \overline{4},   \overline{3},   \overline{2}}]A{***.***.***}$$

 For each connected sub-diagram
$J=[p,q]$ or $[k,n]$ with $1\le p\le q <n$ and  $k\le n$, of $[n]$,
let
 $\bar {J}=[\overline{q+1},\overline{ p+1}]$ respectively $[ \bar n,\overline{k+1}]$, if $k< n$, and $\bar J=\emptyset$ if $k=n$ be the
 corresponding connected
 sub-diagram of $[\bar n,\bar 2]$.

 Each connected component  of the Levi branched crystal $\textsf{KN}_{J\sqcup \bar J}(\lambda, n)$ with $J=[p,q]$, $[k,n]$,  $1\le p \le q<n$,
 $k\le n$, is embedded via $\gls{E}$ into a connected component of the Levi branched  crystal $\textsf{SSYT}_{J\sqcup \bar J}(\lambda, n)$
 such
 that
  $J\sqcup \bar J$ is a disconnected diagram of $[1,\dots,n,\bar n,\dots \bar 2]$ if $q<n$, and otherwise, $J\sqcup \bar J=[k,\overline{k+1}]$
or  $\bar{J}= J$ if $J = \{ n\}$.
  Consider the Levi branching of the  type $C_n$ crystal \gls{symptab} to $A_{q-p+1}, 1\le p\le q<n$, and $C_{n-k+1}$, $ k\le n$.
  The Levi type $A_{q-p+1}$ Dynkin diagram  is obtained via folding from the Levi subtype  $A_{q-p+1}\times A_{q-p+1}$ of $A_{2n-1}$ which is
  obtained
   by removing the nodes $ 1,\dots,p-1,q+1,\dots,n,\bar n,\dots,\overline{q+2}, \overline{ p+2},\dots,\bar 2 $ from the $A_{2n-1}$ Dynkin
   diagram.
The Levi type $C_{n-k+1}, k\le n$, is obtained via folding from the Levi subtype $A_{2n-2k+1}$ of $A_{2n-1}$ obtained by removing the nodes
$1,\dots,k-1, \overline{k},\dots,\bar 2$ from the $A_{2n-1}$ Dynkin diagram \cite{BSch17}.

 In \cite[Proposition 2.3 (ii)]{ba00a}, it is shown that given $b\in   \gls{symptab}$, the $C_n$ crystal length functions $\varepsilon_i^C$,
 $\varphi_i^C$, $1\le i\le n$, on $b$, and the $A_{2n-1}$ crystal length functions $\varepsilon_i^A, \varepsilon_{\overline{i+1}}^A$, $1\le
 i<n$, $\varepsilon_n^A$, and $\varphi_i^A$, $\varphi_{\overline{i+1}}^A$,$1\le i<n$,  $\varphi_n^A$, on $\gls{E}(b)$ are nicely related:
 $$\varepsilon_i^C (b)=\varepsilon_i^A (\gls{E}(b))=\varepsilon_{\overline{i+1}}^A (\gls{E}(b)),\;1\le i<n, \;\text{and}\;\varepsilon_n^C
 (b)= \frac{1}{2}\varepsilon_n^A (\gls{E}(b)),$$
 and similarly for $\varphi_i^C(b)$, $1\le i\le n$, where $ \varepsilon_i (b)=\max\{k\in\mathbb{Z}_{\ge 0}:e_i^k(b)\neq 0\}$ and $ \varphi_i
 (b)=\max\{k\in\mathbb{Z}_{\ge 0}:f_i^k(b)\neq 0\}$. This means that $b$ is the highest weight element of a connected component $U$
 of $\gls{branchedsymptab}$ if and only if, for all $i\in J$,
\[ \varepsilon_i^A (\gls{E}(b))=\varepsilon_{\overline{i+1}}^A (\gls{E}(b))=\varepsilon_i^C (b)=0, \,\text{for all $i\in J\setminus \{n\}$}\]
and
  \[\varepsilon_n^A (\gls{E}(b))=\varepsilon_n^C(b)=0,\;\text{if $n\in J$}.\]

Similarly, in the case  $b$ is the lowest weight element, by replacing appropriately, $\varepsilon_i^C$ with $\varphi_i^C$ and
$\varepsilon_i^A$, $\varepsilon_{\overline{i+1}}^A$ with $\varphi_i^A$, respectively $\varphi_{\overline{i+1}}^A$, and $\varepsilon_n^A$ with
$\varphi_n^A$.

 Henceforth,$$\varepsilon_i^C (b)=0,\,i\in J\Leftrightarrow \varepsilon_i^A (\gls{E}(b))=0,\;i\in J\sqcup\bar J$$and
 $$\varphi_i^C (b)=0,\,i\in J\Leftrightarrow \varphi_i^A (\gls{E}(b))=0,\;i\in J\sqcup\bar J.$$
 In other words, because our crystals are seminormal, $\gls{E}(b)$ is the highest weight element of the connected component $V$ of
 $\textsf{SSYT}_{J\sqcup \bar J}(\lambda, n)$ containing $\gls{E}(b)$ and $\gls{E}(U)$. It is
 therefore unique. A similar statement holds for the lowest weight element. The next proposition now easily follows.

 \begin{proposition} \label{highlow}  Let $J\subseteq [n]$ be a connected sub-diagram of the type $C_n$ Dynkin diagram. Let $U$ be a connected
 component of the Levi branched crystal $\gls{branchedsymptab}$ with highest and lowest weight elements $u^{\textsf{high}}$ and
 $u^{\textsf{low}}$ respectively. Then $\gls{E}(U)$ is contained in a connected component of the Levi branched crystal $\textsf{SSYT}_{J\sqcup
 \bar J}(\lambda, n)$ with
 highest and lowest weight elements $\gls{E}(u^{\textsf{high}})$ and $\gls{E}(u^{\textsf{low}})$ respectively.
 \end{proposition}

\begin{remark}\label{re:split} Given $ T\in \textsf{SSYT}(\mu, n,\bar
  n)$, with $\mu$ a partition with at most $ 2n$ parts,  $T$ may be decomposed into two disjoint semi-standard tableaux $T^+$ and $T^-$,
  $T= (T^+, T^-)$, where $T^+$ is the semi-standard
  tableau of shape $\mu+$ on the alphabet $[n]$ defined by the  entries of $T$ in $[n]$, that is, $T^+\in
  \textsf{SSYT}(\mu_+,n)$, called the positive part of $T$, and $T^-$ is the semi-standard  tableau  of skew shape
  $\mu/\mu_+$ on the alphabet $[\bar n, \bar 1]$ defined by the  entries of $T$ in $[\bar n,\bar 1]$, that is, $T^-\in
  \textsf{SSYT}(\mu/\mu_+,\bar n)$, called the negative part of $T$. Provided that we only apply  $f^A_i$, $e^A_i$ with $i\in
  J\sqcup  J'$ disconnected such that $J\subseteq [n-1]$ and $J'\subseteq [\bar n,\bar 2]$, respectively, this shape  decomposition is
  preserved.
  Those crystal operators preserve the shape decomposition above because, according to the type $A_{2n-1}$ signature rule, they only change
  positive (resp. negative) letters into positive (resp. negative) letters.
For $J\sqcup J'$ disconnected, $f^A_jf^A_{j'}=f^A_{j'}f^A_j$, with $j\in J$, $j'\in J'$.
  We then write,  for $\{j_1,\dots, j_r\}\subseteq J$ and $\{j'_1,\dots, j'_m\}\subseteq J'$,
  \begin{align}f^A_{j_r}\cdots f^A_{{j_1}}
  f^A_{{j'_m}}\cdots f^A_{{j'_1}}(T)= ( f^A_{j_r}\cdots f^A_{{j_1}}(T^+), f^A_{{j'_m}}\cdots
  f^A_{{j'_1}}(T^-)).\label{decompose1}\end{align}
\end{remark}

 \section{The cactus group and virtualization}\label{sec:cacti}
 Halacheva \cite{ha16,hakarywe20} has defined a more general version of the cactus group $\gls{Jn}$ originally defined by Henriques--Kamnitzer
 \cite{hk06} in terms of generators and relations.
 
If $I$ is the $A_{n-1}$ Dynkin diagram, $\theta_J$ acts on $J$ by reversing the connected interval of nodes $J$, whereas  in the $C_n$ type it
depends on whether $J$  contains the node with label  $n$  or not.
 \begin{definition}[\cite{ha16,hakarywe20} ]
 \label{geberalcacti}
 Let $\gls{g}$ be a finite-dimensional, semi-simple Lie algebra with Dynkin diagram $I$. The \textit{cactus group} $J_{\gls{g}}$ has
 generators $s_J$ where $J$ runs over the connected sub-diagrams of the Dynkin diagram $I$ of $\gls{g}$, and relations:
\begin{itemize}
\item[1$\gls{g}$.]	$s_J^2=1$ ,  for all $J\subseteq I$,
\medskip

      \item[2$\gls{g}$.]  $s_Js_{J'}=s_{J'}s_J$,    for all  $J,J'\subseteq I$ such that $J\sqcup J'$ is disconnected,

      \medskip

        \item[3$\gls{g}$.]  $s_Js_{J'}=s_{\theta_J (J')} s_J$, for all $J'\subseteq J\subseteq I$.
\end{itemize}
\end{definition}

\begin{remark}
\label{levirem}
Note that when $J'\subseteq J$, 3$\gls{g}$ says that $s_J$ commutes
with $s_{J'}$ by reversing $J'$ with respect to $J$. Recalling that $W$ is the Weyl group of $\gls{g}  $, we also have a
group epimorphism $J_{\gls{g}} \rightarrow W$ taking $s_J$ to $w_0^J$
{  \cite{hakarywe20}, \cite[Remark 10.0.1]{ha16}. The kernel is called
the \textit{pure cactus group} (see \cite{bcl22} for examples), the fundamental
group of the real locus of the Deligne–Mumford compactification $\overline{ M_{0,n+1}}(\mathbb{R})$ of the moduli space of rational
curves with $n + 1$ marked points \cite{dev99}.}
\end{remark}

\begin{lemma} \label{Acact} The cactus group $J_{\mathfrak{sl}(n,\mathbb{C})}=\gls{Jn} $ is the group with  generators $s_J$, where $J$ runs
over all connected sub-diagrams of $I=[n-1]$, the $A_{n-1}$ Dynkin diagram,  subject to the relations
\begin{enumerate}

\item[1A.]	$s_J^2=1, J\subseteq [n-1]$,
\medskip

\item[2A.]	
$s_Js_{J'}=s_{J'}s_J,  \hbox{ for all } J,J'\subseteq [n-1] \hbox{ such that } J\sqcup J' \hbox{ is disconnected}.$
\medskip
\item[3A.] $s_{[p,q]}s_{[k,l]}=s_{[p+q-l,p+q-k]} s_{[p,q]}$ for $ [k,l]\subset [p,q]\subseteq [n-1]$.

\end{enumerate}
\end{lemma}

\begin{remark}
The first and third relations ensure that the $n-1$ elements of the form
\begin{equation}\label{alterA} s_{[1,k]},\;1\le k\le  n-1,\end{equation}
generate  \gls{Jn},  since any $s_{[i,j]}$ may be written as
\begin{equation}\label{eq:cactusA}
s_{[i,j]} = s_{[1,j]} s_{[1,j-i+1]} s_{[1,j]}.
\end{equation}
\noindent The elements $ s_{[i,n-1]}=s_{[1,n-1]}s_{[1,n-i]}s_{[1,n-1]}, \;1 \leq i \leq n-1$, also form a set of generators.
\end{remark}

\begin{lemma}\label{symplecticact} The cactus group $\gls{Jsp}$ is  the group with  generators $s_J$, where $J$ runs over all connected
sub-diagrams of the $C_n$ Dynkin diagram $I=[n]$,  subject to the relations
\begin{enumerate}

\item[1C.]	$s_J^2=1, J\subseteq [n]$,

\medskip
\item[2C.]	$s_Js_{J'}=s_{J'}s_J,    \hbox{ for all } J,J'\subseteq [n] \hbox{ such that } J\sqcup J' \hbox{ is disconnected}, $

\medskip
\item[3C.]	\begin{enumerate}
\item[(i)] $s_{[p,n]}s_{[k,l]}=s_{[k,l]}s_{[p,n]}$,  $[k,l]\subseteq [p,n]\subseteq [n]$,
\medskip

\item [(ii)]$s_{[p,q]}s_{[k,l]}=s_{[p+q-l,p+q-k]} s_{[p,q]}$, $[k,l]\subseteq [p,q]\subseteq [n-1]$.
\end{enumerate}

\end{enumerate}
\end{lemma}

\begin{proof}
Relations $1 \gls{g}$ and $2\gls{g}$ translate directly to 1A. and 2A. Consider two nested intervals $[k,l]\subset [p,q]$.
 If $[p,q] \subset [n-1]$, we are in type $A_{q-p}$, hence 3C.(ii) holds, which is just relation 3A.
  If $q = n$, then we are in type $C_{n-p+1}$. The Weyl group $W^{[p,n]}$  is the restriction of the hyperoctahedral  group $\gls{bn}$ to the
  generators $r_p,\dots, r_n$, (as a group of signed permutations,  it is the restriction to the set
\begin{align*}
 [ \pm p,n]=\{p<\cdots< n<\bar n<\cdots<\bar p\}),
 \end{align*}
  and $w_0^J(\alpha_j)=-\alpha_j$ for $j\in J$. Therefore $\theta_{[p,n]}(d) = d$ for $d \in [k,l]$ and Relation 3C. (i) follows directly from
  $3\gls{g}$.
\end{proof}

\begin{remark}
Note that the elements $s_J$, $J\subseteq [n-1]$, subject to the relations above, generate the cactus group  \gls{Jn}. As in
\eqref{eq:cactusA},  the following are alternative $2n-1$ generators of $\gls{Jsp}$:
\begin{eqnarray}
\label{alterC1} s_{[1,j]},&\;\quad \quad 1\le j\le n-1,\\
s_{[j,n]},& \; 1\le j\le n.\label{alterC2}
\end{eqnarray}
\end{remark}

\begin{remark}
\label{widetilde}
 We may observe that  \gls{Jn} is a subgroup of $\gls{Jsp}$ defined by the subset of generators $s_J$, $J\subseteq [n-1]$, indexed by
 connected sub-diagrams of the $A_{n-1}$ connected sub-diagram $[n-1]$ of the $C_n$ Dynkin diagram $I=[n]$,  subject to the relations above
 $1C.$, $2C.$ and $3C.$, $(ii)$.
\end{remark}

\begin{proposition}
If $\gls{g}$ is a finite-dimensional semi-simple Lie algebra, and $\mathfrak{l} \subset \gls{g}$ is a Levi sub-algebra, then
$J_{\mathfrak{l}}$ is a subgroup of $J_{\gls{g}}$.
\end{proposition}

\begin{proof}
Let $I$ be the  Dynkin diagram corresponding to $\gls{g}$ and $J\subset I$ the sub-diagram corresponding to the Levi sub-algebra
$\mathfrak{l}$. Any connected sub-diagram $K$ of $J$ is also a connected sub-diagram of $I$, hence one can define a map on generators by
$s^{J}_{K}\mapsto  s^{I}_{K}$. Here generators of $J_{\gls{g}}$ are denoted by $s^{I}_{K}$, and generators of $J_{\mathfrak{l}}$ by
$s^{J}_{K}$. Remark \ref{levirem} implies that this map is a morphism of groups. The map is clearly injective because the generators of
$J_{\gls{g}}$ are all distinct.
\end{proof}

\subsection{Embedding of \texorpdfstring{\gls{Jsp}}{Jsp} into \texorpdfstring{$J_{2n}$}{J2n}} We have observed that  \gls{Jn} is a subgroup of \gls{Jsp}.
We now show that there is a group embedding  of \gls{Jsp} into  $J_{2n}$  by folding the $A_{2n-1}$ Dynkin diagram through the middle node
$n$:

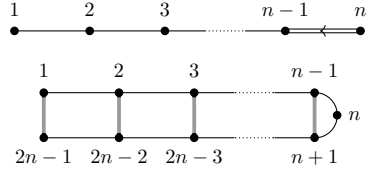
$$\dynkin[edge length=1cm,labels*={1,2,3,n-1,n}] C{***.**}$$
$$\dynkin[edge length=1cm,fold,labels={1,2,3,n-1,n, n+1, 2n-3, 2n-2, 2n-1}]A{***.***.***}$$

Why should such an embedding exist? Let us consider the following elements of $J_{2n}$:
\[ s'_{[p,q]}:=s_{[p,q]}s_{[2n-q,2n-p]}=s_{[2n-q,2n-p]}s_{[p,q]}, \;\text{for all  $[p,q]\subseteq [n-1]$}.\]
In Lemma \ref{injectionbyfolding1} we show that these elements together with the generators $s_{[p,2n-p]}$ for $p \leq n$ generate a subgroup
of $J_{2n}$ isomorphic to $\gls{Jsp}$. Notice the similarity between this and the construction of $\mathfrak{sp}(2n,\mathbb{C})$ as a
sub-algebra of $\mathfrak{sl}(2n,\mathbb{C})$ by folding \cite[Chapter 8, pp. 89 -- 102]{Kac1983}. Moreover, Lemma \ref{symmetrieslemma} below provides not
only concrete combinatorial motivation for Lemma \ref{injectionbyfolding1}, but will also be the main ingredient in its proof.

\begin{lemma}\label{symmetrieslemma}
 The following relations hold in $J_{2n}$:
\begin{align} \label{cactisymm01}
{{s'}^2_{[p,q]}} &=1 \hbox{,     }1 \leq p \leq  q < n,  \\ \label{cactisymm002}
 s^{2}_{[p,2n-p]} &= 1 \hbox{,     }1 \leq p < n,\\  \label{cactisymm02}
 s'_{[p,q]} s'_{[k,l]} &= s'_{[k,l]}s'_{[p,q]} \hbox{,     }\text{ $[p,q]\sqcup[k,l]\subseteq [n-1]$ disconnected}, \\ \label{cactisymm1}
s_{[p,2n-p]}s_{[k,2n-k]} &= s_{[k,2n-k]} s_{[p, 2n-p]}  \hbox{,     }1 \leq p < k < n, \\ \label{cactisymm2}
s_{[p,2n-p]}s'_{[k,l]}&=  s'_{[k,l]} s_{[p,2n-p]} \hbox{,    }  1 \leq p \le  k \le   l < n, \\ \label{cactisymm3}
s'_{[p,q]}s'_{[k,l]} &= s'_{[p+q-l,p+q-k]}s'_{[p,q]} \hbox{,       }   1 \leq p \le k \le l \le  q < n.
\end{align}
There are no more relations among the elements $s'_{[p,q]}$ and $s_{[k,2n-k]}$, for all $[p,q]\subseteq [n-1]$ and $[k,n]\subseteq
[n]$.
\end{lemma}

\begin{proof}
We have  the relations \eqref{cactisymm01} and \eqref{cactisymm002},
\[  {s'_{[p,q]}}^2  = (s_{[p,q]}s_{[2n-q,2n-p]} )^{2}  = s^{2}_{[p,q]}s^{2}_{[2n-q,2n-p]} \overset{ 1A.}{=}1.\]
\[  s^{2}_{[p,2n-p]} \overset{ 1A.}{=} 1.\]
For $ 1 \leq p \leq q <n$ and $ 1 \leq k \leq  l < n$ such that $[p,q] \sqcup [k,l] $ is disconnected,
the sub-diagrams $[p,q] \sqcup [2n-q, 2n-p]$, and $ [k,l] \sqcup  [2n-l, 2n-k]$ of  $[2n-1]$ are disconnected, hence
\[
  s'_{[p,q]} s'_{[k,l]} \overset{ 2A.}{=} s'_{[k,l]}s'_{[p,q]}.  \] This establishes  \eqref{cactisymm02}.
Additionally, if $q = n$, the sub-diagram $[k,l] \sqcup [p, 2n-p] \sqcup [2n-l, 2n-k]$ in $[2n-1]$ is disconnected, hence
\[ s_{[p,2n-p]} s'_{[k,l]} \overset{ 2A.}{=} s'_{[k,l]}s_{[p,2n-p]}.  \]
{ Moreover,}
$s_{[p,2n-p]}s_{[k,2n-k]} \overset{ 3A.}{=} s_{[2n-(2n-k),2n-k]} s_{[p, 2n-p]} = s_{[k,2n-k]}s_{[p,2n-p]},$
 for $1 \leq p < k <n $, hence relation (\ref{cactisymm1}) holds. Now, for $1 \leq p < k < l <n $ we have:
 \begin{align*}
 s_{[p,2n-p]}s_{[k,l]}s_{[2n-l,2n-k]} \overset{ 3A.}{=}& s_{[2n-l,2n-k]} s_{[p,2n-p]} s_{[2n-l,2n-k]} \\
                                              \overset{ 3A.}{=}& s_{[2n-l,2n-k]} s_{[2n-(2n-k),2n-(2n-l)]}s_{[p,2n-p]} \\
                                             =& s_{[2n-l,2n-k]} s_{[k,l]} s_{[p,2n-p]}
 \end{align*}

 \noindent which establishes relation (\ref{cactisymm2}). Finally, for $1 \leq p < k < l < q < n$ the following holds:
 \begin{align*}
 s'_{[p,q]}s'_{[k,l]}  = & s_{[p,q]} s_{[2n-q,2n-p]}s_{[k,l]} s_{[2n-l,2n-k]} \\
                    \overset{ 2A.}{=}& s_{[p,q]} s_{[k,l]} s_{[2n-q,2n-p]} s_{[2n-l,2n-k]} \\
                     \overset{ 3A.}{=}& s_{[p+q-l,p+q-k]} s_{[p,q]} s_{[2n-(p+q-k), 2n-(p+q - l)]} s_{[2n-q, 2n-p]} \\
                     \overset{ 2A.}{=}& s_{[p+q-l,p+q-k]} s_{[2n-(p+q-k), 2n-(p+q - l)]} s_{[p,q]} s_{[2n-q, 2n-p]} \\
                     =& s'_{[p+q-l,p+q-k]}s'_{[p,q]}.
 \end{align*}
 This establishes relation (\ref{cactisymm3}).
Any relation $R'=1$ with the elements $s'_{[p,q]}=s_{[p,q]}s_{[2n-q,2n-p]}$ and $s_{[k,2n-k]}$, for some $[p,q]\subseteq [n-1]$ and
 $[k,n]\subseteq [n]$, translates to a relation $R=1$ involving  generators of $J_{2n}$, of the form $s_{[p,q]}$,  $s_{[2n-q,2n-p]}$ in pairs,
 and $s_{[k,2n-k]}$, for some $[p,q]\subseteq [n-1]$ and $[k,n]\subseteq [n]$, which satisfy the same kind of relations as $s'_{[p,q]}$ and
 $s_{[k,2n-k]}$. Therefore from $R=1$ we don't get new relations $R'=1$.
\end{proof}

 \begin{definition} \label{foldAcact} The \textit{virtual symplectic cactus group} \gls{vJ2n} is the group with  generators $\tilde s_J$,
 where $J$ runs over all  sub-diagrams of $I=[2n-1]$, the $A_{2n-1}$ Dynkin diagram, of the form $J = [p,2n-p]$ for all $[p,n] \subseteq [n]$,
 or $J = [p,q] \sqcup [2n-q,2n-p]$ for all $[p,q] \subseteq [n-1]$ subject to the relations
\begin{enumerate}
\item[1$\tilde A$.]	$\tilde s_J^2=1, J\subseteq [2n-1]$,
\medskip
\item[2$\tilde A$.]	
$\tilde s_J \tilde s_{J'}= \tilde s_{J'} \tilde s_J,$ such that $J \sqcup J'$ is disconnected with respect to all $[p,q] \subseteq [n]$,
\medskip
\item[3$\tilde A$.]\begin{enumerate}
\item[(i)]
$\tilde s_{[p,2n-p]} \tilde s_{[q,l]\sqcup [2n-l,2n-q]}= \tilde s_{[q,l]\sqcup [2n-l,2n-q]} \tilde s_{[p,2n-p]}, [q,l]\subseteq
[p,n]\subseteq
[n]$,
\medskip
 \item[(ii)] for $[k,l]\subseteq [p,q]\subseteq [n-1]$,
 \begin{gather*}
 \tilde s_{[p,q]\sqcup [2n-q,2n-p]} \tilde s_{[k,l]\sqcup [2n-l,2n-k]}  = \\
 \tilde s_{[q+p-l,q+p-k]\sqcup [2n-p+2n-q-(2n-k),2n-p+2n-q-(2n-l)]} \tilde s_{[p,q]\sqcup [2n-q,2n-p]}= \\
  \tilde s_{[q+p-l,q+p-k]\sqcup [2n-(p+q)+k,2n-(p+q)+l]} \tilde s_{[p,q]\sqcup [2n-q,2n-p]}.
 \end{gather*}
\end{enumerate}
	
\end{enumerate}
\end{definition}

The following are $2n-1$ alternative generators of \gls{vJ2n}:
\begin{eqnarray}
\label{virtualalterC1} \tilde s_{[1,j]\sqcup[2n-j,2n-1]},\; 1\le j\le n-1,\\
\tilde s_{[j,2n-j]}, \; 1\le j\le n.\label{virtualalterC2}
\end{eqnarray}

\begin{proposition}
\label{virtualcactus}
There is an isomorphism $\gls{vJ2n}\simeq \gls{Jsp}$.
\end{proposition}

\begin{proof}
Clearly $\gls{vJ2n}$ and $ \gls{Jsp}$ satisfy the same relations corresponding to all  connected sub-diagrams $[p,q]\subseteq [n]$.
Furthermore, the maps
\begin{align*}\gls{Jsp}  &\rightarrow \gls{vJ2n} \\
s_{[p,q]} &\mapsto \tilde s_{[p,q]\sqcup [2n-q,2n-p]} \\
 s_{[p,n]} &\mapsto \tilde s_{[p,2n-p]}, \hbox{ and }\\
 \hbox{  } & \\
\gls{vJ2n} &\rightarrow \gls{Jsp} \\
  \tilde s_{[p,q]\sqcup [2n-q,2n-p]} &\mapsto s_{[p,q]} \\
  \tilde s_{[p,2n-p]} & \mapsto s_{[p,n]},
  \end{align*}

\noindent
are  epimorphisms inverse to each other. This follows directly from the definitions of $\gls{vJ2n}$ and $ \gls{Jsp}$ (Definition
\ref{foldAcact} and Lemma \ref{symplecticact} respectively). Therefore,  $\gls{Jsp}\simeq \gls{vJ2n}$.
\end{proof}

\begin{lemma}
\label{injectionbyfolding1}
The following assignment defines a group injection from $\gls{Jsp}$ to
$J_{2n}$:
$$\begin{array} {cccccc}
\Gamma:& \gls{Jsp}&\hookrightarrow&J_{2n}\\
&s_{[p,q]}&\mapsto& s'_{[p,q]},&1\le p\le q<n,\\
&s_{[p,n]}&\mapsto& s_{[p,2n-p]},&1\le p\le n.
\end{array}
$$
\end{lemma}

\begin{proof} The map induced by $\Gamma$ is indeed a group homomorphism. The relations $1C. - 3C.$ from Lemma \ref{symplecticact} follow from
Lemma \ref{symmetrieslemma}.

To show that it is injective, one needs to show that its left inverse defined by the assignment
$$\begin{array} {cccccc}
\Gamma^{-1}_{left}:& \op{im}(\Gamma) \subset
J_{2n}&\hookrightarrow&\gls{Jsp}\\
&s'_{[p,q]}&\mapsto& s_{[p,q]},&1\le p\le q<n,\\
&s_{[p,2n-p]}&\mapsto& s_{[p,n]},&1\le p\le n.
\end{array}
$$
is also a group morphism. This however follows from Lemma \ref{symmetrieslemma} and the previous calculations: the generators of
$\op{im}(\Gamma)$ satisfy the relations from Lemma \ref{symplecticact}, and there are no more relations between them (all possible cases have
been already covered above).
\end{proof}

\begin{proposition}
\label{injectionbyfolding2}
The group $\gls{vJ2n}$ is isomorphic to a subgroup of $J_{2n}$.
\end{proposition}

\begin{proof} After composing the maps from Proposition \ref{virtualcactus} and Lemma
\ref{injectionbyfolding1}, one obtains
the group injection,
\begin{align*}
\gls{vJ2n} &\hookrightarrow J_{2n}& \\
  \tilde s_{[p,q]\sqcup [2n-q,2n-p]} &\mapsto s'_{[p,q]},  1\le p\le q<n,\\
  \tilde s_{[p,2n-p]} & \mapsto s_{[p,2n-p]},   1\le p\le n.
\end{align*}
\end{proof}

\section{Full Sch\"utzenberger--Lusztig involutions and algorithms}\label{sec:fullschutz}

\subsection{Full Sch\"utzenberger--Lusztig involution } Let $\gls{cncrystal}$ be the  normal $\gls{g}$-crystal with
highest weight $\lambda$. Let $u_{\lambda}$ and $u^{low}_{\lambda}$ be the highest, respectively lowest, weight elements of $\gls{cncrystal}$.
The  \emph{Sch\"utzenberger--Lusztig involution}  $\gls{xi}:\gls{cncrystal}\rightarrow \gls{cncrystal}$  is the unique map of sets
(hence  set involution)
such
that, for all $b\in \gls{cncrystal}$, and $i\in I$,
\begin{itemize}
\item $e_i\gls{xi}(b)=\gls{xi} f_{\theta(i)}(b)$

\item  $f_i\gls{xi}(b)=\gls{xi} e_{\theta(i)}(b)$

\item $\textsf{wt}(\gls{xi}(b))=w_0\textsf{wt(b)}$
\end{itemize}
where $w_0$ is the long element of the Weyl group $W$. (For the existence and uniqueness of $\xi$,  $\xi^2$ is a map of crystals and
hence $\xi^2=1$, see \cite{hk06, BSch17}.) The  involution
$\gls{xi}$ acts by $w_0$ on the weights and
interchanges the action of $e_i$ and $f_{\theta(i)}$. It is an automorphism of the underlying unlabeled, non-oriented and non-weighted graph.
For $A_{n-1}$, $\gls{xi}$ acts
by reversing the weight and interchanges the action of
$e_i$ and $f_{n-i}$; for $C_n$,  $\gls{xi}$ acts by changing the sign of the weight and interchanges the action of $e_i$ and $f_{i}$.

If $\gls{B}$ is a normal $\gls{g}$-crystal, $\gls{B}$ is the disjoint union of  connected components, each of which is a crystal isomorphic to
$\gls{cncrystal}$ for some dominant integral weight $\lambda$. We define  $\gls{xib}$  on $\gls{B}$ by applying $\gls{xi}$ to each one of its
connected components. Each element of $\gls{cncrystal}$ is generated by $u_\lambda$ (resp.  $u_\lambda^{\textsf{low}}$ ) by applying $f_i$'s
(resp. $e_i$'s). Hence the same sequence of $f_i$'s (resp. $e_i$'s) applies to the highest weight (resp. lowest weight) of any connected
component of $\gls{B}$ isomorphic to $\gls{cncrystal}$.

The elements $u_\lambda$ and $u_\lambda^{\textsf{low}}$ are the unique elements of $\gls{cncrystal}$ of  weight  $\lambda$, respectively
$w_0\lambda$.  Hence,  since $\textsf{wt}(\gls{xi}( u_\lambda))=w_0\lambda$, and $\textsf{wt}(\gls{xi}(
u^{\textsf{low}}_\lambda))=\lambda$, $\xi$ interchanges  highest and lowest weight elements of $B(\lambda)$, and so
$u_\lambda^{\textsf{low}}=\gls{xi}(u_\lambda)$, $\gls{xi}(u_\lambda^{\textsf{low}})=u_\lambda$. This implies that,
$u_\lambda=e_{j_r}\cdots e_{j_1} (u_\lambda^{\textsf{low}})$, for some sequence $j_1,\dots,j_r\in I$, and
\[u_\lambda^{\textsf{low}}=\gls{xi}(u_\lambda)=\gls{xi}(e_{j_r}\cdots e_{j_1} (u_\lambda^{\textsf{low}}))=f_{\theta( j_r)}\cdots f_{\theta
(j_1)}(\gls{xi} ( u_\lambda^{\textsf{low}}))=f_{\theta(j_r)}\cdots f_{\theta(j_1)}( u_\lambda).\]

\begin{corollary}
\label{theta}
Let $b\in \gls{cncrystal}$ and $b=f_{ j_r}\cdots f_{ j_1}( u_\lambda)$, { for $j_r,\dots,j_1\in I$}. Then
$$\gls{xi}(b)= e_{\theta(j_r)}\cdots e_{\theta(j_1)}(u_\lambda^{\textsf{low}}),\quad \textsf{wt}(\gls{xi}(b))=w_0\textsf{wt}(b)$$

In particular,
\begin{itemize}
\item in type $A_{n-1}$, $\gls{xi}(b)= e_{n- j_r}\cdots e_{n- j_1}( u_\lambda^{\textsf{low}}),$ and
    $\textsf{wt}(\gls{xi}(b))={rev}\,\textsf{wt}(b)$, where $rev$ is the reverse permutation (long element) of $\mathfrak{S}_n$,
\item in type $C_n$, $\gls{xi}(b)= e_{j_r}\cdots e_{ j_1}(u_\lambda^{\textsf{low}}),$ and $\textsf{wt}(\gls{xi}(b))=-\textsf{wt}(b).$
\end{itemize}
\end{corollary}

 \subsection{The full \texorpdfstring{$\mathfrak{sl}(n,\mathbb{C})$}{sln}  reversal}\label{Areversal} For $\gls{g}=\mathfrak{sl}(n,\mathbb{C})$, $\gls{xi}$
 coincides with the
 Sch\"utzenberger involution \cite{lenart, bz96} also known as
 \textsf{evacuation}
 (\textsf{evac} for short) on \gls{ssyt} \cite{Fu, Sta}, and as \gls{rev} on the set $\textsf{SSYT}(\lambda/\mu,n)$ of skew semi-standard tableaux of
 shape $\lambda/\mu$ in the alphabet $[n]$  \cite{ bss}.

Let $T\in \gls{B}=\textsf{SSYT}(\lambda/\mu,n)$ and let $\textsf{B}(T)$ be the connected component of the crystal
$\textsf{SSYT}(\lambda/\mu,n)$  containing $T$. Then $\textsf{B}(T)\simeq \textsf{B}(\nu)$ for some partition $\nu$ and $\textsf{rect} (T)\in
\textsf{B}(\nu)$. Thereby, $\gls{xi}(T)$ is the unique tableau in $\textsf{B}(T)$ such that
\[\textsf{rect}\,\gls{xi}(T)=\textsf{evacuation}(\textsf{rect} (T)),\]
\begin{align}\label{fullformula} \gls{xi}(T)=\textsf{arect}(\textsf{evacuation}(\textsf{rect} (T))),\end{align}
where $\textsf{arectification}$ ($\textsf{arect}$ for short) denotes the inverse process of \textsf{rectification} \cite{bss,
AzenhasConflittiMamede2019}. More precisely, the $\textsf{rectification}$ ($\textsf{rect}$ for short) procedure is recorded by assigning a
standard tableau $S$ to the inner shape $\mu$ of $T$ to form the tableau pair $(S,T)$. The entries of $S$ govern the \emph{jeu de taquin} on
$T$, as we slide out all letters in the filling of $S$, from largest to smallest, to get a new tableau pair $(\textsf{rect}(T), S')$, where
$S'$ is the skew standard tableau consisting of the slid letters from $S$. The anti-rectification procedure, \textsf{arect}, is defined by the
reverse \emph{jeu de taquin} to $\textsf{evacuation}(\textsf{rect} (T)$ and is governed by the slid letters in $S'$ in the tableau pair
$(\textsf{evacuation}(\textsf{rect} (T)), S')$ from smallest to largest. Eventually one obtains
  the tableau pair $(S, \gls{rev}(T))$ where
 \begin{align}\gls{rev}(T):=\textsf{arect}(\textsf{evacuation}(\textsf{rect}(T))). \label{reversalgln}
 \end{align}

Next we will discuss $\gls{g}=\mathfrak{sp}(2n,\mathbb{C})$.

 \subsection{ Lecouvey--Sheats symplectic jeu de taquin and symplectic Knuth equivalence }
 \label{sjdt}

If $T$ is a KN tableau, we consider its word $w(T)$ $ \in \gls{Cstar}$ obtained by reading the columns of $T$ in Chinese/Japanese order, from rightmost
to leftmost, each column read from top to bottom.  Recall Remark \ref{lecouvey2.2.2}.

\subsubsection{Lecouvey--Sheats symplectic jeu de taquin } In this section, which we include for the comfort of the reader, we recall material from \cite{she99,lec02} which is relevant for our purposes. A punctured \textit{skew} KN tableau is a skew KN tableau one of whose boxes, called its \textit{puncture}, instead of being blank or containing a letter in $\gls{Corderedletters}$, is distinguished by being filled in with an asterisk $*$ instead. The \textit{splitting} $\emph{spl}(T)$ of a KN tableau $T$ contains two punctures, in consecutive columns and in the same row, corresponding to the ``splitting'' of the puncture in $T$.

 Let $T$ be a punctured KN tableau with two columns $C_1$ and $C_2$ and split form $spl(T)=l C_1 rC_1 lC_2rC_2$, and let $C_1$ have the
 puncture $\ast$. Let $\alpha$ be the entry under the puncture of $rC_1$ and $\beta$ the entry to the right of the puncture of $rC_1$,
$$spl(T)= l C_1 rC_1 l C_2rC_2=\YT{0.22in}{}{
	{{\dots},{\dots}, {\dots}, {\dots}},
	{{\ast},{\ast},{\beta},{\dots}},
	{{\dots},{\alpha},{\dots},{\dots}},
	{{\dots},{\dots}}},$$
where $\alpha$ or $\beta$ may not necessarily exist. Following the wording in \cite{sa21a}, the elementary steps of the symplectic jeu de
taquin, or SJDT for short, are the following:

{\textbf{A.}} If $\alpha\leq \beta$ or $\beta$ does not exist,  then the puncture of $T$ will change its position with the cell beneath it.
This is a vertical slide.

{\textbf{B.}}	If the slide is not vertical, then it is horizontal. We then have $\alpha> \beta$ or that $\alpha$ does not exist. Let $C_1'$
and $C_2'$ be the columns obtained after the slide. We have two subcases, depending on the sign of $\beta$:

{\boldmath{1.}} If $\beta$ is barred, we are moving a barred letter, $\beta$, from $l C_2$ to the punctured box of $rC_1$, and the puncture
will occupy $\beta$'s place in $l C_2$. Note that $l C_2$ has the same barred part as $C_2$ and that $rC_1$ has the same barred part as
$\Phi(C_1)$. Looking at $T$, we will have a horizontal slide of the puncture, getting $C_2'=C_2\setminus \{\beta\}\sqcup\{\ast\}$ and
$C_1'=\Phi^{-1}(\Phi(C_1)\setminus{\{\ast\}}\sqcup \{\beta\})$. In a sense, $\beta$ went from $C_2$ to $\Phi(C_1)$.

{\boldmath{2.}} If $\beta$ is unbarred, the procedure is similar, but this time $\beta$ will go from $\Phi(C_2)$ to $C_1$; hence
$C_1'= C_1\setminus{\{\ast\}}\sqcup \{\beta\}$ and $ C_2'=\Phi^{-1}(\Phi(C_2)\setminus \{\beta\}\sqcup{\{\ast\}})$. However, in this case it may happen
that $C_1'$ is no longer admissible. In this situation, if $i$ is the lowest entry such that $i, \bar i$ appear in $C_1'$ and $N(i) > i$, we
erase both $i$ and $\overline{i}$ from the column, remove a cell from the bottom and from the top of the column and place all the remaining cells in order.

After applying elementary SJDT slides successively, the puncture will eventually reach a cell such that $\alpha$ and $\beta$ do not exist. In
this case we redefine the shape to not include this cell and the \textit{jeu de taquin} ends. The SJDT when applied to semi-standard tableaux
in the alphabet $[n]$ reduces to the ordinary \emph{jeu de taquin}.

The SJDT is reversible, meaning that we can move $\ast$, the empty cell outside of $\mu$, to the inner shape $\nu$ of a skew tableau $T$ of
shape $\mu / \nu$, simultaneously increasing both the inner and outer shapes of $T$ by one cell. The slides work similarly to the previous
case: the vertical slide means that an empty cell is going up, and a horizontal slide means that an entry goes from $\Phi(C_1)$ to $C_2$ or
from $C_1$ to $\Phi(C_2)$, depending on whether the slid entry is barred or not, respectively.  For an illustration of SJDT, we refer the reader to the first part of Example \ref{ex:sjdt}.

\subsubsection{Symplectic Knuth equivalence }

In this section we gather the necessary tools from \cite{{LLT95}, lec02}. For $w \in \gls{Cstar}$, let $P(w)$ be the Kashiwara--Nakashima
tableau obtained by performing the Baker--Lecouvey insertion algorithm on $w$. We do not need the algorithm in this paper, but refer the
reader to \cite{ba00b,lec02}  for the original descriptions. A detailed account can also be found in \cite{sa21b}. Given $w_1, w_2 \in
\gls{Cstar}$, the relation $w_1 \sim w_2\Leftrightarrow P(w_1)=P(w_2)$ defines an equivalence relation on $\gls{Cstar}$ known as
\emph{symplectic plactic equivalence}. It is the analogous relation defined by Knuth relations in the alphabet $[n]$ \cite{Fu}. The
\emph{symplectic plactic monoid} is the quotient $\gls{Cstar}/\sim$. Each \emph{symplectic plactic class} is uniquely identified with a KN
tableau.

The plactic monoid $\gls{Cstar}/\sim$ can also be described as the quotient of $ \gls{Cstar}$ by the following \emph{symplectic plactic
relations} (we use the notation from \cite{lec02}):

\begin{itemize}
\item[\textbf{R1}] \label{R1}
$yzx \sim yxz \hbox{ for } x \leq y < z\hbox{ with } z \neq \bar x \hbox{ and }
xzy \sim zxy \hbox{ for } x < y \leq z \hbox{ with } z \neq \bar x
$;
\medskip
\item[\textbf{R2}]\label{R2}
$
y \overline{x-1}(x-1) \sim y x \overline{x}  \hbox{ and }
x \overline{x} y \sim \overline{x-1} (x-1)y \hbox{ for }
1 < x \leq n \hbox{ and } x \leq y \leq \bar x
$;
\medskip
\item[$\gls{r3}$] \label{R3} (Symplectic contraction/dilation relation)
$w\sim w\setminus\{z,\overline{z}\}$, where $w\in \gls{Cstar}$ and $z\in [n]$ are such that $w$ is a non-admissible column, $z$ is the
lowest non-barred letter in $w$ such that $N(z) = z+1$ and any proper factor of $w$ is an admissible column.
\end{itemize}

For example, the words $23\overline{2}\overline{3}1$ and $\overline{1}113\overline{3}$ are symplectic Knuth related:
$\overline{1}113\overline{3} \stackrel{\mathbf{R1}}{\sim}\overline{1}131\overline{3}\stackrel{\mathbf{R1}}{\sim}\overline{1}13\overline{3}1
\stackrel{\mathbf{R2}}{\sim}2\overline{2}3\overline{3}1\stackrel{\mathbf{R1}}{\sim}23\overline{2}\overline{3}1.$ Note that to apply the plactic relation
\gls{r3} to a non-admissible column word $w$, we need only check that all proper prefixes of $w$ are admissible, as opposed to all proper
factors \cite{sa21a}. For example, $ 234\bar 4\bar 3\overset {\gls{r3}}\sim 23\bar 3,$ and $1234\bar 4\bar 3\overset {\gls{r3}}\sim 12 3\bar
3\overset
{\gls{r3}}\sim 12.$

When Knuth relations are applied to factors of a word, the weight is preserved while the length may not be. Knuth relations can be seen as
\textit{jeu de taquin} moves on words or diagonally shaped tableaux, and each symplectic \textit{jeu de taquin} slide preserves the Knuth
class of the reading word of a tableau \cite[Theorem 6.3.8]{lec02}.

\subsection {Full symplectic reversal}
\subsubsection{Symplectic evacuation algorithm}
\label{santosevac}
In \cite{sa21a}, Santos introduced a symplectic evacuation algorithm on tableaux in $\gls{symptab}$ denoted by $\textsf{evac}^{C_n}$ which he
proved coincides with the full Lusztig--Sch\"utzenberger involution on a given $U_{q}(\mathfrak{sp}(2n,\mathbb{C}))$-crystal $\gls{cncrystal}$
associated to a representation of highest weight $\lambda$. Santos' evacuation algorithm mimics the Sch\"utzenberger \textsf{evacuation}  on
$\gls{ssyt}$. It replaces the action of the long element of $\mathfrak{S}_n$ with that of the long element of $\gls{bn}$ and performs
symplectic rectification or insertion  using  Lecouvey--Sheats symplectic jeu de taquin  \cite{she99,lec02,lec07}, or Baker--Lecouvey
insertion \cite{ba00b,lec02,lec07}, respectively.  We refer the reader to \cite[Section 5]{sa21a} for detailed examples of the algorithm.

\subsubsection{Full symplectic reversal on KN skew tableaux} \label{susub:fullskew} The set $\textsf{KN}(\lambda/\mu,m)$   is  a normal $C_m$
crystal whose connected components are isomorphic to $\textsf{KN}(\nu,m)$ for some partition $\nu$ whose number of boxes $|\nu|$ is less
than or equal to $|\lambda|- |\mu|$. Let  $n=m+j-1$, where $1\le j-1<n$ is the number of parts of $\mu$ and $J=[j,n]$.

Let $\textsf{B}(\lambda,\mu)$ be the subset of  
$\gls{symptab}$ consisting of the tableaux in $\gls{symptab}$ with entries exclusively in $1<\cdots
<j<j+1<\cdots <j-1+m <\overline{j-1+m}<\cdots<\overline{j}$ outside of the shape $\mu \subset \lambda$ and whose sub-tableau on the alphabet  $\{1,\dots,j-1\}$ is the fixed Yamanouchi
tableau of shape $\mu$. By \cite[Lemma 6.1.3]{lec02}, $\textsf{B}(\lambda,\mu)$ is a normal $C_{m}$ crystal; in particular it is stable under the action of $f_{i+j-1},e_{i+j-1}$, $i=1,\dots,
m$, so it is a sub-crystal of $\gls{branchedsymptab}$. In particular note that each one of its connected components is contained in a connected component of $\gls{branchedsymptab}$. Shifting the entries of the  KN skew
tableaux in $\textsf{KN}(\lambda/\mu,m)$ by $j-1$, we may identify $\textsf{KN}(\lambda/\mu,m)$ with 
$\textsf{B}(\lambda,\mu)\subset \gls{branchedsymptab}$.
That is, the crystal operators, $f_i, e_i$, $i=1,\dots, m$ do not
change the skew-shape of a KN tableau on the alphabet $\mathcal{C}_m$, and $\textsf{KN}(\lambda/\mu,m)$ decomposes into connected components
that can be identified with the connected components of $\textsf{B}(\mu,\lambda)$.

In both type $A_{n-1}$ and type $C_n$, Kashiwara operators $e_i$ and $f_i$ commute with SJDT slides. Let $T\in
\gls{B}=\textsf{KN}(\lambda/\mu,n)$. An inner corner in $T$ is a box of $\mu$ such that the boxes below and to the right are not in $\mu$; an
outer corner in $T$ is a box of $\lambda$ such that the boxes below and to  the right are not in $\lambda$.  Let $c$ be a fixed inner/outer
corner of $T$. A \emph{SJDT slide} or a \emph{complete  SJDT slide} to the inner corner $c$  means a slide of the box $c$ from an inner
corner to an outer corner, or vice-versa.  A SJDT slide to the  inner/outer corner $c$ of $T$  gives a new KN skew tableau
$\text{SJDT}(T,c)$,
possibly with fewer/more boxes. Applying a SJDT slide to the same inner corner $c$ in all vertices of $\textsf{B}(T)$ defines an isomorphic
crystal $\textsf{B}(\text{SJDT}(T,c))$ \cite[Theorem 6.3.8]{lec02}. The images of the KN tableaux in the same connected component of
$\textsf{KN}(\lambda/\mu,m)$ under this crystal isomorphism have the same skew shape \cite[Theorem 6.3.8]{lec02}. Iterating the SJDT to all
inner corners of $T$ rectifies $T$, producing $\textsf{rect}(T)$ \cite[Proposition 9.2]{she99}, \cite[Theorem 6.1.9, Theorem 6.3.9]{lec02}.

At the end of each SJDT slide, the inner corner (outer corner) where the slide started is filled, or the column where the slide started has
two fewer (more) boxes \cite[Proposition 9.2]{she99}, \cite[Theorem 6.1.9]{lec02}. The SJDT step where the tableau loses two boxes in a column
has a previous step where this column is non-admissible but Knuth equivalent to the new column which is admissible. The step in reverse SJDT
where the tableau gains two boxes in a column is $\gls{r3}$ Knuth equivalent to the previous one which is admissible. Therefore, in each step
of SJDT we get crystals which are isomorphic. This allows, in the vein of reversal for $A_{n-1}$ skew semi-standard tableaux, the definition
of symplectic reversal, $\gls{symprev}$, on type $C_n$ skew tableaux as a coplactic extension of $\textsf{evacuation}^{C_n}$.

\begin{lemma}\label{reversal} Let $T\in \gls{B}= \textsf{KN}(\lambda/\mu,n)$.
 Then $\xi^{C_n} (T)$ is the unique KN tableau in $\textsf{B}(T)$ that is symplectic Knuth equivalent to $\textsf{evac}^{C_n}\,
\textsf{rect}(T)$, and
\begin{align}\label{coplactic}
\textsf{rect}\,\xi^{C_n}(T)&=\textsf{evac}^{C_n}(\textsf{rect} (T)).
                                             \end{align}
\end{lemma}

\begin{proof}  The crystal $\textsf{B}(T)\simeq\textsf{B}(\nu)$ for some partition $\nu$ and $\textsf{rect} (T)\in \textsf{B}(\nu)$.
The full Sch\"utzenberger--Lusztig involution on KN tableaux of straight shape satisfies
$\xi^{C_n}(\textsf{rect}(T))=\textsf{evac}^{C_n}(\textsf{rect}(T))$, and  crystal
operators commute with SJDT when passing from $\textsf{B}(T)$ to $\textsf{B}(\nu)$. Therefore, \eqref{coplactic} holds.
\end{proof}

In  Subsection \ref{partialsymplecticreversal} we will provide an algorithm for partial symplectic reversal on $\gls{branchedsymptab}$ with
$J=[j,n]$.  An algorithm for full $C_{n}$ reversal on $\textsf{KN}(\lambda/\mu,n)$ will result as a special case by considering the normal
 sub-crystal $\textsf{B}(\mu,\lambda)$ of  $ \gls{branchedsymptab}$. See Remark \ref{re:fullreversalskew}.

  \section{Internal cactus group action on a normal crystal}\label{sec:internalpartialschutzaction}
 \subsection{Partial Sch\"utzenberger--Lusztig   involutions}
Partial Sch\"utzenberger  involutions were first studied in the case $\gls{g}=\mathfrak{sl}(n,\mathbb{C})$ by Berenstein and Kirillov
\cite{bk95} but have been defined by Halacheva  in general: given  $J\subseteq I$ any sub-diagram, the partial
Sch\"utzenberger--Lusztig involution $\gls{xij}$ is defined to be the Sch\"utzenberger--Lusztig  involution $\xi_{\gls{Bj}}$ on the normal
crystal $\gls{Bj}$  \cite{ha16, ha20} (see also \cite{hakarywe20}).
The crystal $\gls{Bj}$ decomposes into connected components, and we apply the Sch\"utzenberger--Lusztig involution to each connected
component.  Let $b\in \gls{B}$, and let $u^{\textsf{high}}$, $u^{\textsf{low}}$ be  the highest  and lowest weight elements of the connected
component of $\gls{Bj}$ containing $b$.  Let $b=f_{ j_r}\cdots f_{ j_1}( u^{\textsf{high}})$, with $j_r\cdots { j_1}\in J$. Then, for $j \in
J$,
 \begin{align} \gls{xij}e_j(b)&=f_{\theta_J(j)}\gls{xij}(b),\label{partialschutz**}\\
  \gls{xij}f_j(b)&=e_{\theta_J(j)}\gls{xij}(b),\label{partialschutz***}\\
 \textsf{wt}_J(\gls{xij}(b))&=w_0^J\textsf{wt}_J(b),\label{partialschutz*} \end{align}
 and $\gls{xij}(b)= e_{\theta_J (j_r)}\cdots e_{\theta_J( j_1)}(u^{\textsf{low}}).$

\begin{remark}  If $J=K\sqcup K'\subseteq I$ is disconnected with $K$ and $K'$ connected sub-diagrams of $I$, we have the sub-type Dynkin
diagram $K\times K'$, and the Weyl group is $W^K\times W^{K'}$ with longest elements $w_0^K$ and $w_0^{K'}$, respectively, such that
$w_0^J=w_0^Kw_0^{K'}=w_0^{K'}w_0^K$.  The weight lattice of $\mathfrak{g}_{K}\oplus \mathfrak{g}_{K'} $ is $\Lambda_{K\sqcup
K'}:=\Lambda_K\oplus \Lambda_{K'}$  (see \cite{BSch17}).
Then, if $\theta_K$ and $\theta_{K'}$ are the graph automorphisms defined by $w_0^K$ and $w_0^{K'}$ in $K$
and $K'$, respectively, $\theta_J=\theta_K\theta_{K'}=\theta_{K'}\theta_K$ is a graph automorphism of the Dynkin graph $K\times K'$ and hence
preserves the connected sub-diagrams $K$ and $K'$ of $I$ as defined in Section~\ref{sec:basics}.
Thanks to \cite[Lemmas 10.1.3, 10.1.4]{ha16}, \cite[2368--2369]{hakarywe20}, the crystal operators act componentwise on the normal
crystal $\textsf{B}_{K\sqcup K'}$ (a normal $\mathfrak{g}_{K}\oplus \mathfrak{g}_{K'} $--crystal) and satisfy the following properties
\begin{align}
& f_kf_{k'} =f_{k'}f_k,\, f_ke_{k'}=e_{k'}f_k,\, e_ke_{k'}
=e_{k'}e_k,\, e_kf_{k'} =f_{k'}e_k, \hbox{ for } k\in K, k'\in
K',\label{disjoint0} \\
&e_{k'}\xi_{\textsf{B}_K}=\xi_{\textsf{B}_K}e_{k'},\,f_{k'}\xi_{\textsf{B}_K}=\xi_{\textsf{B}_K}f_{k'},
\hbox{ for } k\in K, k'\in
K',\label{disjoint1}\\
&e_{k}\xi_{\textsf{B}_{K'}}=\xi_{\textsf{B}_{K'}}e_{k},\,f_{k}\xi_{\textsf{B}_{K'}}=\xi_{\textsf{B}_{K'}}f_{k},
\hbox{ for } k\in K, k'\in
K'\label{disjoint2},
\end{align} and
$\xi_K$ and $\xi_{K'}$ commute:
$\xi_K\xi_{K'}=\xi_{K'}\xi_K.$
\noindent This extends to a disconnected sub-diagram with more  than two connected sub-diagrams.
\end{remark}

\begin{lemma} \label{th:disconnected}
Let $J=K\sqcup K'\subseteq I$ be a disconnected sub-diagram of $I$  with $K$ and $K'$ connected. Then $\textsf{B}_{K\sqcup K'}$ is a normal
crystal, and the Sch\"utzenberger--Lusztig  involution on $\textsf{B}_{K\sqcup K'}$, $\xi_{K\sqcup K'}$, satisfies $\xi_{K\sqcup
K'}=\xi_K\xi_{ K'}=\xi_{ K'}\xi_K.$
\end{lemma}

 \begin{proof} The result follows from  the previous remark:   $\xi_K\xi_{ K'}=\xi_{ K'}\xi_K$ is an involution, and from
 \eqref{disjoint0}, \eqref{disjoint1} and \eqref{disjoint2},
 it  satisfies
 the conditions \eqref{partialschutz**}, \eqref{partialschutz***} above. In addition, the weight map  $\textsf{wt}_{K\sqcup
 K'}:\gls{B}\overset {\textsf{wt}}\rightarrow \Lambda \overset{can}\rightarrow \Lambda_{K\sqcup K'}=\Lambda_K\oplus \Lambda_{K'}$ and
therefore,  $\textsf{wt}_{K\sqcup K'}(\xi_K\xi_{ K'}(b))=(\textsf{wt}_{K}(\xi_K(b)),\textsf{wt}_{K'}(\xi_{
K'}(b))=w_0^Kw_0^{K'}(\textsf{wt}_{K}(b),\textsf{wt}_{K'}(b))$.
  Since there is only one set involution on  $\textsf{B}_{K\sqcup K'}$
 satisfying \eqref{partialschutz**}, \eqref{partialschutz***}, and \eqref{partialschutz*}, we have that $\xi_{K\sqcup K'}=\xi_K\xi_{ K'}=\xi_{
 K'}\xi_K.$
 \end{proof}

The partial Sch\"utzenberger--Lusztig involutions $\gls{xij}$, for any $J \subseteq I$   a \emph{connected Dynkin sub-diagram} of $I$, satisfy
the $J_{\gls{g}}$ cactus relations.

  \begin{theorem}  [\cite{ha16}] \label{cactusg} The map $s_J\mapsto \gls{xij}$, for all $J \subseteq I$ connected Dynkin sub-diagrams of $I$,
  defines an action of the cactus group $J_{\gls{g}}$ on the set $\gls{B}$; that is, the following is a group homomorphism
  $$\begin{array} {cccccc}
  \Phi_{\gls{g}}:&J_{\gls{g}}&\rightarrow&\mathfrak{S}_{\gls{B}} \\
  &s_J&\mapsto&\gls{xij}.\\
 \end{array}
 $$
   Moreover
  $ \textsf{wt}_J(\gls{xij}(b))=w_0^J\textsf{wt}_J(b)$, $b\in \gls{B}$.
  \end{theorem}

The $s_J$ act via $\gls{xij}$ on each connected component of $\textsf{B}_J$ as a graph automorphism of the underlying unlabeled, non-oriented
and non-weighted connected graph by exchanging highest and lowest weight elements.

\begin{remark}\label{cactus} The Weyl group $W$ of $\mathfrak{g}$ acts on the set $\textsf{B}$ by $ r_i.b=\xi_i(b)$, $i\in I$, where
$\xi_i=\xi_{\{i\}}$, see  Section \ref{sec:solenodeweyl} and Theorem \ref{weylactionkas}. The action of $J_{\gls{g}}$ factorizes then through the quotient by the braid relations
of
$W$.
\end{remark}

The following corollary motivates what comes in the next section.

\begin{corollary}\label{Acactusaction}

$(a)$ For the $\mathfrak{sl}(n,\mathbb{C})$-crystal \gls{ssyt}, the map $$s_{[1,j]}\mapsto \xi_{[1,j]}=\textsf{evac}_{j+1}, \;1\le j\le n-1,$$
\noindent where $\textsf{evac}_{j+1}$ denotes the evacuation on the sub-tableaux  of straight shape obtained by restricting the entries to
$\{1,\dots,j+1\}$ and fixing the remaining ones,  defines an action of the cactus group \gls{Jn} on the set \gls{ssyt}.

$(b)$ For the $\mathfrak{sp}(2n,\mathbb{C})$-crystal \gls{symptab}, the map
\begin{eqnarray}\label{cactusCa}s_{[1,j]}&\mapsto& \xi^{C_n}_{[1,j]},\;  \quad 1\le j\le n-1,\\
~~~~~~~~~~~~s_{[j,n]}&\mapsto&\xi^{C_n}_{[j,n]},\quad 1\le j\le n,\label{cactusC}
\end{eqnarray}
defines an action of $\gls{Jsp}$ on the set \gls{symptab}, where $\xi^{C_n}_{[1,n]}=\xi^{C_n}=\textsf{evac}^{C_n}$,
$\xi^{C_n}_{[1,j]},\;  1\le j\le n-1$ is given by the Baker embedding, Theorem \ref{goingandbacR1}, and $\xi^{C_n}_{[j,n]}, 1\le
j\le n-1$ is given either by the partial
symplectic reversal in \eqref{formula}, Subsection \ref{partialsymplecticreversal} or by the Baker embedding, Theorem \ref{goingandbacR2}.
\end{corollary}

\subsection{ The virtual symplectic cactus group action on an \texorpdfstring{$\mathfrak{sl}(2n,\mathbb{C})$}{sl2n}-crystal and the virtualization of an
\texorpdfstring{$\mathfrak{sp}(2n,\mathbb{C})$}{sp2n}-crystal }\label{subsec:internalvirtual}
On $A_{n-1}$ semi-standard tableaux, there is a straightforward algorithm to compute the action of a partial Sch\"utzenberger--Lusztig
involution $\gls{xij}$ with $J$ a  connected $A_{n-1}$  Dynkin sub-diagram. Let $I = [n-1]$ and $J=[p,q]\subset I$,  $1\le p\le q<n$, be a
connected sub-diagram. The \emph{$J$-partial reversal}, \gls{jrev}, is the \gls{rev} on $\gls{branchedssyt}$ which means the \gls{rev} or
Sch\"utzeberger involution $\gls{xi}$ applied to each connected component of $\gls{branchedssyt}$. Let $T \in \gls{ssyt}$, then, from
\eqref{fullformula} and \eqref{reversalgln}:
 \begin{align}
 \label{def:reversal}
\gls{xij}(T)&=\gls{jrev}(T)\nonumber\\
&:=(T_{[1,p-1]},\gls{rev}(T_{[p,q+1]}),T_{[q+2,n]} )\nonumber\\
&=(T_{[1,p-1]},\textsf{arect}(\textsf{evacuation}(\textsf{rect} (T_{[p,q+1]}))), T_{[q+2,n]}),
\end{align}
\noindent where $T=(T_{[1,p-1]},T_{[p,q+1]}, T_{[q+2,n]})$ is such that $T_{[1,p-1]}$ is the tableau obtained by restricting $T$ to the
alphabet $[1,p-1]$, $T_{[p,q+1]}$ is  the skew tableau obtained by restricting to the alphabet $[p,q+1]$, and $T_{[q+2,n]}$ is obtained by
restricting to the alphabet $[q+2,n]$. Indeed, if $J=[1,q]$, $ \textsf{reversal}_{[1,q]}(T)=\textsf{evac}_{q+1}(T)$. The case where $J$ is a
disconnected sub-diagram of $I$ will be a consequence of Lemma \ref{th:disconnected}.

To define an internal action of the \emph{virtual symplectic cactus group} $\gls{vJ2n}$ on a  crystal $\textsf{SSYT}(\mu, n,\bar n)$ with
$\mu$ a partition  with at most $2n$ parts, thanks to Lemma \ref{th:disconnected}, we now explicitly characterize the partial
Sch\"utzenberger--Lusztig involution on a disconnected sub-diagram $J\sqcup J'$ of the $A_{2n-1}$ Dynkin diagram such that $J\subseteq [n-1]$
and $ J' \subseteq [\bar n, \bar 2]$ are  connected sub-diagrams. In the case of the $A_{2n-1}$ Dynkin diagram,  we label its nodes either in
$[2n-1]$ or in $\{1,\dots,n,\bar n,\dots,\bar 2\}$.

\begin{theorem}\label{th:schutzunion} Let $J\sqcup J'$ be a disconnected sub-diagram of the $A_{2n-1}$ Dynkin diagram $I=\{1,\dots, n,\bar
n,\dots, \bar 2\}$ such that $J\subseteq [n-1]$ and $ J' \subseteq [\bar n, \bar 2]$ are  connected sub-diagrams. Then  $\xi_{J\sqcup
J'}^{A_{2n-1}}$, the Sch\"utzenberger--Lusztig involution on $\textsf{SSYT}_{J\sqcup J'}(\mu, n,\bar n)$, with $\mu$ a partition with at most
$
2n$ parts, satisfies
\begin{align} \xi_{J\sqcup J'}^{A_{2n-1}}&=\xi_{J}^{A_{2n-1}}\xi_{ J'}^{A_{2n-1}}=\xi_{ J'}^{A_{2n-1}}\xi_{J}^{A_{2n-1}}\label{schutzunion}\\
&=\textsf{reversal}^{A_{2n-1}}_{J}\textsf{reversal}^{A_{2n-1}}_{J'}=\textsf{reversal}^{A_{2n-1}}_{J'}\textsf{reversal}^{A_{2n-1}}_{J},\label{schutzunion2}
\end{align}
where $\xi_{J}^{A_{2n-1}}=\textsf{reversal}^{A_{2n-1}}_{J}$ and $\xi_{ J'}^{A_{2n-1}}=\textsf{reversal}^{A_{2n-1}}_{J'}$  are the
Sch\"utzenberger--Lusztig involutions on $\textsf{SSYT}_{J}(\mu, n,\bar
n)$ and $\textsf{SSYT}_{ J'}(\mu, n,\bar n)$, respectively.
\end{theorem}

 \begin{remark}\label{re:disconnectedsubdiag}
 This statement is indeed also valid for the  Sch\"utzenberger--Lusztig involution on $\textsf{SSYT}_{J\sqcup J'}(\mu,n)$ where $J\sqcup J'$
 is a
 disconnected sub-diagram of the $A_{n-1}$ Dynkin diagram with $n$ odd.
 \end{remark}

The cactus group $J_{2n}$ acts on an $A_{2n-1}$-crystal of semi-standard tableaux via partial reversals. We now conclude that the virtual
symplectic cactus $\gls{vJ2n}$  also does. In the next section, Subsection \ref{sec:virtaction}, we establish  that this action preserves the
subset $\gls{E}(\gls{symptab})$.

\begin{theorem}   \label{cactusj2n} For the $\mathfrak{sl}(2n,\mathbb{C})$-crystal of tableaux $\textsf{SSYT}(\mu,2n)$, with $\mu$ a partition
with at most $ 2n$ parts, the map
\begin{align}\tilde s_{[1,q]\sqcup [2n-q,2n-1]}&\mapsto   \xi^{A_{2n-1}}_{[1,q]\sqcup
[2n-q,2n-1]}&=&\xi^{A_{2n-1}}_{[1,q]}\xi^{A_{2n-1}}_{[2n-q,2n-1]}&\nonumber \\
&&=&\textsf{evac}_{q+1}\textsf{evac}_{2n}\textsf{evac}_{q+1}\textsf{evac}_{2n}, &1\le q< n,\label{cactusj2n1}\\
\tilde s_{[q,2n-q]}&\mapsto   \xi^{A_{2n-1}}_{[q,2n-q]}&=&\textsf{reversal}^{A_{2n-1}}_{[q,2n-q]},& 1\le q\le n,\label{cactusj2n2}
\end{align}
\noindent defines an action of the virtual symplectic cactus group
  $\gls{vJ2n}$ on the set $\textsf{SSYT}(\mu,2n)$. That is, the following is a group homomorphism
  $$\begin{array} {cccccc}
  \tilde \Phi_{\mathfrak{sl}(2n,\mathbb{C})}:&\gls{vJ2n}&\rightarrow&\mathfrak{S}_{\gls{B}} \\
  &\tilde s_J&\mapsto&\xi^{A_{2n-1}}_J,\\
 \end{array}
 $$
 where $ \textsf{B} =\textsf{SSYT}(\mu,2n)$ and $J$ as in \eqref{cactusj2n1}or  \eqref{cactusj2n2}. Moreover, the action of
 $\gls{vJ2n}$ on $\gls{ssytv}$
preserves the subset $\gls{E}(\gls{symptab})$.

  \end{theorem}
  \begin{proof}
  Since $J_{2n}$ acts on  $\textsf{SSYT}(\mu,2n)$, the partial Sch\"utzenberger involutions $\gls{xij}$, with $J$ a connected
  sub-diagram of the $A_{2n-1}$ Dynkin diagram $I=[2n-1]$,
  satisfy the $J_{2n}$ cactus relations namely the ones in Lemma \ref{symmetrieslemma} which are the $\gls{vJ2n}$ relations. We
  consider $\gls{vJ2n}$ with generators \eqref{virtualalterC1}, \eqref{virtualalterC2}. In Subsections \ref{sec:embedsymplecticreversal} and
  \ref{sec:virtaction}, \eqref{preserve2},
 we conclude that $\gls{vJ2n}$ acts on the set $\gls{ssytv}$   permuting its elements in a way that
the subset $\gls{E}(\gls{symptab})$ is preserved.
  \end{proof}

  Therefore, the partial Sch\"utzenberger involutions $\gls{xij}$, with $J$ any connected sub-diagram of the $A_{2n-1}$ Dynkin diagram of the
  form $J = [q,2n-q]$, $[q,n] \subseteq [n]$, or $J = [1,q] \sqcup [2n-q,2n-1]$, $ [1,q] \subseteq [n-1]$, satisfy the virtual symplectic
  cactus
  $\gls{vJ2n}$ relations.

 \section{Partial symplectic Sch\"utzenberger--Lusztig involutions and algorithms}\label{sec:partialschutzalg}

For a connected sub-diagram $J$ of the Dynkin diagram $ I=[n-1]$ of type  $A_{n-1}$, the partial Sch\"utzenberger involution $\gls{xij}$
coincides with $J$-\emph{partial reversal}, that is,  $\gls{jrev}$ \eqref{def:reversal}. The case wherein $J$ is a disconnected sub-diagram of
$I$ has been  studied in Theorem \ref{th:schutzunion} and Remark \ref{re:disconnectedsubdiag}.

So far, there is no known form of tableau-switching for KN tableaux. The algorithm  to compute  $J$-\emph{partial symplectic reversal},
$\gls{jsymprev}$, with $J=[p,n]$ a sub-diagram of the Dynkin diagram $I$ of type $C_n$, presented in Subsection
\ref{partialsymplecticreversal0} and summarized in \eqref{formula}, is inspired by this problem and mimics the type $A$ partial reversal
algorithm on type $A_{n-1}$ tableaux summarized in \eqref{def:reversal}. The case $J=[p,q]\subseteq I$, $p<q<n$, is solved by virtualization
in Subsection \ref{sec:embedsymplecticreversal}. In fact, all partial symplectic reversals can be virtualized as shown in Subsection
\ref{sec:embedsymplecticreversal}.

 \subsection{ Dynkin sub-diagram with a sole node and the Weyl group action} \label{sec:solenodeweyl} Let $\gls{B}$ be a normal crystal. If $J$ has a sole node $i$ of
 $I$, $\xi_i:=\xi_{\{i\}}$, the Sch\"utzenberger--Lusztig involution on the $i$-strings (the connected components of  $\textsf{B}_{\{i\}}$),
 agrees with the Kashiwara $\gls{g}$-crystal reflection operator $  S_i$ {  \cite[Section 7]{kash94} originally studied by Lascoux and
 Sch\"utzenberger in the
 $\mathfrak{sl}(n,\mathbb{C})$ case \cite{LaSchu81} and rediscovered by Kashiwara for any Cartan type. Let $b\in \gls{B}$,  and
 $k=\varphi_i(b)-\varepsilon_i(b)$. The  \emph{crystal reflection operator} $S_i$ is defined by
 \begin{equation}S_i(b)=\begin{cases}\; f_i^k(b),& \mbox{if $k\ge 0$},\\
  e_i^{-k}(b),& \mbox{if $k\le 0$}.
  \end{cases}
  \label{crystalreflection}
  \end{equation} 

The following assertion combines the
 Sch\"utzenberger--Lusztig involution $\xi_i$  on the $i$-strings  and the Kashiwara crystal reflection operator $S_i$  \cite[Section 7]{kash94}.

 \begin{theorem}\label{weylactionkas}  The operators $\xi_i$, $i\in I$, define an action of the Weyl group $W$ on the
 underlying set of the normal crystal $\gls{B}$, $r_i.b=  S_i(b)=\xi_i(b)$, $b\in  \textsf{B}$, such that
 \begin{enumerate}
 \item $  S_i^2=1$ and  $  e_iS_i=S_if_i$  $(  S_i(0):=0)$,
   \item $r_i .\textsf{wt}( b)=  \textsf{wt} (S_i(b))$,
   \item $u_\lambda^{\textsf{low}}=w_0.u_\lambda^{\textsf{high}},$  if $\textsf{B}=\gls{cncrystal}$.
   \end{enumerate}

 \end{theorem}

 The action of $\xi_i$ on the $i$-string of $b\in \gls{B}$: $\varphi_i(\xi_i(b))=\varepsilon_i (b)$ or
 $\varepsilon_i(\xi_i(b))=\varphi_i (b)$. An illustration with $k>0$:

\begin{center}
\begin{tikzpicture}[xscale=1,yscale=1]
\draw[gray, thick] (-2,2) -- (10,2);
\filldraw[black] (-2,2) circle (2pt) ;
\filldraw[black] (2,2) circle (2pt);
\filldraw[black] (1,2) circle (2pt) ;
\filldraw[black] (3,2) circle (2pt) ;
\filldraw[black] (10,2) circle (2pt) ;
\filldraw[black] (6,2) circle (2pt) ;
\draw (1,1.6) node {\scriptsize$e_i(b)$};
\draw (2,1.6) node {\scriptsize$b$};
\draw (3,1.6) node {\scriptsize$f_i(b)$};
\draw (6,1.6) node {\scriptsize$  \xi_i(b)=f_i^{\varphi_i(b)-\varepsilon_i(b)}(b)$};
\draw (10,1.5) node {\scriptsize$f_i^{\varphi_i(b)}(b)$};
\draw (-2,1.5) node {\scriptsize$ e_i^{\varepsilon_i(b)}(b)$};
\draw [decorate, decoration = {brace, amplitude = 5pt}, yshift = 5pt] (-2,2) -- (2,2) node [black, midway, yshift = 10pt] {{\scriptsize
$\varepsilon_i(b)$}};
\draw [decorate, decoration = {brace, amplitude = 5pt}, yshift = 5pt]
(6,2) -- (10,2) node [black, midway, yshift = 10pt] {{\scriptsize
$\varphi_i(\xi_i(b))=\varepsilon_i(b) $}};
\end{tikzpicture}
\end{center}
}
The propositions below follow from the action of the Weyl group $W$ on $i$-strings, where useful information is gathered in the case of $\mathfrak{sp}(2n,\mathbb{C}) $.

Notice  that  the action  of $S_i=\xi_i $ \eqref{crystalreflection}  on a KN tableau of type $C_n$ is given by the \textit{signature rule} on its
reading word \cite{kn91,lec02} in the alphabet  $\gls{Corderedletters}$. Thus, the first point in the next proposition is just the translation of \eqref{crystalreflection} on words in the alphabet    $\gls{Corderedletters}$.
Let $i\in I$ and let $w$ be a word in the alphabet $\gls{Corderedletters}$ restricted to the alphabet $\{i<i+1<\overline {i+1}<\bar i\}$.  Substitute each letter in $ \{i, \overline{i+1}\}$ by $+$  and by $-$ if in  $\{i+1, \bar{i}\}$. Then    bracket all factors $+-$ and erase all  $(+-)$. The remaining $+,-$ form a subword $-^r +^s$ with  $r=\varepsilon_i(w)$  and $s=\varphi_i(w)$.
The Kashiwara operator $e_i$ acts on the letter $w_j$ of $w$ associated to the rightmost unbracketed $-$ (i.e., not erased), whereas $f_i$ acts on the letter $w_j$ of $w$ associated to the leftmost unbracketed $+$, and the other letters of $w$ are unchanged:
\begin{align}e_i(w_j)=\begin{cases}
	i\,\,&\text{if } w_j=i+1 \text{ and } i\neq n \\
	\overline{i+1}\,\,&\text{if }  w_j=\overline{i} \text{ and } i\neq n\\
	{n}\,\, & \text{if } w_j=\overline n \text{ and } i=n
	\end{cases},& \quad \mbox{ and  $f_i$ is the inverse map}.\label{signatureC}\end{align}
  If $s=0$ then $f_i(w)=0$ and if $r=0$ then $e_i(w)=0$.

\begin{proposition}\label{weylaction}  Let $W=B_n$.
\begin{enumerate}
\item For $U_q(\mathfrak{sp}(2n,\mathbb{C}))$ and the alphabet $\gls{Corderedletters}$:  given $i\in [n-1]$, let
$u^-$ be a word in the alphabet  $\{\overline i, i+1\}$ with length $\ell(u^-)=r$, and let $v^+$ be a word in the alphabet
$\{i,\overline{i+1}\}$ with length $\ell(v^+)=s$. Then, for all $r_i\in \gls{bn}$, $1\le i\le n-1$,
\begin{equation}\label{Weylaction1}  r_i.(u^-v^+)=\xi_i(u^-v^+)=\begin{cases} u_1^-e_i^{r-s}(u_2^-)v^+,&\,  r>s\\
u^-v^+, &\,r=s\\
u^- f_i^{s-r}(v_1^+)v_2^+ ,& \,r<s\end{cases},\end{equation}
\noindent such that when $r>s$,  $u^-=u_1^-u_2^-$, with $\ell(u_2^-)=r-s$, and when $r<s$, $v=v_1^+v_ 2^+$ with $\ell(v_1^+)=s-r$.
When $i=n$,
\begin{equation}\label{Weylaction2}  r_n.({\overline n}^r{ n}^s)=\xi_n({\overline n}^r{ n}^s)={\overline n}^s{ n}^r .\end{equation}

\item \label{refection c} For $U_q(\mathfrak{sp}(2n,\mathbb{C}))$: the crystal reflection operators $\xi_i$ satisfy the relations of the
    Weyl group $\gls{bn}$: \begin{itemize}
\item  $\xi_i^2=1$, $1\le i\le n$,\; and $\xi_i\xi_j=\xi_j\xi_i$, $|i-j|>1$, $1\le i,j\le n$;
\item $(\xi_i\xi_{i+1})^3=1$, $1 \le i \le n - 2$,
 and $(\xi_{n-1}\xi_n)^4=1.$

\end{itemize}
\end{enumerate}
\end{proposition}

\begin{example} 
 Let $n=4$  and $T=\YT{0.17in}{}{
			{1,2,2,3,\overline{2},\overline{1}},
			{2,\overline{4},\overline{3},\overline{3},\overline{1}},
{4,\overline{2},\overline{1}},
{\overline{4}}}. $ From \eqref{Weylaction1}, \eqref{Weylaction2}, the action  of $\xi_i $  on the KN tableau $T$ is given by the {signature rule} on its
reading word as follows.

\begin{enumerate}

\item  For $i=1$,     the letters $\{1<2<\bar 2<\bar 1\}$ in $T$  are  highlighted in blue. 
 The word of $T$ restricted to the alphabet $\{1<2<\bar 2<\bar 1\}$ is $w=$ $\color{blue}\overline{1}\,\color{blue}\,\overline{2}\,\color{blue}\overline{1}\,{2}	\,
\color{blue}\overline{1}\,2\,
\color{blue}\overline{2}\,1\,2$.  The red color  indicates the iterated action of $e_1$ on $w$:
\begin{align*} w&={\color{blue}\overline{1}\,\color{blue}\,\overline{2}\,\color{blue}\overline{1}\,{2}	\,
\color{blue}\overline{1}\,2\,
\color{blue}\overline{2}\,1\,2}\rightarrow {-(+-)---+(+-)}\rightarrow ----+=(-)^4(+)^1
\\&
\stackrel{(e_1)^3}{\rightarrow} -{\color{red} +++}+=(-)^1(+)^4\rightarrow
-(+-){\color{red}+++}+(+-)\rightarrow{\color{blue}\overline{1}\,\color{blue}\overline{2}\,\color{blue}\overline{1}\,\color{red}1	
\color{red}\overline{2}\,\color{red}1\, \color{blue}\overline{2}12}=\xi_1(w).
\end{align*}
Hence,
$$T=\YT{0.17in}{}{
			{\color{blue}1,\color{blue}2,\color{blue}2,3,\color{blue}\overline{2},\color{blue}\overline{1}},
			{\color{blue}2,\overline{4},\overline{3},\overline{3},\color{blue}\overline{1}},
{4,\color{blue}\overline{2},\color{blue}\overline{1}},
{\overline{4}}}\rightarrow \YT{0.17in}{}{
			{\color{blue}+,\color{blue}-,\color{blue}-,3,\color{blue}+,\color{blue}-},
			{\color{blue}-,\overline{4},\overline{3},\overline{3},\color{blue}-},
{4,\color{blue}+,\color{blue}-},
{\overline{4}}}\rightarrow \YT{0.17in}{}{
			{\color{blue}+,\color{red}+,\color{red}+,3,\color{blue}+,\color{blue}-},
			{\color{blue}-,\overline{4},\overline{3},\overline{3},\color{blue}-},
{4,\color{blue}+,\color{red}+},
{\overline{4}}}$$
$$
\rightarrow\YT{0.17in}{}{
			{\color{blue}+,\color{red}1,\color{red}1,3,\color{blue}+,\color{blue}-},
			{\color{blue}-,\overline{4},\overline{3},\overline{3},\color{blue}-},
{4,\color{blue}+,\color{red}\overline 2},
{\overline{4}}}
\rightarrow \xi_1(T)=\YT{0.17in}{}{
			{\color{blue}1,\color{red}1,\color{red}1,3,\color{blue}\bar 2,\color{blue}\bar 1},
			{\color{blue}2,\bar{4},\bar{3},\bar{3},\color{blue}\bar 1},
{4,\color{blue}\bar 2,\color{red}\overline 2},
{\overline{4}}},$$
\noindent where ${\textsf{wt}}(\xi_1(T))=r_1. {\textsf{wt}}(T)=r_1(-2,1,-1,-1)= (1,-2,-1,-1).$

\item For $i=4$,   the letters $\{4,\bar 4\}$ in $T$ are  highlighted in blue. Now, \begin{align*}
    w&={\color{blue} \bar 4\,4\,\bar 4} {\rightarrow } { -}\,(+\,-)\stackrel{e_4}{\rightarrow } {\color{red} +}\,(+\,-)\rightarrow  {\color{red}4}\,{\color{blue}4\,\bar 4}=\xi_4(w) , \mbox{ and}\end{align*}
$$\xi_4(T)=\xi_4\YT{0.17in}{}{
			{1,2,2,3,\overline{2},\overline{1}},
			{2,\color{blue}\overline{4},\overline{3},\overline{3},\overline{1}},
{\color{blue}4,\overline{2},\overline{1}},
{\color{blue}\overline 4}}=\YT{0.17in}{}{
			{1,2,2,3,\overline{2},\overline{1}},
			{2,\color{red}{4},\overline{3},\overline{3},\overline{1}},
{\color{blue}4,\overline{2},\overline{1}},
{\color{blue}\overline 4}}$$
\noindent where $\textsf{wt}(\xi_4(T))=r_4.\textsf{wt}(T)=r_4(-2,1,-1,-1)=(-2,1,-1,1).$ 
\end{enumerate}
\end{example}


\begin{proposition}\label{prop:slandweyl} Let $\gls{cncrystal}$ be a $C_n$ crystal and $\gls{Bj}= \gls{Bj}(\lambda)$ for  $J\subseteq I$. Let
$b\in \gls{cncrystal}$. The connected component of $\gls{Bj}$ containing $b$ has highest weight element $b_J^{\textsf{high}}$ and  lowest
weight element $b_J^{\textsf{low}}$. Then
\begin{enumerate}
\item $b_J^{\textsf{low}}=r_{a}\cdots r_{d}.b_J^{\textsf{high}}=\xi_{a}\cdots \xi_{d}(b_J^{\textsf{high}})$ where $r_a\cdots r_d$ is a
    reduced word for $ w_0^J \in W^J$ with $a,\dots,d\in J$, and   $b=f_{ j_r}\cdots f_{ j_1}( b_J^{\textsf{high}})$ for some $j_r,\dots,{
    j_1}\in J$.
\item If  $J=[p,n]$, $\textsf{B}_{[p,n]}$ is a $C_{n-p+1}$ crystal, then $$\gls{xij}(b)= e_{ j_r}\cdots e_{ j_1}(r_{a}\cdots
    r_{d}.b_J^{\textsf{high}}),\;\textsf{wt}_J(\gls{xij}(b))=-\textsf{wt}_J(b).$$
\item If $J=[p,q]$, $1\le p\le q<n$, $\textsf{B}_{[p,q]}$ is an $A_{q-p+1}$, crystal, and $$\gls{xij}(b)= e_{q-p- j_r+1}\cdots
    e_{q-p-j_1+1}(r_{a}\cdots r_{d}.b_J^{\textsf{high}}),\;\textsf{wt}_J(\gls{xij}(b))= rev(\textsf{wt}_J(b)).$$
\end{enumerate}

\end{proposition}

 \subsection{Dynkin sub-diagram \texorpdfstring{$J=[j,n]$}{jn}: J-symplectic reversal}
 \label{partialsymplecticreversal0}

Recall that the partial Sch\"utzenberger--Lusztig involution $\xi^{C_n}_{[j,n]}$ is  the Sch\"utzenberger--Lusztig involution on each
connected component  of $\textsf{KN}_{[j,n]}(\lambda,n)$. On the set \gls{symptab}, $\xi^{C_n}=\xi^{C_n}_{[1,n]}$ coincides with Santos'
symplectic evacuation $\textsf{evac}^{C_n}$ (see Section \ref{santosevac} or \cite[Section 5]{sa21a}).

\subsubsection{The  Knuth  \gls{r3} relation on a skew tableau}
\label{partialsymplecticreversal}
Given $1<j\le n$,  the Levi branched crystal $\textsf{KN}_{[j,n]}(\lambda,n)$  decomposes into connected components. Let $T\in \gls{symptab}$,
which  belongs to some connected component of $\textsf{KN}_{[j,n]}(\lambda,n)$, and let $T_{[\pm j,n]}$ denote the restriction of  $T$ to the
alphabet $[\pm j,n]$. We claim that the connected component containing $T_{[\pm j,n]}$ is symplectic Knuth equivalent to a crystal connected
component of admissible skew tableaux on the alphabet $[\pm j,n]= [\pm 1,n]\setminus \{1<\cdots<j-1<\overline
    {j-1}<\cdots<\bar 2<\bar 1\}$ (of the same skew shape). The following remark gathers the missing
observations we need to verify our claim.

\begin{remark}\label{actionK3}
\begin{itemize}
\item \cite[Proposition 2.3.3]{lec02} Let $C_1,\dots,C_k$ be admissible columns on the alphabet $\gls{Corderedletters}$. Then $T = C_1C_2
    \cdots C_k$ is a KN tableau on the alphabet $\gls{Corderedletters}$ if and only if $ r (C_i) \le l( C_{i+1})$, that is, if $ r (C_i)
    l(C_{i+1})$ is a type $A_{2n-1}$ semi-standard  tableau for $i = 1,\dots, k-1$.

\item \label{skew0} For $T\in \gls{symptab}$, the restriction of $T$ to the alphabet $[\pm j,n]$, $ T_{[\pm j,n]}$, is a KN skew tableau on the alphabet $\gls{Corderedletters}$, but $ T_{[\pm j,n]}$ may have
    non-admissible columns with respect to the alphabet $[\pm j,n]$. In particular, we might produce a non-admissible skew tableau on the
    alphabet $\mathcal{C}_{n-j+1}$ (recall Definition \ref{def:admissible} and Remark \ref{lecouvey2.2.2}).

\item \label{k5skew} Let $C_1$ and $C_2$ be two columns with entries on the alphabet $[\pm j,n]$ such that $C_1C_2$ is a KN skew  tableau on
    the alphabet $\gls{Corderedletters}$. Assume  that $C_1$ and $C_2$  have exactly $m\ge 0$ and $t\ge 0$ pairs of symmetric entries
    $(x,\bar x)$, respectively, with $N(x)>x$ with respect to the alphabet $[\pm j,n]$. Then $C_1$ has at least $m$ boxes strictly below the
    row containing the last box of $C_2$, and $C_2$ has at least $t$ boxes strictly above the row containing the top box of column $C_1$.

\item Let $(\gls{r3})^m$ denote the iteration of the  Knuth relation \gls{r3} $m\ge 0$ times, and let $C_1$ and $C_2$ be two columns with
    conditions as described in the previous point of this remark. Let $C_1\overset{(\gls{r3})^m} \sim X$, where $X$ is an admissible column
    on $[\pm j,n]$, and $ C_2\overset{(\gls{r3})^t} \sim Y$, where $Y$ is an admissible column on $[\pm j,n]$. Then $
    C_1C_2\overset{(\gls{r3})^m}\sim XC_2\overset{(\gls{r3})^t} \sim XY$ is a KN skew tableau on $[\pm j,n]$.

\end{itemize}
\end{remark}

\subsubsection{Reduced symplectic jeu de taquin} Given $T\in \textsf{KN}(\lambda/\mu,n)$ and $j\in [n]$ such that  $T$ has all  its entries  in
$[\pm j,n]$, the following is an algorithm to compute the \textit{reduced symplectic jeu de taquin} on $T$ on the interval $[\pm j,n]$,
denoted $SJDT_j$. The skew tableau $T$ might not be admissible on the alphabet $[\pm j,n]$. This means that we apply the SJDT after shifting
all entries in $T$ by $-(j-1)$ and iterating on $T$ the contraction relation $\gls{r3}$ the needed number of times to get an admissible skew
tableau on the alphabet $\mathcal{C}_{n-j+1}$. When $j=1$, we recover the ordinary SJDT.
\begin{definition}{Reduced SJDT} ($SJDT_j$)\label{def:reduced}
\begin{itemize}
    \item Let $T_j$ be the tableau obtained by replacing each non-barred entry $c$ and barred entry $\bar c$ in $T$ by $c-j+1$ and
        $\overline{ c-j+1}$, respectively.
    \item If $T_j$ is not a KN tableau in $\textsf{KN}(\lambda/\mu,n-j+1)$, we have some columns containing pairs of the form $b,
        \overline{b}$ such that $b \in [n-j+1]$ is lowest in the column and $N(b)> b$. Iteratively, we apply the  Knuth contraction
        $\gls{r3}$  to $T_j$ until we make all columns admissible. Define $T_j$ to be the resulting tableau with all admissible
        columns.
    \item Compute SJDT on $T_j$ as usual.
    \item Replace each non-barred entry $m$ and $ \overline m$ in $SJDT(T_j)$ by $m+j-1$ and $\overline{m-j+1}$, respectively.
\end{itemize}
\end{definition}

The \emph{reduced rectification}  to the alphabet $[\pm j,n]$, denoted $\textsf{rectification}_j$ ($\textsf{rect}_j$ for short), of $T$ is the
iteration of $SJDT_j$ to all inner corners in $T$. Indeed,
$\textsf{rect}_j(T)$ is the shift by $j-1$ of all entries of
$\textsf{rect}(T_j)$. When $j=1$ we recover the ordinary rectification.

\begin{example}\label{ex:sjdt}
For $T\in \textsf{KN}( (2,2,1)/ (1), 3 )$, we compute a complete SJDT slide on the interval $[\pm 1,3]$:
$$T=\YT{0.17in}{}{
	    	{{\ast},{2}},
			{{3},{\overline{2}}},
			{ {\overline{3}} }}
			\overset{\textsf{SJDT}}\rightarrow
\YT{0.17in}{}{
	    	{{1},{\overline{1}}},
			{{3},{\ast}},
			{ {\overline{3}} }}.
			$$
However, if we compute the complete $SJDT_2$ slide, that is, the complete SJDT slide reduced to the interval $[\pm 2,3]$, we get:
$$\YT{0.17in}{}{
	    	{{\ast},{2}},
			{{3},\overline{2}},
			{ \overline{3} }}
			\rightarrow
T_2=\YT{0.17in}{}{
	    	{{\ast},{1}},
			{{2},\overline{1}},
			{ \overline{2} }}
			\overset{\gls{r3}}\rightarrow
\YT{0.17in}{}{
	    	{{\ast}},
			{2},
			{ \overline{2} }}
\overset{\textsf{SJDT}}\rightarrow
\YT{0.17in}{}{
	    	{{2}},			
{\overline{2}},
			{ {\ast} }}
			\Rightarrow \YT{0.17in}{}{
	    	{{\ast},{2}},
			{{3},\overline{2}},
			{ \overline{3} }}
			\overset{SJDT_2}\rightarrow \YT{0.17in}{}{
	    	{{3}},			
{\overline{3}},
			{ {\ast} }}.$$
\end{example}

\subsubsection{Partial symplectic reversal: colorful symplectic tableau switching.} \label{colors}

 We now generalize  the Benkart--Sottile--Stroomer reversal \cite{bss}
on skew semi-standard 
tableaux to symplectic KN skew  tableaux. The procedure consists of the  \emph{colorful symplectic tableau switching}  governed by SJDT and, consequently, also by the Knuth  $\gls{r3}$ relation   due to SJDT B.2 case; symplectic Santos evacuation; and reverse SJDT (RSJDT). More generally, we define   $J=[j,n]$-\textit{partial symplectic reversal} for KN skew tableaux,  where  $J$ is a Dynkin sub-diagram of the ambient Dynkin diagram of type $C_n$ containing the node $n$, $1\le j\le n$. 
Supplementary ordered \textit{colored letters} 
will be needed to record the use of the Knuth contractor $\gls{r3}$ relation    when non-admissible columns occur: purple \eqref{coloralphabet} for the possible  non-admissible columns in $T_{[\pm j,n]}$, Figure \ref{fig:colorful}, and red for SJDT B.2 case \eqref{coloralphabetall}. 
For illustrations, we refer to Subsection \ref{ex:partialsymplecticreversal}.

Let $T\in \gls{symptab}$ and  $j\ge 1$. Let $\gls{B}$ be the crystal connected component of $\textsf{KN}_{[j,n]}(\lambda,n)$ containing
$T$. $\gls{B}$ is a highest weight crystal and all vertices of $\gls{B}$ are KN tableaux on the alphabet $\gls{Corderedletters}$, with the
letters in $[\pm 1,j-1]$ frozen, as the crystal operators in $\gls{B}$ are indexed by $[j,n]$ and do not act on the entries filled in $[\pm 1,
j-1]$.

Let $H$ be the highest weight element of $\gls{B}$,  and let  $\textsf{wt}(H_{[\pm j,n]})\in \mathbb{Z}^{n-j+1}$  be its highest weight,
where $H_{[\pm j,n]}$ is the restriction of $H$ to the alphabet $[\pm j,n]$. The restriction of $H$ to the alphabet  $[\pm j,n]$ is a KN skew
tableau on the alphabet $\gls{Corderedletters}$.
The entries of $H$ in $[1,j-1]$ define a semi-standard tableau $ T^+_{[1,j-1]}$ of shape, say $\mu$, and the entries in $[
   \overline{j-1},\bar 1]$ define a
    skew semi-standard tableau $ T^-_{[\overline{j-1},\bar 1]}$ of shape $\lambda/\nu$, where $\mu\subseteq \nu\subseteq \lambda$. Hence
    the cells of $H$ filled in
    $[\pm j,n]=[\pm 1,n]\setminus
   \{1<\cdots<j-1<\overline {j-1}<\cdots<\bar 2<\bar 1\}$  define  the skew shape $\nu/\mu$, and because the crystal operators in $\gls{B}$
   are indexed by $[j,n]$, they do not change the skew shape $\nu/\mu$ either. Therefore, since all the vertices of $\gls{B}$ are connected
   to $H$ through those crystal operators, the vertices of $\gls{B}$ restricted to the alphabet $[\pm j,n]$ have the same skew shape
   $\nu/\mu$
   and the same semi-standard tableaux $ T^+_{[1,j-1]}$ and $ T^-_{[\overline{j-1},\bar 1]}$  \cite[Lemma 6.1.3]{lec02}.\\

\noindent  \textbf{ Step I}. {\bf The  sequence of isomorphic crystals from $T_{[\pm j,n]}$ to its reduced rectification.}
\medskip

 {\noindent \textbf{ I.1} -  {\textsc{ The $C_{n-j+1}$ connected crystal $\gls{B}^0$ containing  $T_{[\pm j,n]}$.}

Erase in the vertices of $\gls{B}$ the entries in $[\pm 1, j-1]$;  that is, erase the semi-standard tableaux $ T^+_{[1,j-1]}$ and $
T^-_{[\overline{j-1},\bar 1]}$. We obtain the connected $C_{n-j+1}$ crystal $\gls{B}^0$ of semi-standard skew  tableaux of shape
$\nu/\mu$
with entries in the alphabet $[\pm j,n]$,
possibly with some non-admissible columns, containing $T_{[\pm j,n]}$. (For short we call to $\mu$ the inner shape of each vertex of $\textsf{B}$ or $\gls{B}^0$.) These KN skew tableaux over $\gls{Corderedletters}$ might have
non-admissible columns over $[\pm j,n]$. More precisely, $\gls{B}^0$ is the  connected crystal of the reading words of the aforesaid semi-standard skew tableaux on the alphabet $[\pm j,n]$, with
highest  weight element the word
 of $H_{[\pm j,n]}$.
  Hence $\gls{B}^0$ and $\gls{B}$  are isomorphic crystals, with a crystal isomorphism given by the reading word map on the alphabet $[\pm j,n]$.


\medskip
{\noindent \bf I.1.1}-  \textsc{The green inner standard tableau $U_0$ for any vertex of $\textsf{B}^0$.}}

 Fix a standard tableau $U_0$ of shape $\mu$ filled in a completely ordered alphabet of \textit{green letters} $\{{{\green g_1} <\cdots<\green
 g_{|\mu|}}\}$ where $|\mu|$ is the number of boxes of $\mu$.
 Assign the inner standard tableau $U_0$ to the inner shape $\mu$ of each vertex of $\textsf{B}^0$. Recall that $T_{[\pm j,n]}$ is the image of $T$
 in $\textsf{B}^0$; see the tableau pair $(U_0,T_{[\pm j,n]})$ schematically depicted in Figure~\ref{fig:1}.

 \begin{figure}[h]
  \centering
$(U_0,T_{[\pm j,n]})=$\[ \vcenter{\hbox{\begin{tikzpicture}
\node at (.25, -.25) {$\color{green}\bullet $};
\node at (1.25, -.25) {$\color{green}\dots$};
\node at (2.25, -.25) {$\color{green}\bullet  $};
\node at (.5, -.5) {$\color{green}\bullet $};
\node at (1.25, -.5) {$\color{green}\bullet$};
\node at (.25, -.75) {$\color{green}\vdots$};
\node at (1, -.75) {$\color{green}\ddots$};
\node at (1.125, -1.25) {$\color{green}\bullet$};
    \draw[line width=1pt]
    (0,0) -- ++(2.5,0) --++(0,-.5) -- ++(-.5,0) --++(0,-.5) --++(-.5,0) --++(0,-.5) --++(-1.5,0) -- cycle;
    \draw[line width=1pt]
    (2.5,0) -- ++(2.5,0) --++(0,-.5) -- ++(-.5,0) --++(0,-.5) --++(-1,0) --++(0,-1) --++(-2.5,0)--++(0,-1) --++(-1,0)--++(0,1.5)--++(1.5,0)
    --++(0,.5) -- ++(.5,0)--++(0,.5) -- ++(.5,0)-- cycle;
    \node at (2.2, -1.5) {$T_{[\pm j,n]}$};
\end{tikzpicture}}}\]
  \caption{$T_{[\pm j,n]}$ in the crystal $\textsf{B}^0$ and the inner tableau $U_0$ in green.\label{fig:1}}
\end{figure}

\medskip

{\noindent \textbf{ I.2} -  \textsc{The  $C_{n-j+1}$-crystal $\gls{B}^x$ of KN skew tableaux $\gls{r3}$ isomorphic to $\textsf{B}^0$}.}

Let $H^0:=H_{[\pm j,n]}$ be the highest weight element of the $C_{n-j+1}$ crystal $\textsf{B}^0$.
The skew tableau
 $H^0$ of shape $\nu/\mu$ may have non-admissible columns on the alphabet $[\pm j,n]$. Let $q< r<\dots<s< t$
be the non-admissible  columns  of $H^0$. Then exactly the same columns in all  vertices of $\textsf{B}^0$ are non-admissible.
The Knuth contraction $\gls{r3}$ relation,  on Subsection \ref{R3}, defines a crystal isomorphism; it commutes with the crystal
operators and preserves the weight. Moreover, each time $\gls{r3}$ is applied to a column  of some vertex of $\textsf{B}^0$, it is also
applied to the same column in every vertex of $\textsf{B}^0$  (see \cite[Proposition 3.2.4, Corollary 3.2.5]{lec02}).

In each vertex of $\textsf{B}^0$, apply the  contraction $\gls{r3}$ relation to  column $i$, for  $i=q,r,\dots,s,t$,
  until column $i$ becomes  admissible. For $i=q,r,\dots,s,t$, each time we apply $\gls{r3}$ to  column $i$, a pair
   of entries $(k,\bar k)$ is erased (whenever $k \in [n]$ is minimal for $N(k)> k$, $k$ and $\bar k$ appear in the column and all prefixes
   are admissible).
   Then the cells from the top and the bottom of the current column $i$ are emptied; the remaining entries are placed in order in the
   remaining cells between those erased.  We obtain a new crystal of KN skew tableaux on the alphabet
   $[\pm j,n]$ isomorphic
   to the crystal  $\textsf{B}^0$.

Let  $x$ be the total number of times $ \gls{r3}$ has to be applied  to $H^0$,  from column $r$ to column $t$ as explained above, to get  a KN
skew
tableau
on alphabet $[\pm j,n]$. Denote the resulting KN skew tableau by $H^x$. Note that for each
column of any vertex of $\textsf{B}^0$,
 the number of times $\gls{r3}$ is applied is the same. We then obtain the sequence of isomorphic crystals

\begin{align*}
&\textsf{B}^0\overset{\gls{r3}}\simeq\textsf{B}^1\overset{\gls{r3}}\simeq\cdots
\overset{\gls{r3}}\simeq\textsf{B}^{x_r}\overset{\gls{r3}}\simeq
\textsf{B}^{x_r+1}\overset{\gls{r3}}\simeq\cdots\overset{\gls{r3}}\simeq\textsf{B}^{x_r+x_{s}} \\
&\overset{\gls{r3}}\simeq \cdots\overset{\gls{r3}}\simeq\textsf{B}^{x_q+x_r+\cdots+x_s+x_t}=\textsf{B}^x,
\end{align*}

\noindent where $x=x_q+x_r+\cdots+x_s+x_t$ and $x_i$ is the number of times we apply $\gls{r3}$  to column $i$ of $H^0$, for
 $i=q,r,\dots,s,t$.
The crystal $\textsf{B}^x$, isomorphic to $\textsf{B}^0$, is obtained by applying  $\gls{r3}$ $x$ times to each vertex of $\textsf{B}^0$,
namely,
 $x_i$ times to column $i$, for $i=q,\dots,s,t$, of each vertex of $\textsf{B}^x$.
Equivalently, $\textsf{B}^x$ is the  crystal whose highest weight element is the KN skew tableau $H^x$ of shape $\nu^x/\mu^x$, where
$\nu^x\subseteq \nu$, $\mu\subseteq \mu^x$ and
$|\mu^x|-|\mu| =|\nu|-|\nu^x|= x$ is the  number of  times $\gls{r3}$ has been applied to  $H^0$ (or $T_{[\pm j,n]}$).

\medskip
{\noindent{\bf I.2.1} -  \textsc{The pair $(U_x,V_x)$ of green-purple inner and purple outer standard tableaux for any vertex of
$\textsf{B}^x$}.}

 Let  \begin{align}\{{\green \textbf{g}_1<\cdots<\green \textbf{g}_{|\mu|}} < {\pur \textbf{p}_1} <{\pur \textbf{p}_2} <\cdots<{\pur
 \textbf{p}_x}< {\pur \textbf{p}_x'}<\cdots<{\pur \textbf{p}'_2}<{\pur
 \textbf{p}_1'}\}\label{coloralphabet}\end{align}
  be a completely ordered alphabet of $|\mu|+2x$ letters consisting of $|\mu|$ green letters and  $x$ unprimed and $x$ primed \textit{purple
  letters}.

Define  the  standard tableau $U_x$ of shape $\mu^x$,  where $\mu\subseteq \mu^x$ and $|\mu^x|=|\mu| + x$, to be an extension of $U_0$
filled with the $|\mu|$ green letters by filling the extra
$x$ cells,  the  total number of cells made empty at the top of each non-admissible column in a vertex of $\textsf{B}^0$,   with  the
unprimed
purple letters $\{ {\pur \textbf{p}_1}<\cdots<{\pur
\textbf{p}_{x_r}}<\cdots<{\pur \textbf{p}_x}\}$. Define the standard tableau
 $V_x$ of shape $\nu/\nu^x$ by filling the $x$ cells  made empty at the bottom of each non-admissible column  in a vertex of
 $\textsf{B}^0$  with the primed purple
 letters ${ {\pur \textbf{p}_x'}<\cdots< {\pur \textbf{p}_{x_r}'}<\cdots<{\pur \textbf{p}_1'}}$. The filling rule is as follows.

 Successively fill the pair of cells made empty each time $\gls{r3}$ is applied with one unprimed purple letter and one primed purple
 letter,
 ${\pur \textbf{p}_1}<{\pur
 \textbf{p}_1'},\dots,
{\pur \textbf{p}_{x_r}}<{\pur \textbf{p}'_{x_r}},{\pur \textbf{p}_{x_r+1}}<{\pur \textbf{p}'_{x_r+1}},\dots,{\pur \textbf{p}_{x_r+x_s}}<{\pur
\textbf{p}'_{x_r+x_s}},\dots,{\pur \textbf{p}_x}<{\pur \textbf{p}_x'}$, with the unprimed letter at the top of the column  and the primed
letter at the bottom of the column.  We impose the order
$${\green \textbf{g}_1}<\cdots<{\green \textbf{g}_{|\mu|}}<{\pur \textbf{p}_1}<\cdots<{\pur \textbf{p}_{x_r}}<{\pur
\textbf{p}_{x_r+1}}<\cdots<{\pur \textbf{p}_{x_r+x_s}}<\cdots<{\pur \textbf{p}_x}<$$
$$<{\pur \textbf{p}'_x}<\cdots<{\pur \textbf{p}'_{x_r+x_s}}<\cdots<{\pur \textbf{p}'_{x_r+1}}<{\pur \textbf{p}_{x_r}'}<\cdots<{\pur
\textbf{p}'_1}.$$
 That is, each time an unprimed purple letter and a primed purple letter are added to $U_x$ and $V_x$, respectively, the unprimed letter is
 strictly larger than any green letter and any unprimed purple letter already added to $U_x$, and simultaneously, the primed purple letter
 is
 strictly smaller than any primed purple letter already added to $V_x$.

By construction, the pair $(U_x,V_x)$ of inner and outer standard tableaux is the same for any vertex of $\textsf{B}^x$.  More
precisely, $U_x$ of
shape $\mu^x$ is the extension of $U_0$ filled with the alphabet $\{{\green \textbf{g}_1}<\cdots<{\green \textbf{g}_{|\mu|}}<{\pur
\textbf{p}_1}<{\pur \textbf{p}_2}<\cdots<{\pur
\textbf{p}_x}\}$; $V_x$ of skew shape $\nu/\nu^x$ is filled with the alphabet of primed purple letters ${\pur
\textbf{p}'_x}<\cdots<{\pur \textbf{p}'_{x_r+\cdots+x_q}} <\cdots<{\pur \textbf{p}'_{x_r+x_s}}<\cdots<{\pur \textbf{p}'_{x_r+1}}<{\pur
\textbf{p}_{x_r}'}<\cdots<{\pur \textbf{p}'_1}.$  Regarding $U_x$, extend the column $r$ of $U_0$ with  the $x_r$
unprimed
purple letters ${\pur \textbf{p}_{1}}<\cdots<{\pur \textbf{p}_{x_r}}$, the column $s$  with  the $x_s$ unprimed purple letters ${\pur
\textbf{p}_{x_r+1}}<\cdots<{\pur
\textbf{p}_{x_r+x_s}}$, and finally the column $t$ with the $x_t$ unprimed purple letters ${\pur \textbf{p}_{x_r+\cdots+x_q+1}}<\cdots<{\pur
\textbf{p}_{x_r+\cdots+x_q+x_t}}={\pur \textbf{p}_x}$; regarding $V_x$ of skew shape $\nu/\nu^x$, start with  the skew shape $ \nu/\mu^x$, and
fill   the
bottom
$x_r$ boxes of column $r$ with the alphabet of primed purple letters ${\pur \textbf{p}_{x_r}'}<\cdots<{\pur \textbf{p}'_1}$, the bottom $x_s$
of column $s$ with
the alphabet ${\pur \textbf{p}'_{x_r+x_s}}<\cdots<{\pur \textbf{p}'_{x_r+1}}$, and, finally, the bottom $x_t$ boxes of column $t$ with the
alphabet ${\pur
\textbf{p}'_x}<\cdots<{\pur \textbf{p}'_{x_r+x_s+\cdots+x_{q}+1}}$. See the triple $(U_x,H^x,V_x)$ in Figure~\ref{fig:colorful}.

\begin{figure}[h]
  \centering
\[ \vcenter{\hbox{\begin{tikzpicture}
\path[line width=1pt, fill=green]
(0,0) -- ++(3.5,0) --++(0,-.5) -- ++(-1,0) --++(0,-.5) --++(-1,0) --++(0,-.5) --++(-1,0) --++(0,-2) -- ++(-.5,0)  -- cycle;
\path[line width=1pt, fill=violet]
(2.5,-.5) -- ++(.5,0) --++(0,-.5) -- ++(-.5,0)  -- cycle;
\path[line width=1pt, fill=violet]
(.5,-1.5) -- ++(1,0) --++(0,-1) -- ++(-.5,0) --++(0,-.5) -- ++(-.5,0)-- cycle;
\path[line width=1pt, fill=lightgray]
(0,-3.5) -- ++(.5,0) --++(0,.5) -- ++(.5,0) --++(0,.5) -- ++(.5,0)-- ++(0,1.5)-- ++(1.5,0)--++(0,.5) -- ++(.5,0)--++(0,-1.5) --
++(-1,0)--++(0,-1.5) -- ++(-1,0) --++(0,-1)-- ++(-1,0)--++(0,-2)-- ++(-.5,0) --++(0,3)-- cycle;

\path[line width=1pt, fill=violet]
(.5,-4.5) -- ++(1,0) --++(0,-1) -- ++(-.5,0) --++(0,-.5) -- ++(-.5,0)-- cycle;
\path[line width=1pt, fill=violet]
(2.5,-2) -- ++(.5,0) --++(0,-.5) -- ++(-.5,0)  -- cycle;
\end{tikzpicture}}}\]
  \caption{The triple $(U_x, H^x, V_x)$ with $H^x$ in gray, $V_x$ in purple, and $U_0(\subseteq U_x)$ in green. 
$U_x$ consists of the green region together with the purple inner regions of $ V_x$.\label{fig:colorful}}
\end{figure}
 \medskip
{\noindent\noindent{\bf I.3} -  \textsc{Rectification of the $C_{n-j+1}$ crystal $B^x$ and reduced rectification of $T_{[\pm j,n]}$}}

Consider the triple of tableaux $(U_x,H^x, V_x)$ previously defined. Apply complete $SJDT_j$ slides successively to the cells of $U_x$, from
the largest entry to the smallest one, to rectify
$H^x$.  At the end of each complete $SJDT_j$ slide,  we get an outer cell filled with the  letter where the slide started in $U_x$. While
$H^x$ is being rectified, the cells of $U_x$ are slid to end up as outer corners and added to the skew standard tableau $V_x$.

The rectification of $H^x$ does not depend on the choice of the inner corner made in each step during the rectification process
\cite[Corollary 6.3.9]{lec02}. Applying $SJDT_j$ to any corner of $U_x$ in an element of $\textsf{B}^x$ (recall that for all elements of
$\textsf{B}^x$, $U_x$ is the same) gives a crystal isomorphism. This observation is equivalent to the fact that the rectification does not
depend on the filling of $U_x$: $U_x$ is a choice to keep track of the rectification of $H^x$ (or of any other vertex of $\textsf{B}^x$). If
we may apply a complete $SJDT_j$ slide to an inner corner of $H^x$, then we may apply a complete $SJDT_j$ slide to the same inner corner in
every vertex of the crystal $\textsf{B}^x$, and this slide will create the same outer corner filled with the same letter.

 However, if the number of boxes of $H^x$, $|H^x|$,  exceeds  the minimal number  of boxes of its Knuth class, it will be necessary to apply
 $SJDT_j$ more than $|U_x|=|\mu|+x$ times to rectify $H^x$. When $H^x$ has the minimal number of boxes of its Knuth class,  only $x$
 unprimed
 purple letters and $|\mu|$ green letters will slide outwards and join the outer tableau $V_x$.

Let  $2y\ge 0$ be the number of boxes of  $H^x$ that exceeds the minimal number of boxes of its Knuth class, that is,
 $$2y=|H^x|-| \textsf{rect}_j(H^x)|.$$
Necessarily $2y$  boxes of $H^x$ will be lost in the $SJDT_j$ rectification process. Henceforth, the $SJDT_j$ B.2 case will be applied $y$
times, each of which will create a non-admissible column followed by the application of an $\gls{r3}$ contraction relation, each of which will
result in the loss of two boxes.

\begin{remark} Theorem 6.1.9 in Lecouvey's paper \cite{lec02} says: if the B.2 case appears with the creation of a non-admissible column when
applying complete SJDT to an inner corner of a KN skew tableau, it has to be at the initial column where the inner corner was originally
contained. \end{remark}

This observation implies that each of the $y$ mentioned non-admissible columns will only occur in the columns containing the inner corners
where the slide started.

The complete $SJDT_j$ slides applied successively to the entries of $U_x$, as mentioned, will  transform  the crystal $\textsf{B}^x$
into an isomorphic crystal of KN skew tableaux, as long as the $SJDT_j$ B.2 case   does not create a non-admissible column.
Otherwise, one has an isomorphic crystal where each vertex has  a non-admissible column. In this case, we apply the contraction relation
$\gls{r3}$ to that column in each vertex,  erasing a pair $(k,\bar k)$ if  $k \in [\pm j,n]$ is the lowest entry such that $N(k)>k$ .
Then, as in $\bf I.2$ above, the cells from the top and the bottom of the current column are emptied and the remaining entries
are placed in order. We get a new isomorphic crystal of KN skew tableaux where each vertex has two fewer boxes.  As observed above, this may
only happen in the $y$ columns where $SJDT_j$ was applied, specifically, those containing the inner corners where the slides started; no
other
boxes are deleted in the rectification process of $\textsf{B}^x$.

Eventually, $H^x$ is rectified to $\textsf{rect} (H^x)$, as are all vertices of  $\textsf{B}^x$, and  we get $\textsf{B}^0\simeq \textsf{B}^x
\simeq \textsf{R}$, where the crystal $\textsf{R}$ is of straight KN tableaux with highest weight element $\textsf{rect} (H^x)$.
 \medskip

 {\noindent {\bf I.3.1} - \textsc{The green-purple-red standard tableau $V$ of
every vertex in the $C_{n-j+1}$ crystal $\textsf{R}$ containing $
\textsf{rect}_j(T_{[\pm j,n]})$.}}

  Let $$\textsf{B}^{x,1},\textsf{B}^{x,1,\text{-}},\textsf{B}^{x,2},\textsf{B}^{x,2,\text{-}}\dots, \textsf{B}^{x,y},\textsf{B}^{x,y,\text{-}}$$ be the sequence
  of
  $2y$ isomorphic crystals appearing in the  rectification process  from  $\textsf{B}^{x}$ to $\textsf{R}$, tracking each complete $SJDT_j$
  slide which
  triggers a  B.2 case,   and the subsequent application of an $\gls{r3}$ contraction relation to that non-admissible column. Namely, $\textsf{B}^{x,i}$  correspond to the B.2 case and $\textsf{B}^{x,i,\text{-}}$  to the $\gls{r3}$  contraction relation.

  For $i=1,\dots,y$, let ${H}^{x,i}$ and ${H}^{x,i,\text{-}}$ be the  pair of  highest weight elements of the crystal pair
  $\textsf{B}^{x,i}$ and $\textsf{B}^{x,i,\text{-}}$, respectively. Each ${H}^{x,i}$ has exactly one non-admissible column, and ${H}^{x,i,\text{-}}$ has no
  non-admissible columns.

We have to store $2y$ new auxiliary letters to record the $2y$ empty cells created by the $y$ applications of $\gls{r3}$
as a consequence of the creation of $y$ non-admissible columns by the complete $SJDT_j$ slide where the B.2 case appeared and created a
non-admissible column.

Consider the triple of tableaux $(U_x,H^x, V_x)$ corresponding to  the crystal $\textsf{B}^x$. Let  $(U_{x,1},$ $H^{x,1},$ $V_{x,1})$ be the
triple of tableaux  obtained from $(U_x,H^x, V_x)$ by applying  complete $SJDT_j$ slides to the entries of $U_x$ and transforming the KN  skew
tableau  $H^x$ into $H^{x,1}$, where for the first time in the complete $SJDT_j$ slide, the B.2 case appears and creates a non-admissible
column. After the said complete $SJDT_j$ slides to $U_x$, $U_{x,1}$ is the inner standard tableau of $H^{x,1}$, and $V_{x,1}$ is obtained from
$V_x$ by adding the slid entries from $U_x$ to $V_x$.  $V_{x,1}$ is indeed a standard tableau because by construction  the green entries of
$U_x$ are strictly smaller than the primed purple entries of $V_x$. The pair $(U_{x,1},V_{x,1})$ of inner and outer standard tableaux is the
same for every vertex of $\textsf{B}^{x,1}$: $U_{x,1}\subseteq U_x,\quad V_{x,1}\supseteq V_x.$

 We have to apply an $\gls{r3}$ contraction relation  to $H^{x,1}$ (and to every vertex of $\textsf{B}^{x,1}$) to transform the non-admissible
 column into an admissible one: a pair of symmetric entries in each vertex of $\textsf{B}^{x,1}$ will be deleted, the top and bottom cells
 of
 that
 column will be emptied and the remaining entries will be placed in order. Fill the empty entries with \textit{red letters} ${\re
 \textbf{r}_1} <{\re
 \textbf{r}'_1}$,
 with $\re \textbf{r}_1$ on the top and $\re \textbf{r}'_1$ on the bottom, where in the complete  $SJDT_j$ slide, the $\text{B}.2$ case
 appears and has created
 a non-admissible column such that $ \re \textbf{r}_1$ is strictly larger than any entry of $U_{x,1}$, and $\re \textbf{r}'_1$ is strictly
 smaller than any
 entry of $V_{x,1}$. $V_{x,1}$  is filled with the entries of $U_x$ already slid  and with all primed purple letters. The cell with the red
 letter $\re \textbf{r}_1$ was the cell of $U_x$ where the complete $SJDT_j$ slide started and the B.2 case appeared with the
 creation of a non-admissible column.

Let  $U_{x,1,+}$  be the standard tableau obtained by adding the red letter $\re \textbf{r}_1$ to  $U_{x,1}$, and let $V_{x,1,+}$ be the
standard
tableau obtained by adding the primed red letter $\re \textbf{r}'_1$ to $V_{x,1}$ in the manner described,
$$U_{x,1}\subset U_{x,1,+}\subseteq U_x,\quad V_{x,1,+}\supset V_{x,1}\supseteq V_x.$$

We keep applying complete $SJDT_j$ slides to entries of $U_{x,1,+}$, from the largest to the smallest, to rectify $H^{x,1,-}$, so the cell
$\re \textbf{r}_1$ will be the first to slide outwards and become an outer corner.

 Let  $(U_{x,2},H^{x,2},V_{x,2})$ be the triple of tableaux obtained from $(U_{x,1,+},H^{x,1,-}, V_{x,1,+})$ by applying complete $SJDT_j$
 slides to the  entries of $U_{x,1,+}$ and transforming the KN  skew tableau  $H^{x,1,-}$ into  $H^{x,2}$, where for the second time  in
 the
 complete $SJDT_j$ slide, the B.2 case appears with the creation of a non-admissible column; that is, $H^{x,2}$ has a non-admissible column,
 and the
 highest weight elements of all previous crystals obtained from $\textsf{B}^{x,1,-}$ had all columns admissible. After these complete
 $SJDT_j$
 slides to $U_{x,1,+}$, $U_{x,2}$ is the inner standard tableau of $H^{x,2}$; $V_{x,2}$ is obtained from $V_{x,1,+}$ by adding the slid
 entries from $U_{x,1,+}$ to
 $V_{x,1,+}$. $V_{x,2}$ is indeed a standard tableau because by construction the entries of $U_{x,1,+}$ are strictly smaller than the
 entries
 of $V_{x,1,+}$. At this point, the red letter $\re \textbf{r}_1$ has already slid from $U_{x,1,+}$ to
 $V_{x,2}$; that is, $\re \textbf{r}_1$ is no longer in $ U_{x,2}$ and instead belongs to $V_{x,2}$,
 $$U_{x,2}\subseteq U_{x,1}\subset U_{x,1,+}\subseteq U_x,\quad V_{x,2}\supset V_{x,1,+}\supset V_{x,1}\supseteq V_x$$

 Continuing in this fashion for $i=2,\dots,y-1$, we eventually reach a point wherein the $2(y-1)$ red letters have the following relative
 ordering,
 $${\re \textbf{r}_{y-1}}<{\re \textbf{r}'_{y-1}}<\cdots<{\re \textbf{r}_2}<{\re \textbf{r}'_2}<\cdots<{\re \textbf{r}_1}<{\re
 \textbf{r}'_1}$$
and the rectification storing tableaux are such that
$$ U_{x,y}\subset  U_{x,y-1}\subset U_{{x,y-1}+}\subset \cdots\subset  U_{x,2}\subset U_{x,2,+}\subseteq U_{x,1}\subset
U_{x,1,+}\subseteq
U_x,$$
$$ V_{x,y}\supset V_{x,y-1,+}\supset V_{x,y-1}\supset\cdots \supset  V_{x,2,+}\supset V_{x,2}\supset
V_{x,1,+}\supset
V_{x,1}\supseteq V_x.$$

 \medskip
Let  $\textsf{B}^{x,y}$ be the crystal  with highest weight element $H^{x,y}$.
We have to apply the $\gls{r3}$ contractor operator to $H^{x,y}$ (and to every vertex of $\textsf{B}^{x,y}$) to transform the non-admissible
column into an admissible one: a pair of symmetric entries in each vertex of $\textsf{B}^{x,y}$ will be deleted,  the top and bottom cells
of that column will be emptied and the remaining entries will be placed in order. Let $\textsf{B}^{x,y,-}$ be the new crystal of KN skew
tableaux
isomorphic to $\textsf{B}^{x,y}$, and let $H^{x,y,-}$ be its highest weight element (it has two  fewer boxes than $H^{x,y}$). Fill the empty
entries with
red letters ${\re \textbf{r}_y}<{\re \textbf{r}'_y}$, $\re \textbf{r}_y$ on the top and $\re \textbf{r}'_y$ on the bottom  of the column,
where in the complete  $SJDT_j$ slide,
the
$B.2$ case appears and creates a non-admissible column such that $ \re \textbf{r}_y$ is strictly larger than any entry of $U_{x,y}$, and $\re
\textbf{r}'_y$
is strictly smaller than any entry of $U_{x,y}$ already slid. The primed letters are considered to be slid because by the time they are
 created, they are outer corners.

 The cell with the red letter $\re \textbf{r}_y$ was the cell of $U_{x,y-1,+}$ where the complete $SJDT_j$ slide started and the B.2 case
 appeared with
 the creation of a non-admissible column. Let  $U_{x,y,+}$ be the standard tableau obtained by adding the red letter $\re \textbf{r}_y$ to
 $U_{x,y}$,
 and let $V_{x,y,+}$  be the standard tableau obtained by adding the primed red letter $\re \textbf{r}'_y$ to $V_{x,y}$.

 We keep applying  complete $SJDT_j$ slides to the entries of
 $U_{x,y,+}$
 from the largest  to the smallest, and eventually, we rectify $H^{x,y,-}$ without further use of the contractor $\gls{r3}$. We reach
 the
 crystal
 $\textsf{R}$, where every vertex is rectified. The crystal $\textsf{R}$ is called the rectification of $\textsf{B}^0$. At this point, one has
 the following relative ordering among the $2y$ red letters:
 $${\re \textbf{r}_y}<{\re \textbf{r}'_y}<\cdots<{\re \textbf{r}_2}<{\re \textbf{r}'_2}<\cdots<{\re \textbf{r}_1}<{\re \textbf{r}'_1}$$
and the rectification storing tableaux
$$\emptyset\subset U_{x,y}\subset U_{x,y,+}\subseteq U_{x,y-1}\subset \cdots\subset  U_{x,2}\subset U_{x,2,+}\subseteq U_{x,1}\subset
U_{x,1,+}\subseteq
U_x,$$
$$V\supset V_{x,y,+}\supset V_{x,y}\supset V_{x,y-1,+}\supset V_{x,y-1}\supset\cdots \supset  V_{x,2,+}\supset V_{x,2}\supset
V_{x,1,+}\supset
V_{x,1}\supseteq V_x,$$
where $V$ is the standard tableau obtained  by adding to $V_{x,y,+}$ via sliding the  letters from $U_{x,y,+}$.
We have the following ordering of all colored letters, green, purple (primed and unprimed),  and  red (primed and unprimed) in the skew
standard tableau $V$:
\begin{align}&{\green \mathbf{g}_1}<\cdots<{\re \textbf{r}_y}<{\re \textbf{r}'_y}<{\green \textbf{g}_l}<\cdots<{\re \textbf{r}_d}<{\re
\textbf{r}'_d}<\cdots<{\green \textbf{g}_{|\mu|}}<\nonumber\\
&<{\pur \textbf{p}_1}<{\pur \textbf{p}_2}<\cdots<{\re \textbf{r}_k}<{\re \textbf{r}'_k}
<\cdots<{\pur \textbf{p}_i}<\cdots<{\re \textbf{r}_1}<{\re \textbf{r}'_1}<\cdots<{\pur \textbf{p}_x}< {\pur \textbf{p}_x'}<\cdots<{\pur
\textbf{p}_1'}.\label{coloralphabetall}\end{align}

 We have constructed the following sequence of isomorphic crystals, stored in $V$ via the slid colorful letters:
 \begin{eqnarray}\label{sequence1}\textsf{B}^0\overset{\pur \gls{r3}}\simeq\cdots \textsf{B}^{x_r} \overset{\pur \gls{r3}}\simeq\cdots
 \textsf{B}^{x_r+x_{s}}\overset{\pur \gls{r3}}\simeq \cdots \textsf{B}^{x_r+x_s+\cdots+x_t}=\textsf{B}^{x}
  \end{eqnarray}
 \begin{eqnarray}\label{sequence2}
 \textsf{B}^{x}\overset{{ SJDT_j}}\simeq\cdots \textsf{B}^{x,1}\overset{\re \gls{r3}}\simeq \textsf{B}^{x,1,-}\overset{
 SJDT_j}\simeq
 \cdots
 \textsf{B}^{x,2}\overset{\re \gls{r3}}\simeq \textsf{B}^{x,2,-}\overset{SJDT_j}\simeq \cdots \textsf{B}^{x,y}\overset{\re
 \gls{r3}}\simeq
 \textsf{B}^{x,y,-}
  \end{eqnarray}
\begin{eqnarray}\label{sequence3}\textsf{B}^{x,y,-}\overset{SJDT_j}\simeq\cdots \overset{ SJDT_j}\simeq \textsf{R}.
 \end{eqnarray}

\begin{remark} In our construction, purple letters are larger than all green ones \eqref{coloralphabet}. However, for the red ones together
with the two other colors, we just write \eqref{coloralphabetall}.
\end{remark}

{\noindent  {\bf I.4} -  \textsc{The Sch\"utzenberger--Lusztig involution on the $C_{n-j+1}$ crystal $\textsf{B}^0$, its rectification and the
reversal.}}

Let ${L}^0$ be the lowest weight element of the $C_{n-j+1}$ connected normal crystal $\textsf{B}^0$. The crystal
$\textsf{R}$ with highest weight element $\textsf{rect}_j(H^0)$ is the rectification of the crystal $\textsf{B}^0$  and contains
$\textsf{rect}_j(T_{[\pm j,n]})$.
Let
$F$ be the composition of the  sequence of lowering operators connecting $H^0$ to $T_{[\pm j,n]}$ in $\textsf{B}^0$, $F(H^0)=T_{[\pm j,n]}$.
The Sch\"utzenberger--Lusztig involution $\gls{xi}$ in $\textsf{B}^0$ gives $\gls{xi}(T_{[\pm j,n]})=F^{-1}(L^0)$, where $F^{-1}$ means the
sequence
obtained by replacing each lowering operator $f_i$ in $F$ with the corresponding raising operator $e_i$. In each crystal of the sequence
\eqref{sequence1}, \eqref{sequence2}, \eqref{sequence3} above, the same sequence $F$ ($F^{-1}$) connects  the corresponding highest (lowest)
weight element  to the corresponding coplactic image of $T_{[\pm j,n]}$ ($\gls{xi}(T_{[\pm j,n]}$)). In particular, $F$ connects
$\textsf{rect}_j(H^0)$ to $\textsf{rect}_j(T_{[\pm j,n]})$: $F(\textsf{rect}_j
H^0)=\textsf{rect}_j(T_{[\pm
j,n]})$.
By Santos \cite{sa21a}, the Sch\"utzenberger--Lusztig involution in $\textsf{R}$  guarantees that
$\textsf{evac}^{ C_{n-j+1} }(\textsf{rect}_j(T_{[\pm j,n]}))=F^{-1}(\textsf{rect}_j(L^0)$ is in \textsf{R}.
Thanks to the crystal isomorphisms \eqref{sequence1}, \eqref{sequence2}, \eqref{sequence3} and Lemma \ref{reversal},
\begin{align}\label{formula0}
\textsf{\textsf{reversal}}^{C_{n-j+1}}(T_{[\pm
j,n]})=F^{-1}(L^0)=\textsf{arect}_j\,\textsf{evac}^{C_{n-j+1}}(\textsf{rect}_j(T_{[\pm
j,n]})).
\end{align}
To compute  the reversal of $T_{[\pm j,n]}$ in $\textsf{B}^0$ without using the sequence $F$ of crystal operators and the highest/lowest
weight elements $H^0$,
$L_0$ of $\textsf{B}^0$, we use Santos' evacuation on $\textsf{rect}_j(T_{[\pm j,n]})$ and the rectification sequence of crystals backwards
in
\eqref{sequence1}, \eqref{sequence2},\eqref{sequence3} stored in the standard skew tableau $V$.

\medskip
\noindent \textbf{Step II}. \textsc{ Computation of symplectic evacuation of
$\textsf{rect}_j(T_{[\pm j,n]})$ in the ${C_{n-j+1}}$ crystal  $R$.}

The tableau $\textsf{rect}_j(T_{[\pm j,n]})$ is admissible in the alphabet $[\pm j,n]$. Use Santos' algorithm  as follows: take
$\pi$-rotation
and change the sign of $\textsf{rect}_j (T_{[\pm j,n]})$; then, apply $SJDT_j$ to obtain
$\textsf{evac}^{C_{n-j+1}}(\textsf{rect}_j(T_{[\pm j,n]}))$ in
the crystal $\textsf{R}$. Replace the tableau pair  $(\textsf{rect}_j(T_{[\pm j,n]}), V)$ with
$(\textsf{evac}^{C_{n-j+1}}(\textsf{rect}_j(T_{[\pm j,n]})), V).$

\medskip
\noindent \textbf{Step III}. \textsc{Symplectic reversal of $T_{[\pm j,n]}$ in the ${C_{n-j+1}}$ crystal $\textsf{B}^0$}.}

Consider the pair of tableaux $ (\textsf{evac}^{C_{n-j+1}}(\textsf{rect}_j(T_{[\pm j,n]})), V)$, where $V$ is the standard tableau
consisting of all the slid letters in the rectification sequence \eqref{sequence1}, \eqref{sequence2},\eqref{sequence3} on the alphabet of
green, purple and red letters.

Apply the \emph{reverse} $SJDT_j$, $RSJDT_j$, to the entries of $V$ from the smallest to the largest to send $
\textsf{evac}^{C_{n-j+1}}(\textsf{rect}_j(T_{[\pm j,n]}))$ to
$\gls{rev}(T_{[\pm j,n]})=
F^{-1}(L_0)$ in the ${C_{n-j+1}}$ crystal $\textsf{B}^0$.

When the $SJDT_j$ is applied to an unprimed red letter $\re \textbf{r}_i$, $i\in\{1,\dots,y\}$, in $V$, the letter $\re \textbf{r}_i$ slides
to the top of a
column
with the cell $\re \textbf{r}'_i$ on the bottom. At this point, we have reached the crystal $\textsf{B}^{x,i}$. Then we apply the  operator
$\gls{r3}$ to the column containing the pair $({\re \textbf{r}_i},\re \textbf{r}'_i)$ by erasing those entries and adding a pair of symmetric
entries $(k,\bar
k)$
so that we get a non-admissible column on the alphabet $[\pm j,n]$. The $SJDT_j$ applies now to the next  letter bigger  than $\re
\textbf{r}'_i$.
In
this complete reverse slide, the $SJDT_j$ B.2 case  occurs.

When the reverse $SJDT_j$ slides have been applied to all non-primed purple letters, we have reached the crystal $\textsf{B}^x$, where
columns $r<s<\dots<t$ have $x_i$ non-primed  purple  letters ${\pur \textbf{p}_{x_1+\cdots+x_{i-1}+1}}<\cdots<\pur
\textbf{p}_{x_1+\cdots+x_i}$ on the top and the
corresponding primed letters on the bottom for $i=r,s,\dots,t$. Then, for $i=r,\dots,s,t$, we apply $\gls{r3}$ $x_i$ times to each
such column $i$,  and we reach the crystal $\textsf{B}^0$, where each vertex has non-admissible columns $r,s,\dots, t$. In
particular, we obtain $\textsf{reversal}^{C_{n-j+1}}(T_{[\pm j,n]})$.

\medskip
\noindent\textbf{Step IV}. \textsc{Partial symplectic reversal of $T$ computes  $ \xi^{C_n}_{[j,n]}(T)$.}

{ $ \xi^{C_n}_{[j,n]}(T)$ is the Sch\"utzenberger--Lusztig  involution of $T=(T^+_{[j-1]},T_{[\pm j,n]},T^-_{[\overline{j-1},\bar 1]})$
in the crystal connected component $ \textsf{B}\simeq\textsf{B}^0$ of $\textsf{KN}_{[j,n]}(\lambda,n)$.}
 Replace $T_{[\pm j,n]}$ with $\textsf{reversal}^{C_{n-j+1}}(T_{[\pm j,n]})$, \eqref{formula0}, in $T$, which gives the formula
\begin{align}\label{formula}
  \xi_{[j,n]}^{C_n}(T)= \textsf{reversal}^{C_n}_{[j,n]}(T)=(
  T^+_{[j-1]},\textsf{arect}_j\textsf{evac}^{C_{n-j+1}}(\textsf{rect}_j(T_{[\pm
j,n]})), T^-_{[\overline{j-1},\bar 1]}).
\end{align}

\begin{remark} Some consequences of our colorful algorithm.

 \begin{itemize}

 \item If we put $j=1$ in the colorful algorithm, we reduce to Step II of $C_n$ evacuation.


\item \label{re:fullreversalskew}The algorithm for the full $C_{m}$ reversal of a KN skew tableau $T\in \textsf{KN}(\lambda/\mu,m)$ results from our colorful tableau
    switching algorithm by considering the image  of $T$, $( T_\mu,\hat T)$, in the  sub-crystal
    $\textsf{B}(\mu,\lambda)\subseteq
    \textsf{KN}_{[j,n]}(\lambda,n)$, where   $n=m+j-1$ (Subsection \ref{susub:fullskew}). Let $\gls{B}$ be the crystal connected
    component of $\textsf{B}(\mu,\lambda)$
    containing $( T_\mu,\hat T)$, where $ T_\mu$ is the Yamanouchi tableau of shape $\mu$ and $\hat T$ is obtained by increasing
    each of the entries
    of $T$ by $j-1$. Then, restricting $( T_\mu,\hat T)$ to the alphabet $[\pm j,n]$, $\hat T$ is an admissible $C_{n-j+1}$ skew
    tableau in the
    $C_{n-j+1}$ crystal $\textsf{B}^{\,0}$. Our algorithm reduces to Step I with just green and red, Step II and Step III. Finally, we
    subtract
    $j-1$ from the entries of $ \textsf{reversal}^{\,C_{n-j+1}}(\hat T)$ to get $\textsf{reversal}^{\,C_{m}}(T)$. However, subtraction by
    $j-1$
     cancels the last step in the reduced $SJDT_j$ (Definition \ref{def:reduced}), and therefore it is enough to apply
    SJDT.

    This means that the algorithm for the full $C_{m}$ reversal of the KN skew tableau $T$ results from our algorithm with
    $\textsf{B}^{\,0}=\textsf{B}(T)$ a type $C_m$ crystal, $x=0$ and applying SJDT to $U_0$ to get $(\textsf{rect}(T),V)$, where $V$ is a skew
    standard
    tableau without purple letters.  Then RSJDT applied to $V$ gives $$\textsf{reversal}^{\,C_{m}}(
    T)=\textsf{arect}\,\textsf{evac}^{C_{m}}(\textsf{rect}(T)).$$ 
\end{itemize}

\end{remark}

\subsection{Examples of full and partial symplectic reversal}\label{ex:partialsymplecticreversal}
\begin{enumerate}
\item Full reversal of a skew tableau, $J = I$.
 Let $$T=\YT{0.17in}{}{
			{{},{2},{\overline{2}}, {\overline{1}} },
			{{\overline{2}},{\overline{2}},{\overline{1}} },
			{{\overline{1}}}}\in \textsf{KN}((4,3,2)/(1),3).$$ We compute $\xi^{C_3} (T)$ as follows. First, we fill in the empty box in $T$
with
a green letter (it defines the one box standard tableau $U_0$), to which we perform symplectic jeu de taquin until it becomes an outer
corner.
\begin{align*}
    (U_0,T)  &=\quad \;  \YT{0.17in}{}{
			{\color{green} g,{2},{\overline{2}}, {\overline{1}} },
			{{\overline{2}},{\overline{2}},{\overline{1}} },
			{{\overline{1}} }}
\overset{\textsf{SJDT}}\rightarrow	
\YT{0.17in}{}{
			{{1},\color{green} g,{\overline{2}}, {\overline{1}} },
			{{\overline{2}},{\overline{1}},{\overline{1}} },
			{{\overline{1}}, }}
\overset{\textsf{SJDT}}\rightarrow
\YT{0.17in}{}{
			{\mt r,\color{green} g,{\overline{2}}, {\overline{1}} },
			{{\overline{2}},{\overline{1}},{\overline{1}} },
			{\mt r' }}
\overset{\textsf{SJDT}}\rightarrow
\YT{0.17in}{}{
			{\mt r,{\overline{2}},\color{green}g, {\overline{1}} },
			{{\overline{2}},{\overline{1}},{\overline{1}} },
			{\mt r' }}
\\
&\overset{\textsf{SJDT}}\rightarrow
\YT{0.17in}{}{
			{\mt r,{\overline{2}},{\overline{1}}, {\overline{1}} },
			{{\overline{2}},{\overline{1}},\color{green}g  },
			{\mt r' }}
\overset{\textsf{SJDT}}\rightarrow
\YT{0.17in}{}{
			{{\overline{2}},{\overline{2}},{\overline{1}}, {\overline{1}} },
			{\mt r,{\overline{1}},\color{green}g },
			{\mt r'}}
\overset{\textsf{SJDT}}\rightarrow
\YT{0.17in}{}{
			{{\overline{2}},{\overline{2}},{\overline{1}}, {\overline{1}} },
			{{\overline{1}},\mt r,\color{green}g },
			{\mt r'}}\\
&\Rightarrow
\textsf{rect}(T)=\YT{0.17in}{}{
			{{\overline{2}},{\overline{2}},{\overline{1}}, {\overline{1}} },
			{{\overline{1}} }},\quad
			V=\YT{0.17in}{}{
			{{},{},{}, {} },
			{{},\mt r,\color{green}g },
			{\mt r'}},\;\;{\mt r}< {\mt r'}< \green g 
            \end{align*}
Taking $\pi$-rotation and changing the signs of $\textsf{rect}(T)$, we again apply SJDT to compute $\textsf{evac}^{C_3} (\textsf{rect}(T))$:
\begin{align*}
\YT{0.17in}{}{
			{\ast,\ast,\ast, 1 },
			{1,1,2, 2}}
&\overset{\textsf{SJDT}}\rightarrow
\YT{0.17in}{}{
			{\ast,\ast,1, \ast },
			{1,1,2,2}}
\overset{\textsf{SJDT}}\rightarrow
\YT{0.17in}{}{
			{\ast,\ast,1, 2 },
			{1,1,2, \ast}}
\overset{\textsf{SJDT}}\rightarrow
\YT{0.17in}{}{
			{\ast,1,1,2 },
			{1,\ast,2, \ast}} \\           &\overset{\textsf{SJDT}}\rightarrow
\YT{0.17in}{}{
			{\ast,1,1, 1 },
			{1,2,\ast, \ast}}
\overset{\textsf{SJDT}}\rightarrow
\YT{0.17in}{}{
			{1,1,1, 2 },
			{\ast,2, \ast,\ast}}\\
&\overset{\textsf{SJDT}}\rightarrow
\YT{0.17in}{}{
			{1,1,1, 2 },
			{2,\ast, \ast,\ast}} 
            = \textsf{evac}^{C_3} (\textsf{rect}(T)).
\end{align*}
We replace $\textsf{rect}(T)$ with $\textsf{evac}^{C_3} (\textsf{rect}(T))$ in $(\textsf{rect}(T),V)$ and apply reverse SJDT to $V$ to
compute
$\xi^{C_3}(T)=\textsf{reversal}^{C_3}(T)$:
\begin{align*}( \textsf{evac}^{C_3} (\textsf{rect}(T)),V)&=\;\;\;\;\YT{0.17in}{}{
			{1,1,1, 2 },
			{2,\mt r, \color{green}g },
			{ \mt r' }}
			\overset{\textsf{RSJDT}}\rightarrow
			\YT{0.17in}{}{
			{1,1,1, 2 },
			{\mt r,2, \color{green}g },
			{ \mt r' }}
			\overset{\textsf{RSJDT}}\rightarrow
			\YT{0.17in}{}{
			{\mt r,1,1, 2 },
			{1,2, \color{green}g },
			{ \mt r', }}\\
			&\overset{\gls{r3}}\sim \;\;\;\,\YT{0.17in}{}{
			{ 1,1,1, 2 },
			{2,2, \color{green}g },
			{ \overline{2} }}
\overset{\textsf{RSJDT}}\rightarrow
			\YT{0.17in}{}{
			{ 1,1,1, 2 },
			{2, \color{green}g,2 },
			{ \overline{2} }}
	\overset{\textsf{RSJDT}}\rightarrow
			\YT{0.17in}{}{
			{ 1,1,1, 2 },
			{ \color{green}g,2,2 },
			{ \overline{2} }}\\
			&\overset{\textsf{RSJDT}}\rightarrow
			\YT{0.17in}{}{
			{ \color{green}g,1,1, 2 },
			{ 1,2,2 },
			{ \overline{2} }}
            =(U_0,\xi^{C_3}(T))\\
			&\Rightarrow  \xi^{C_3}(T)= \YT{0.17in}{}{
			{ ,1,1, 2 },
			{ 1,2,2 },
			{ \overline{2} }}.
			\end{align*}


\item
 Let $P = \YT{0.17in}{}{
			{1,2 ,2, \bar 1},
			{ 4,4,\overline 3},
			{\overline 4,\bar 2,  \bar 1},
{\bar 3}}\in \textsf{KN}((4,3,3,1),4)$. We have $\textsf{wt}(P)=(-1,1,-2,1)$.  To compute
$\xi_{[2,4]}^{C_4}(P)=\textsf{reversal}_{[2,4]}^{C_4}(P)$, we freeze the letters $1,\bar 1$ in $P$ and consider $P_{[\pm 2,4]}$, which is
not
an admissible $C_3$ tableau in the alphabet $[\pm 2,4]$: the second column $2 4 \bar 2$ is not an admissible $C_3$ column. The column
reading
of $P_{[\pm 2,4]}$ is  $2\bar 324\bar 24\bar 4\bar 3 \overset{\gls{r3}}\sim 2\bar 344\bar 4\bar 3$. We include this non-admissible second
column
in the $SJDT_2$ sequence to rectify $P_{[\pm 2,4]}$.
\medskip

\textbf{1.} Rectification of $P_{[\pm 2,4]}$:
\begin{align*}
    (U_0,P_{[\pm 2,4]})&=\;\;\;\;\;\YT{0.17in}{}{
			{\color{green} g,2 ,2,},
			{ 4,4,\overline 3},
			{\overline 4,\bar 2,},
{\bar 3}}
\overset{\textsf{SJDT}_2}\rightarrow
\YT{0.17in}{}{
			{\color{green} g, \pur p,2,},
			{ 4,4,\bar 3},
			{\overline 4,\pur p',},
{\bar 3}}
\overset{\textsf{SJDT}_2}\rightarrow
\YT{0.17in}{}{
			{\color{green} g,2,\overline 3,},
			{4, 4,\pur p},
			{\overline 4,\pur p',},
{{\bar 3}}}\\
&\overset{\textsf{SJDT}_2}\rightarrow
\YT{0.17in}{}{
			{2,4,\overline 3,},
			{4, \green g,\pur p},
			{\overline 4,\pur p',},
{{\bar 3}}} \overset{\textsf{SJDT}_2}\rightarrow
\YT{0.17in}{}{
			{\mt r,4,\overline 3,},
			{2, \green g,\pur p},
			{\overline 3,\pur p',},
{{\mt r'}}}\overset{\textsf{SJDT}_2}\rightarrow
\YT{0.17in}{}{
			{2,4,\overline 3,},
			{\bar 3, \green g,\pur p},
			{\mt  r,\pur p',},
{{\mt r'}}}\\
&=(\textsf{rect}_2 P_{[\pm 2,4]},V)\\
&\Rightarrow \textsf{rect}_2 P_{[\pm 2,4]}=\YT{0.17in}{}{
			{2,4,\overline 3},
			{\bar 3}},\end{align*}
            \noindent and
\[V= \YT{0.17in}{}{
			{,,,},
			{, \green g,\pur p},
			{\mt  r,\pur p',},
{{\mt r'}}}, {\mt r}<{\mt  r'}<{\green g}<{\pur p}<{\pur p'} .\]

 \textbf{2.} Computation of $\textsf{evac}^{C_3}\,\textsf{rect}_2 (P_{[\pm 2,4]})$. Taking $\pi$-rotation and changing the signs of
 $\textsf{rect}_2
 (P_{[\pm 2,4]})$, we again apply $SJDT_2$:
$$\YT{0.17in}{}{
			{, , 3},
			{ 3,\bar 4, \bar 2}}\overset{\textsf{SJDT}}\rightarrow\YT{0.17in}{}{
			{,  3,\bar 2},
			{ 3,\bar 4}}\overset{\textsf{SJDT}}\rightarrow\YT{0.17in}{}{
			{3,  3,\bar 2},
			{ ,\bar 4}}\overset{\textsf{SJDT}}\rightarrow
\YT{0.17in}{}{
			{3,  3,\bar 2},
			{ \bar 4}}=\textsf{evac}^{C_3}\textsf{rect}_2 P_{[\pm 2,4]}.$$

\medskip

\textbf{3.} Reversal of $P_{[\pm 2,4]}$. Replace  $\textsf{rect}_2(P_{[\pm 2,4]})$ with $\textsf{evac}^{C_3}(\textsf{rect}_2(P_{[\pm
2,4]}))$
in $(\textsf{rect}_2(P_{[\pm 2,4]}),V)$ and  apply $RSJDT_2$ to $V$:
\begin{align*}
    (\textsf{evac}^{C_3}\,(\textsf{rect}_2 P_{[\pm 2,4]}),V)&=\quad\;\;\; \YT{0.17in}{}{
			{3,3 ,\overline 2,},
			{\bar 4,\green g,\pur p},
			{{\mt r},\pur p',},
{{\mt r'}}}\overset{\textsf{RSJDT}_2}\rightarrow
\YT{0.17in}{}{
			{\mt r,3 ,\overline 2,},
			{3,\green g,\pur p},
			{\bar 4,\pur p',},
{{\mt r'}}}\overset{\textsf{RSJDT}_2}\rightarrow
\YT{0.17in}{}{
			{2,3 ,\overline 2,},
			{3,\green g,\pur p},
			{\bar 4,\pur p',},
{{\bar 2}}}\\
&\overset{\textsf{RSJDT}_2}\rightarrow
\YT{0.17in}{}{
			{\green g,2,\overline 2,},
			{3, 3,\pur p},
			{\overline 4,\pur p',},
{{\bar 2}}}\overset{\textsf{RSJDT}_2}\rightarrow
\YT{0.17in}{}{
			{\green g,\pur p, 3,},
			{3, 3,\bar 3,},
			{\overline 4,\pur p',},
{{\bar 2}}}\overset{\textsf{RSJDT}_2}\rightarrow
\YT{0.17in}{}{
			{\green g,2, 3,},
			{3, 3,\bar 3,},
			{\overline 4,\bar 2,},
{{\bar 2}}}\\
&=(U_0, \textsf{reversal}^{C_3}(P_{[\pm 2,4]}))\\ 
&\Rightarrow\textsf{reversal}^{C_3}(P_{[\pm 2,4]})=
\YT{0.17in}{}{
			{,2,3,},
			{3,3,\bar 3,},
			{\bar 4,\bar 2,},
{{\bar 2}}}.
\end{align*}

\textbf{4.} Replace $P_{[\pm 2,4]}$ with $\textsf{reversal}^{C_3}(P_{[\pm 2,4]})$ in $P$ to obtain
$$\textsf{ reversal}_{[ 2,4]}^{C_4}(P)=
\YT{0.17in}{}{
			{\color{gray} 1,2,3,\color{gray}\overline 1},
			{3,3,\bar 3},
			{\bar 4,\bar 2,\color{gray}\overline 1},
{{\bar 2}}}, \textsf{wt}_{[2,4]}( \xi^{C_4}_{[2,4]}(P))=-\textsf{wt}_{[2,4]}(P)=(-1,2,-1).
$$


\end{enumerate}

\subsection{General Dynkin sub-diagram and virtualization}
  Let $\xi^{C_n}_{[1,j]}$,  $ 1\le j< n$, be  the Sch\"utzenberger--Lusztig involution on $ \textsf{KN}_{[1,j]}(\lambda,n)$.
  Notice that the
  unique
  crystal operators which change the signs of the entries are $f_n$ and $e_n$, which are forgotten. Next  we give a computation of
  $\xi^{C_n}_{[p,q]}$, for any $ 1\le p\le q\le n$,  via virtualization $\gls{E}$ and bring it back to $
 \textsf{KN}_{[p,q]}(\lambda,n)$ by applying $E^{-1 }$.  See below (Theorem \ref{goingandbacR1}) for $ 1\le q< n$, and  Theorem
 \ref{goingandbacR2} otherwise.

\subsubsection{Embedding of a partial symplectic Sch\"utzenberger--Lusztig involution and back.} \label{sec:embedsymplecticreversal}
Let $J\subseteq [n]$ be a
sub-Dynkin diagram of the type $C_n$ Dynkin diagram $I=[n]$. Let $U$ be a connected component of the Levi branched crystal
$\gls{branchedsymptab}$ with $J\subseteq [n]$ and with highest and lowest weight elements $u^{\textsf{high}}$ and $u^{\textsf{low}}$,
respectively.  Recall from Subsection \ref{branchvirt}, Proposition \ref{highlow}, that each connected component $U$ of the Levi branched
crystal  \gls{branchedsymptab} is embedded via $\gls{E}$ into a connected component of the Levi branched  crystal  $\textsf{SSYT}_{J\sqcup
\bar
J}(\lambda, n)$ with highest and lowest weight elements $\gls{E}(u^{\textsf{high}})$ and $\gls{E}(u^{\textsf{low}})$, respectively.

Let $P=(P^+,P^-)\in $ $\textsf{SSYT}(\lambda^A, n,\bar n)$.  The crystals
 $\textsf{SSYT}_{[p,q]}(\lambda^A, n,\bar n)$ and \linebreak  $\textsf{SSYT}_{[\overline{q+1},\overline {p+1}]}(\lambda^A, n,\bar
n)$ are  isomorphic to $\textsf{SSYT}_{[p,q]}(\lambda^A_+, n)$ respectively  to \linebreak $\textsf{SSYT}_{[\overline{q+1},\overline
{p+1}]}(\lambda^A/\lambda^A_+, \bar
n)$, 
(recall Remark \ref{re:split}).  The corresponding pair of isomorphic crystals has the same multiset of highest
weight vectors  in $\mathbb{Z}^{q-p+1}$ respectively, regarding the sub-Dynkin diagram $[p,q]$.  We may then write
\begin{align}\xi^{A_{2n-1}}_{[p,q]}
 \xi^{A_{2n-1}}_{[\overline{q+1},\overline {p+1}]}( P)=(\xi^{A_{n-1}}_{[p,q]}(P^+), \xi^{A_{n-1}}_{[\overline{q+1},\overline
{p+1}]}(P^-)).
 \label{LSbranchvirt}\end{align}

\begin{theorem}
\label{goingandbacR1}
 Let  $ T\in  \textsf{KN}_{[p,q]}(\lambda,n)$ of type $A_{p-q+1}, 1\le p\le
 q < n$.
 Then
 \[\xi^{A_{2n-1}}_{[p,q]\sqcup [\overline{q+1},\overline{ p+1}]}(\gls{E}(T))=\xi^{A_{2n-1}}_{[p,q]} \xi^{A_{2n-1}}_{[\overline{q+1},\overline{
 p+1}]}(\gls{E}(T)) = \gls{E}(\xi^{C_n}_{[p,q]}(T)).\]
 Moreover,
 \begin{equation}\xi^{C_n}_{[p,q]}(T)= {E^{-1}\textsf{reversal}^{A_{2n-1}}_{[p,q]} \textsf{reversal}^{A_{2n-1}}_{[\overline{q+1},\overline{
 p+1}]}
 E}(T).\label{branchedreversal}
 \end{equation}

 \end{theorem}

 \begin{proof}  Recall Proposition \ref{highlow}, Remark \ref{re:split}  and \eqref{decompose1}. Then it follows from Theorem
 \ref{th:schutzunion}.
 \end{proof}

 It is now convenient to change the labeling of the $A_{2n-1}$ Dynkin diagram. Instead of $ [k,\overline{k+1}]$, we write $[k,2n-k]$, and
 $\textsf{SSYT}_{[k,2n-k]}(\lambda, n,\bar n)$. This relabelling is illustrated in the picture below.
 $$\dynkin[edge length=1cm,labels*={1,2,3,n-1,n}] C{***.**}
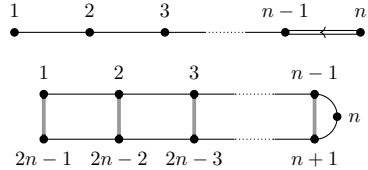
$$
$$\\~\\ \dynkin[edge length=1cm,fold,labels={1,2,3,n-1,n, n+1, 2n-3, 2n-2, 2n-1}]A{***.***.***}$$

\begin{theorem}
\label{goingandbacR2}
 Let $ T\in \textsf{KN}_{[k,n]}(\lambda,n)$ of type $C_{n-k+1}$ for some
 $1\le k\le n$. Then
\[ \xi^{A_{2n-1}}_{[k,2n-k]}(\gls{E}(T)) = \gls{E}(\xi^{C_n}_{[k,n]}(T)).\]
Moreover, on  $\gls{ssytv}$, $\xi^{A_{2n-1}}_{[k,2n-k]}=\textsf{reversal}^{A_{2n-1}}_{[k,2n-k]}$, and
\begin{align}\xi^{C_n}_{[k,n]}=E^{-1}\textsf{reversal}^{A_{2n-1}}_{[k,2n-k]} E.\label{branchedreversal2}\end{align}
 \end{theorem}

 \begin{proof}
  Recall Corollary \ref{theta} and that, in the case of the branched crystal  $$\textsf{SSYT}_{[k,2n-k]}(\lambda, n,\bar n),$$
  $\theta(i)=2n-i$, for $i\in [k,2n-k]$. Let $ U$ be the connected component of  $\textsf{KN}_{[k,n]}(\lambda,n)$  containing $T$, and
  let the highest and lowest weight
elements of $U$ be $u^{\textsf{high}}$ and $u^{\textsf{low}}$, respectively.
 $$ T=f_{i_r}\dots f_{i_1}(u^{\textsf{high}}), i_1,\dots,i_r\in[k,n],$$
\begin{align}\gls{E}(T)=f^A_{i_r}f^A_{2n-i_r}\dots f^A_{i_1}f^A_{2n-i_1}(\gls{E}(u^{\textsf{high}})),
i_1,\dots,i_r\in[k,n],\label{eq:here}\end{align}
and $$\xi^{C_n}_{[k,n]}(T)=e_{i_r}\dots e_{i_1}(u^{\textsf{low}}).$$
Then, from  Subsection \ref{bakerembedding}, \eqref{ssytv},
$$ \gls{E}(\xi^{C_n}_{[k,n]}(T))=\gls{E}(e_{i_r}\dots e_{i_1}(u^{\textsf{low}}))=e^A_{i_r}e^A_{2n-i_r}\dots
e^A_{i_1}e^A_{2n-i_1}\gls{E}(u^{\textsf{low}})$$
and, from \eqref{eq:here},\begin{align*}\xi^{A}_{[k,2n-k]}(\gls{E}(T))&=e^A_{\theta (i_r)}e^A_{\theta (2n-i_r)}\dots e^A_{\theta
(i_1)}e^A_{\theta
(2n-i_1)}(\gls{E}(u^{\textsf{low}}))\\
&=e^A_{2n-i_r}e^A_{i_r}\dots e^A_{2n-i_1}e^A_{ i_1}(\gls{E}(u^{\textsf{low}}))\\
&=e^A_{i_r}e^A_{2n-i_r}\dots
e^A_{i_1}e^A_{2n-i_1}\gls{E}(u^{\textsf{low}})=\gls{E}(\xi^{C}_{[k,n]}(T)).\end{align*}

Finally, \eqref{branchedreversal2} follows  from (\ref{def:reversal}).
\end{proof}

Using  a generalized form of Lemma \ref{th:disconnected}, the following corollary is  a generalization of the two  theorems above.
 \begin{corollary}
\label{goingandbacR3}
 Let $T\in \textsf{KN}_{[p,q]\sqcup [k,n]}(\lambda,n)$  of subtype $A_{p-q+1}\times C_{n-k+1}$ for some
 $  1\le p$ $\le q<$ $k-1< n$. Then
\begin{align*} \xi^{A_{2n-1}}_{[p,q]\sqcup[2n-q,2n-p]\sqcup[k,2n-k]}(\gls{E}(T)) &=\xi^{A_{2n-1}}_{[p,q]} \xi^{A_{2n-1}}_{[2n-q,2n-p]}
\xi^{A_{2n-1}}_{[k,2n-k]}(\gls{E}(T)) \\
&=  \gls{E}(\xi^{C_n}_{[p,q]\sqcup [k,n]}(T)).\end{align*}
Moreover, on  $\gls{ssytv}$, 
\begin{align*}
    \xi^{A_{2n-1}}_{[p,q]}&=\textsf{reversal}^{A_{2n-1}}_{[p,q]},\\
\xi^{A_{2n-1}}_{[k,2n-k]}&=\textsf{reversal}^{A_{2n-1}}_{[k,2n-k]}, 
\\
\xi^{A_{2n-1}}_{[2n-q,2n-p]}&=\textsf{reversal}^{A_{2n-1}}_{[2n-q,2n-p]}.&\end{align*}
 \end{corollary}

\begin{remark}\label{re:preserve}
Both $\xi^{A_{2n-1}}_{[p,q]\sqcup [2n-q,2n-p]}$ and $\xi^{A_{2n-1}}_{[k,2n-k]}$ act on the set $\gls{ssytv}$ to define a graph automorphism of
the underlying graph such that the subset $\gls{E}(\gls{symptab})$ is preserved. In other words, each of these involutions defines a  graph
automorphism of the underlying graph of $\gls{E}(\gls{symptab})$ when its action is restricted to this subset.
\end{remark}

\begin{corollary}
\label{virtWeylaction}
Let $\textsf{SSYT}(\mu,2n)$ with $\mu$ a partition with at most $2n$ parts, and let $\gls{bn}$ be the Weyl group realized as
$$ \langle r_i=(i\;i+1)(2n-i\;2n-i+1),\, r_n=(n\;n+1): \;1\le i\le n-1\rangle.$$ Then $\xi^{A_{2n-1}}_i\xi^{A_{2n-1}}_{2n-i},\xi^{A_{2n-1}}_n,\; 1\le i\le
n-1,$ define an action of $\gls{bn}$ on $\textsf{SSYT}(\mu,2n)$ by $ r_i.b=\xi^{A_{2n-1}}_i\xi^{A_{2n-1}}_{2n-i}(b)$, $1\le i\le n-1$,
and
$r_n.b=\xi^{A_{2n-1}}_n(b)$, for $b\in \textsf{SSYT}(\mu,2n)$ such that
\begin{enumerate}
\item
    $e^{A_{2n-1}}_ie^{A_{2n-1}}_{2n-i}\xi^{A_{2n-1}}_i\xi^{A_{2n-1}}_{2n-i}=\xi^{A_{2n-1}}_i\xi^{A_{2n-1}}_{2n-i}f^{A_{2n-1}}_if^{A_{2n-1}}_{2n-i}$,
    $1\le i<n$,
    \item $e^{A_{2n-1}}_ne^{A_{2n-1}}_{n}\xi^{A_{2n-1}}_n=\xi^{A_{2n-1}}_nf^{A_{2n-1}}_nf^{A_{2n-1}}_{n}$,
\item $\textsf{wt}(r_i.b)=r_i.\textsf{wt}(b)$, $1\le i\le n$,
\item $w_0.T_\mu=T^{\textsf{low}}_\mu$, $w_0$ the long element of $\gls{bn}$,
\item if $\mu=\lambda^A$,  for some $\lambda$, it preserves the action of $r_i\in \gls{bn}$ on the $i$-strings,
    $1\le i\le n$, of the crystal $\gls{symptab}$ embedded in the crystal $\textsf{SSYT}(\lambda^A,2n)$. \label{BN}
\end{enumerate}
\end{corollary}
\begin{proof}  Recall that $\xi_i$, $1\le i\le 2n-1$, define an action of $\mathfrak{S}_{2n}$ on $\textsf{SSYT}(\mu,2n)$. Thus the
involutions $\xi^{A_{2n-1}}_i\xi^{A_{2n-1}}_{2n-i}$ and $\xi^{A_{2n-1}}_n$, $1\le i<n-1$, satisfy the braid relations of $B_n$. The connected
components of the crystal $\textsf{B}_{\{i\}\sqcup \{2n-i\}}$, the restriction of $\textsf{B}=\textsf{SSYT}(\mu,2n)$ to the Dynkin diagram
$\{i\}\sqcup \{2n-i\}$ of $\mathfrak{sl}_2\times \mathfrak{sl}_2$, are grids of rectangles where we have $(1)$. It  follows from Theorem
\ref{goingandbacR1} and Theorem \ref{goingandbacR2} with $p=q$ and $k=n$, respectively, that
$\xi^{A_{2n-1}}_{p} \xi^{A_{2n-1}}_{\overline{ p+1}}(\gls{E}(T)) =\xi^{A_{2n-1}}_{p} \xi^{A_{2n-1}}_{ 2n-p}(\gls{E}(T))=
\gls{E}(\xi^{C_n}_{p}(T))$ and $\xi^{A_{2n-1}}_{n}(\gls{E}(T)) = \gls{E}(\xi^{C_n}_{n}(T))$. From Proposition \ref{weylaction}  (\ref{refection
c}), the involutions $\xi^{C_n}_{i}$, $1\le i\le n$, define an action of $\gls{bn}$ on $\gls{symptab}$. Therefore $
\xi^{A_{2n-1}}_i\xi^{A_{2n-1}}_{2n-i},\xi^{A_{2n-1}}_n, \,1\le i\le n-1,$ are the translation of this action to the embedded crystal
$\gls{E}(\gls{symptab})$ in $\textsf{SSYT}(\lambda^A,2n)$. This proves \eqref{BN}.
\end{proof}

\subsection{Virtualization of the action  of \texorpdfstring{$\gls{Jsp}$}{Jsp} on the crystal \texorpdfstring{\gls{symptab}}{symptab}}
 \label{sec:virtaction}

We have the following commutative diagram corresponding to the crystal embedding $\gls{E}$ and the partial $C_n$ and $A_{2n-1}$
Sch\"utzenberger--Lusztig involutions, where $[p,q]\subseteq [n-1]$ and $[p,n]\subseteq [n]$ are connected subintervals of the Dynkin diagram
of $C_n$,
\[
\vcenter{\hbox{\begin{tikzpicture}[scale=0.6]
\node(tl) at (0,3){\gls{symptab}};
\node(tr) at (5,3){\gls{ssytv}};
\node(bl) at (0,0){\gls{symptab}};
\node(br) at (5,0){\gls{ssytv}};
\draw[->](tl)--node[above]{$\gls{E}$}(tr);
\draw[->](bl)--node[below]{$\gls{E}$}(br);

\draw[->](tl)--node[left]{$\xi^{C_{n}}_{[p,n]}$}(bl);
\draw[->](tl)--node[right]{$\xi^{C_{n}}_{[p,q]}$}(bl);
\draw[->](tr)--node[left]{$\xi^{A_{2n-1}}_{[p,\overline{ p+1}]}$}(br);
\draw[->](tr)--node[right]{$\xi^{A_{2n-1}}_{[p,q]\sqcup[\overline{ q+1},\overline{ p+1}]}$}(br);
\end{tikzpicture}}}.
\]
Theorem \ref{cactusj2n} and Remark \ref{re:preserve} imply that the action of $\gls{vJ2n}$ on $\gls{ssytv}$
preserves the subset $\gls{E}(\gls{symptab}$, and thus, we have an action of $\gls{vJ2n}$ on the set $\gls{E}(\gls{symptab})$ defined by
\begin{align} \widetilde\Phi^E_{\mathfrak{sl}(2n,\mathbb{C})}&:& \gls{vJ2n}&&\longrightarrow&& \mathfrak{S}_{\gls{E}(
\gls{symptab})}&\label{preserve2}\\
&& \tilde s_{[p,q]\sqcup [\overline{ q+1},\overline{ p+1}]}&&\mapsto&& \xi^{A_{2n-1}}_{[p,q]\sqcup[\overline{ q+1},\overline{
p+1}]}&=\xi^{A_{2n-1}}_{[p, q]}\xi^{A_{2n-1}}_{[\overline{ q+1}, \overline{ p+1}]} \nonumber\\
&&\tilde s_{[p,\overline{ p+1}]}&&\mapsto&&\xi^{A_{2n-1}}_{[p,\overline{ p+1}]}& \nonumber
\end{align}

\noindent such that $\widetilde\Phi^E_{\mathfrak{sl}(2n,\mathbb{C})}(\tilde s_J)=\widetilde\Phi_{\mathfrak{sl}(2n,\mathbb{C})}(\tilde
s_J)_{|\gls{E}(
\gls{symptab})}\in \mathfrak{S}_{\gls{E}(   \gls{symptab})}$.
Let $\tilde \imath:\gls{Jsp}\rightarrow \gls{vJ2n}$ be the group isomorphism defined by $s_{[1,j]}\mapsto \tilde
s_{[1,j]\sqcup [\overline{ j+1} ,\bar 2]}$, $1\le j<n$,   and $s_{[j,n]}\mapsto \tilde s_{[j,\overline{ j+1}]}$, $1\le j< n$, (see
Proposition \ref{virtualcactus}), and $
\imath:\mathfrak{S}_{  \gls{symptab}}\rightarrow \mathfrak{S}_{\gls{E}(  \gls{symptab})}$ the group isomorphism defined by
$\imath(\sigma)=E\sigma
E^{-1}$. The virtualization of the action  of $\gls{Jsp}$ on the crystal $\gls{symptab}$ is then realized from the
following commutative diagram
\begin{align}
\vcenter{\hbox{\begin{tikzpicture}[scale=0.6]
\node(tl) at (0,2){$\gls{Jsp}$};
\node(tr) at (5,2){$\mathfrak{S}_{  \gls{symptab}}$};
\node(bl) at (0,0){$\gls{vJ2n}$};
\node(br) at (5,0){$\mathfrak{S}_{\gls{E}(   \gls{symptab})}$};
\draw[->](tl)--node[above]{$\Phi_{\mathfrak{sp}(2n,\mathbb{C})}$}(tr);
\draw[->](bl)--node[below]{$\widetilde\Phi^E_{\mathfrak{sl}(2n,\mathbb{C})}$}(br);
\draw[->](tl)--node[left]{$\tilde\imath$}(bl);
\draw[->](tr)--node[left]{$\imath$}(br);
\end{tikzpicture}}}& \widetilde\Phi^E_{\mathfrak{sl}(2n,\mathbb{C})}\tilde\imath=\imath\Phi_{\mathfrak{sp}(2n,\mathbb{C})} .\label{diag}
\end{align}
From \eqref{diag}, 
\begin{align*}\widetilde\Phi^E_{\mathfrak{sl}(2n,\mathbb{C})}\tilde\imath(s_{[1,j]})&=\widetilde\Phi^E_{\mathfrak{sl}(2n,\mathbb{C})}(
\tilde s_{[1,j]\sqcup [\overline{j+1},\bar 2]})= \xi^{A_{2n-1}}_{[1,j]\sqcup[\overline{ j+1},\bar
2]}\\
&=\imath\Phi_{\mathfrak{sp}(2n,\mathbb{C})}(s_{[1,j]})=\imath\xi^{C_{n}}_{[1,j]}\\
&=E\xi^{C_{n}}_{[1,j]}E^{-1}
= \xi^{A_{2n-1}}_{[1,j]\sqcup
[\overline{ j+1},\bar 2]},\end{align*} and 
\begin{align*}\widetilde\Phi^E_{\mathfrak{sl}(2n,\mathbb{C})}
\tilde\imath(s_{[j,n]})&=\widetilde\Phi^E_{\mathfrak{sl}(2n,\mathbb{C})}(s_{[j,\overline{j+1}]})
=\xi^{A_{2n-1}}_{[j,\overline{
j+1}]}\\
&=\imath\Phi_{\mathfrak{sp}(2n,\mathbb{C})}(s_{[j,n]})=\imath\xi^{C_{n}}_{[j,n]}\\&=E\xi^{C_{n}}_{[j,n]}E^{-1}
=\xi^{A_{2n-1}}_{[j,\overline{
j+1}]}.\end{align*}

\subsection{Virtualization example}\label{ex:virt}

Consider $n=5$, J=[1,4] and the $KN$ tableau $T$ of shape $\lambda=\omega_4+\omega_3$:
\[
T=\YT{0.19in}{}{
	    	{{1},{1}},
			{{3},{\overline{5}}},
			{{\overline{4}},{\overline{3}}},
			{{\overline{3}}}},\quad { \textsf{wt}(T)=(2,0,-1,-1,-1)}.
\]
Following the conventions in Section \ref{virtualization}, and labeling the columns of $T$ from  right to left as $C_1$ and $C_2$, we obtain $E(T)$ with shape
$\lambda^A=\omega_7+\omega_3+\omega_6+\omega_4$:

\begin{align*}
&\psi(C_1)=\YT{0.19in}{}{
	    	{{1},{1}},
			{{2},{\overline{5}}},
			{{4},{\overline{3}}},
			{{\overline{5}}},
			{{\overline{4}}},
			{{\overline{3}}},
			{{\overline{2}}}},
\psi(C_2)=\YT{0.19in}{}{
	    	{{1},{1}},
			{{2},{3}},
			{{5},{\overline{4}}},
			{{\overline{5}}, {\overline{2}} },
			{{\overline{4}}},
			{{\overline{3}}}} \\
            &\Rightarrow
\gls{E}(T)=[\emptyset\leftarrow w(\psi(C_1)) \leftarrow w(\psi(C_2))]=
\YT{0.19in}{}{
	    	{{1},{1},{1},{1}},
			{{2},{2},{4},{\overline{5}}},
			{{3},{\overline{5}},{\overline{4}},{\overline{3}}},
			{{5},{\overline{4}},{\overline{2}}},
			{{\overline{5}},{\overline{3}}},
			{ {\overline{4}},{\overline{2}} },
			{{\overline{3}}}}.
\end{align*}
Considering the barred and unbarred parts of $\gls{E}(T)$ separately, we compute the evacuation, $\textsf{evac}$, of the unbarred part and the
reversal, $\gls{rev}$, of the barred part, yielding:
\[
\textsf{evac} \YT{0.19in}{}{
	    	{{1},{1},{1},{1}},
			{{2},{2},{4}},
			{{3}},
			{{5}}}
= \YT{0.19in}{}{
	    	{{1},{3},{4},{5}},
			{{2},{5},{5}},
			{{4}},
			{{5}}} ,
\gls{rev} \YT{0.19in}{}{
	    	{{\ast},{\ast},{\ast},{\ast}},
			{{\ast},{\ast},{\ast},{\overline{5}}},
			{{\ast},{\overline{5}},{\overline{4}},{\overline{3}}},
			{{\ast},{\overline{4}},{\overline{2}}},
			{{\overline{5}},{\overline{3}}},
			{ {\overline{4}},{\overline{2}} },
			{{\overline{3}}}}
=\YT{0.19in}{}{
	    	{{\ast},{\ast},{\ast},{\ast}},
			{{\ast},{\ast},{\ast},{\overline{3}}},
			{{\ast},{\overline{4}},{\overline{2}},{\overline{2}}},
			{{\ast},{\overline{3}},{\overline{1}}},
			{{\overline{4}},{\overline{2}}},
			{{\overline{3}},{\overline{1}} },
			{{\overline{1}}}}.
\]
Putting these tableaux together, one obtains the $A_{9}$ tableau
$$\xi^{A_{9}}_{[1,4]\sqcup[\bar 5, \bar 2]}(\gls{E}(T))=(\textsf{evac}(\gls{E}(T)^+), \gls{rev}(\gls{E}(T)^-)).$$
Using $Q_\lambda$ to perform the reverse column Schensted insertion on $\xi^{A_{9}}_{[1,4]\sqcup[\bar 5, \bar 2]}(\gls{E}(T))$ provides the
image under $\psi$ of two KN columns $C'_1$ and $C'_2$. Applying $\psi^{-1}$ to each column results in:

\begin{align*}
    &Q_\lambda=\YT{0.19in}{}{
	    	{{1},{4},{11},{15}},
			{{2},{5},{12},{16}},
			{{3},{6},{13},{17}},
			{{7},{14},{18}},
			{{8},{19}},
			{{9},{20} },
			{{10}}} \Rightarrow
\psi(C'_1)=\YT{0.19in}{}{
	    	{{1},{5}},
			{{4},{\overline{3}}},
			{{5},{\overline{2}}},
			{{\overline{4}}},
			{{\overline{3}}},
			{{\overline{2}}},
			{{\overline{1}}}},
\psi(C'_2)=\YT{0.19in}{}{
	    	{{2},{3}},
			{{4},{5}},
			{{5},{\overline{2}}},
			{{\overline{4}}, {\overline{1}}},
			{{\overline{3}}},
			{{\overline{1}}}}\\
            &\Rightarrow
C_2'C_1'=\YT{0.19in}{}{
            {{3},{5}},
			{{5},{\overline{3}}},
			{{\overline{3}},{\overline{2}}},
			{{\overline{1}}}}=\xi^C_{[1,4]}(T).
\end{align*}
We note that $\textsf{wt}(\xi^C_{[1,4]}(T))=rev(\textsf{wt}(T))=(-1,-1,-1,0,2).$


\section{The type \texorpdfstring{$C_n$}{C}  Berenstein--Kirillov group}
\label{sec:symplecticbk}
\subsection{The type \texorpdfstring{$A$}{A}  Berenstein--Kirillov group}

The type $A$ \emph{Berenstein--Kirillov group} $\gls{BK}$ (or \emph{Gelfand--Tsetlin group}) \cite{bk95} is the free group generated by
the Bender--Knuth involutions \cite{bekn72} $t_i$, $i > 0$, modulo the relations they satisfy on semi-standard Young tableaux of any
(straight) shape.

\begin{definition}\label{def:benderknuth}
 The Bender--Knuth involution $t_i$, $i\ge 1$, is an operation that acts on a  semi-standard tableau $T$ of any shape (skew or straight)
 as follows:

\begin{itemize}
\item pairs $(i,i+1)$ within each column of $T$ are considered fixed, and other occurrences of $i$’s or $i + 1$’s are considered free
\item  if a row within $T$ has $k$ free $i$’s followed by $l$ free $i + 1$’s, then we replace these letters
by $l$ free $i$’s followed by $k$ free $i + 1$’s.
\end{itemize}
\end{definition}

 The $t_i$'s have many known relations in $\gls{BK}$ \cite{bk95, cgp16}:
\begin{align}
 t_i^2 &= 1,& \;
 \text{for\;$i\ge 1$}&\,\text{\cite[Corollary 1.1]{bk95}}\label{eq:bkrelations0}\\
 t_i t_j &= t_j t_i,&\; \text{for\; $|i-j|>1$}&\,\text{\cite[Corollary 1.1]{bk95}}  \label{eq:bkrelations1},\\
(t_1 q_{[1,i]})^4 &= 1,&\; \text{for\; $i > 2$}&\,\text{\cite[Corollary 1.1]{bk95}}  \label{eq:bkrelations2+},\\
 (t_1 t_2)^6 &= 1,&&\,\text{ \cite[Corollary 1.1]{bk95}} \label{eq:bkrelationspecial},\\
(t_i q_{[j,k-1]})^2 &= 1,& \;\qquad \quad \qquad \text{ \qquad \quad   for \;  $i+1<j<k$},& \;\;\text{\cite{cgp16}} \label{eq:bkrelationextra},
\end{align}
where
 \begin{align}\label{evacBK}q_{[1,i]} &:= t_1 (t_2 t_1) \cdots (t_i t_{i-1} \cdots t_1),& \text{for $i \geq 1$},\\ 
  q_{[j,k-1]}&:= q_{[1,k-1]} q_{[1,k-j]} q_{[1,k-1]},&\text{ for $j<k$}.\label{evacBK2}
\end{align}

\begin{remark}\label{re:Abenderknuth10}

$1.$
It is not known whether the latter set forms a complete set of relations \cite{bgli19}.

\noindent $2.$ \cite[Section 2]{bk95} On straight-shaped semi-standard Young tableaux,
  \begin{align}\label{bkschutz}q_{[1,i]}=\xi_{[1,i]},\; i\ge 1,\;&\;\;
 q_{[j,k-1]}=\xi_{[j,k-1]},\; j<k,\end{align}
\noindent and
$q_{[j,j]}= q_{[1,j]} q_{[1,1]} q_{[1,j]}$   computes the crystal reflection operator $\xi_j=\xi_{[j,j]}$,
where $q_{[1,1]}=\xi_{[1,1]}=t_1$, for $j\ge 1$. In particular,  $q_{[1,i]}=\xi_{[1,i]}=\textsf{evac}_{i+1}$, the evacuation restricted to
the alphabet
$[1,i+1]$, and $q_{[j,k-1]}$  computes  the Sch\"utzenberger evacuation restricted to the alphabet $[j,k]$, 
\begin{align*}\xi_{[j,k-1]}=\break
\textsf{evac}_{k}\textsf{evac}_{k-j+1} \textsf{evac}_{k},\;\mbox{ for $j<k$}.
\end{align*}

\noindent $3.$
Relation \eqref{eq:bkrelationextra}
    implies that in particular, $(t_i\xi_j)^2=1$, $j>i+1$,  which generalizes the relation $(t_1 q_{[1,i]})^4=1$ .


\noindent $4.$
For a generic (straight or skew) shaped semi-standard Young tableau $T$,
    \begin{align*}
        \textsf{wt}(t_i(T))=\textsf{wt}(\xi_i(T))=r_i.\textsf{wt}(T),
    \mbox{ $r_i\in\mathfrak{S}_n$ },\; \mbox{  for all $n\ge 1$}.
    \end{align*}
    However, $t_i\neq\xi_i$, for $i>1$; $t_1$ and $\xi_1$ need only coincide on straight shaped
    semi-standard Young tableaux whereas $t_i$ and 
    $\xi_i$, for $i>1$, do not. Moreover, $t_i$,  $1\le i<n$, do not need to satisfy the braid relations of $\mathfrak{S}_n$, however, they do on  key tableaux, that is, straight shaped tableaux whose weight is a permutation of its shape \cite{Fu}.
\end{remark}

Let $\gls{BKn}$ be the subgroup of $\gls{BK}$ generated by $t_1,\dots,t_{n-1}$.
\begin{proposition}{\em \cite[Remark 1.3]{bk95}}\label{prop:ti_si}
As elements of $\gls{BK}$,
$${t}_i = {q}_{[1,i-1]} {q}_{[1,i]}{q}_{[1,i-1]} {q}_{[1,i-2]}, \,\text{for}\; i\geq 1,\;\;{q}_{[1,0]}= {q}_{[1,-1]}:= 1.$$
 The  elements ${q}_{[1,1]}, \ldots, {q}_{[1,n-1]}$ are generators of $\gls{BKn}$.

\end{proposition}

The following result is both a consequence of the combinatorial action of the cactus group \gls{Jn} via partial Sch\"utzenberger involutions
$\xi_{[1,i]}$ on the straight-shape tableau crystal \gls{ssyt}, as defined by Halacheva \cite{ha16}, and the cactus group \gls{Jn} relations
satisfied by the generators $q_{[i,j]}=\xi_{[i,j]}$ of $\gls{BKn}$ when acting on \gls{ssyt}, as studied by Chmutov, Glick and Pylyavskyy
via
the growth diagram approach \cite{cgp16}.

\begin{theorem} \label{th:jnbk}
The  following are group epimorphisms from  \gls{Jn} to $\gls{BKn}$:
\begin{enumerate}
\item $s_{[i,j]} \mapsto q_{[i,j]}$ \cite[Theorem 1.4]{cgp16},
\item $s_{[1,j]} \mapsto q_{[1,j]}$ \cite[Remark 1.3]{bk95}, \cite[Section 10.2]{ha16}, \cite[Remark 3.9]{ha20}.
\end{enumerate}

The group $\gls{BKn}$ is isomorphic to a quotient of  \gls{Jn}. The generators ${q}_{[1,1]}, \ldots, {q}_{[1,n-1]}$  of $\gls{BKn}$
 $($and therefore $q_{[i,j]}$$)$  satisfy the relations of  \gls{Jn}.
\end{theorem}

\begin{remark}\label{re:special} It follows from \cite{cgp16} that \eqref{eq:bkrelationspecial}  is the only known relation which does not
follow from the cactus group $\gls{Jn}$ relations. It is in fact equivalent to the braid relations  satisfied  by the crystal reflection
operators $\xi_i=\xi_{[1,i]}t_1\xi_{[1,i]}$, $1\le i<n$, on a $U_q(\mathfrak{sl}(n,\mathbb{C}))$ crystal \cite[Proposition 1.4]{bk95},
\cite{ro21}.
\end{remark}

\begin{remark}\label{re:dualgenerator} We may define the two dual sets of generators of $\gls{BKn}$
\begin{align}\tilde t_{n-i}:={q}_{[1,n-1]}t_i{q}_{[1,n-1]},\;1\le i<n, \label{eq:tildet}
\end{align}
called dual Bender--Knuth involutions, and
\begin{align}{\tilde q}_{[1,i]}:={q}_{[1,n-1]}q_{[1,i]}{q}_{[1,n-1]}={q}_{[n-i,n-1]},\; \;1\le i<n,\label{eq:tildeq}\end{align}
 for $\gls{BKn}$. Indeed, from Proposition \ref{prop:ti_si} and Theorem \ref{th:jnbk}, one has
\begin{align} 
{\tilde t}_{n-i} = {q}_{[n-i+1,n-1]} {q}_{[n-i,n-1]}{q}_{[n-i+1,n-1]} {q}_{[n-i+2,n-1]},\; \mbox{  for
$1\le i < n$ } \label{eq:tildet2}\end{align}

\noindent with  ${q}_{[n,n-1]}={q}_{[n+1,n-1]} := 1$, and $\textsf{wt}(\tilde t_{n-i}(T))=r_{n-i}.\textsf{wt}(T)$ for $T \in \gls{ssyt}$ and $r_i\in
\mathfrak{S}_n$, $i<n$.

The dual generators satisfy a list of relations similar to \eqref{eq:bkrelations0} \eqref{eq:bkrelations1}, \eqref{eq:bkrelations2+},
\eqref{eq:bkrelationspecial}, \eqref{eq:bkrelationextra}:
\begin{align}
 \tilde t_{n-i}^{\,2} &= 1, \;
 \text{for\;$i\ge 1$}\,\label{eq:bkrelations0tilde}\\
 \tilde t_{n-i} \tilde t_{n-j} &= \tilde t_{n-j} \tilde t_{n-i},\; \text{for\; $|i-j|>1$}  \label{eq:bkrelations1tilde},\\
 (\tilde t_{n-1} \tilde t_{n-2})^6 &= 1,\label{eq:bkrelationspecialtilde}\\
(\tilde t_{n-i} \tilde q_{[j,k-1]})^2 &= (\tilde t_{n-i}  q_{[n-k+1,n-j]})^2\nonumber \\
&=1, \; \text{for \; $n-k<n-j<n-i-1$,}
\label{eq:bkrelationextratilde}
\end{align}
where
\begin{align}\label{tildeevacBK}\tilde q_{[1,i]}& = \tilde t_{n-1} (\tilde t_{n-2} \tilde t_{n-1}) \cdots (\tilde t_{n-i} \tilde t_{n-i+1}
 \cdots \tilde t_{n-1}),\;  \text{for $i \geq 1$},\\
  \tilde q_{[j,k-1]}&:=\tilde q_{[1,k-1]}\tilde q_{[1,k-j]}\tilde q_{[1,k-1]}\nonumber\\
  & \;=q_{[n-k+1,n-j]} \nonumber \\
   & =  q_{[n-k+1,n-1]} q_{[n-k+j,n-1]} q_{[n-k+1,n-1]},\;\text{ for $j<k$}.\label{tildeevacBKlast}
\end{align}
On $ \gls{ssyt}$, $\tilde q_{[j,k-1]}=\xi_{[n-k+1,n-j]}$. 

\end{remark}

\begin{remark}\label{re:tildet}  We note some features of the operators \eqref{eq:tildet} when acting on straight shaped semi-standard tableaux. Set $\textsf{evac}:=\textsf{evac}_n$. Let $1\le i<n$, and $T=(A,B) \in \gls{ssyt}$ where $A$  of straight shape is the restriction of $T$ to the alphabet $[1,n-i-1]$ and $B$, an extension of $A$, is the restriction of $T$ to the alphabet $[n-i, n]$. One has $\textsf{evac}(A,B)=(\textsf{evac}\,\textsf{rect}(B),X)$ with $X$ such that $\textsf{rect}\,(X)=\textsf{evac}\,( A)$. Therefore,
\begin{align}\tilde t_{n-i}(T)&=\tilde t_{n-i}(A,B)=\textsf{evac}\,\, t_i\,\textsf{evac}(A,B) \hbox{ by \eqref{eq:tildet} }\nonumber\\
&=\textsf{evac}\,\,t_i(\textsf{evac}\,\textsf{rect}(B),X),\hbox{ such that } \textsf{rect}(X)=\textsf{evac}(A) \nonumber\\
&=\textsf{evac}\,(t_i\,\,\textsf{evac}\,\textsf{rect}(B),X)\nonumber\\
&=(\textsf{evac}\,\textsf{rect}(X),Z)\nonumber\\
&=(A,Z) \hbox{ such that } \textsf{rect}(Z)=\textsf{evac}\,\,t_i\,\textsf{evac}\,\textsf{rect}(B).
\end{align}

\end{remark}

\subsection{The type \texorpdfstring{$C_n$}{C} Berenstein--Kirillov group and virtualization}
Symplectic Bender--Knuth involutions $ t_i^{C_n}$ are not known for KN tableaux.
Motivated by the fact that for $n\geq 1$, ${q}_{[1,1]}, \ldots, {q}_{[1,n-1]}$  are generators for the Berenstein--Kirillov group
$\gls{BKn}$ in type $A$,  and that on straight shaped semi-standard tableaux, they coincide with the action of the partial
Sch\"utzenberger--Lusztig  involutions $\xi_{[1,i]}$, $1\le i<n$, we use the action of the  partial Sch\"utzenberger--Lusztig involutions
$\xi_{[1,i]}^{C_n}, \;1\le i\le n-1,$ and $\xi_{[i,n]}^{C_n}\; 1\le i\le n,$ on  KN tableaux of any straight shape  on the alphabet
$\gls{Corderedletters}$ to define the type $C_n$ \emph{ Berenstein--Kirillov group}, $\gls{BKc}$.

\begin{definition} \label{def:bksymplectic} Given $n\ge 1$, the \emph{symplectic Berenstein--Kirillov group} $\gls{BKc}$  is the free group
generated by the
$2n-1$  partial Sch\"utzenberger--Lusztig  involutions $$q^{C_n}_{[1,i]}:=\xi^{C_n}_{[1,i]}, \; 1\le i <n,$$ and
$$q^{C_n}_{[i,n]}:=\xi^{C_n}_{[i,n]}, \;1\le i\le n,$$ on straight shaped KN tableaux on the alphabet $\gls{Corderedletters}$ modulo the
relations they
satisfy on those tableaux. We also define ${q}^{C_n}_{[1,-1]}={q}^{C_n}_{[1,0]} ={q}^{C_n}_{[0,n]}=q^{C_n}_{[n+1,n]}:=1$ and $q^{C_n}_{[j,k-1]}:=
q^{C_n}_{[1,k-1]} q^{C_n}_{[1,k-j]} q^{C_n}_{[1,k-1]}$, $1\le j<k\le n$.
\end{definition}

\begin{remark}\label{re:C16(2)}Likewise in type $A$, Remark \ref{re:Abenderknuth10} \eqref{bkschutz}, one also has in type $C_n$, thanks to Theorem \ref{goingandbacR1},
$q^{C_n}_{[j,k-1]}=\xi^{C_n}_{[j,k-1]}=E^{-1}\xi_{[j,k-1]}\xi_{[2n-k+1,2n-j]} E$, $1\le j<k\le n$. 
For  $ 1\le j<k\le n$, 
\begin{align*}
q^{C_n}_{[j,k-1]}&=q^{C_n}_{[1,k-1]} q^{C_n}_{[1,k-j]} q^{C_n}_{[1,k-1]}=\xi^{C_n}_{[1,k-1]}\xi^{C_n}_{[1,k-j]}
\xi^{C_n}_{[1,k-1]}, \mbox{ Definition \ref{def:bksymplectic}}  \\
&=E^{-1}\xi^{A_{2n-1}}_{[1,k-1]} \xi^{A_{2n-1}}_{[2n-k+1, 2n-1]}\xi^{A_{2n-1}}_{[1,k-j]} \xi^{A_{2n-1}}_{[2n-k+j, 2n-1]} \xi^{A_{2n-1}}_{[1,k-1]} \xi^{A_{2n-1}}_{[2n-k+1, 2n-1]}E, \\
&  \quad\mbox{ by Theorem \ref{goingandbacR1}}\\
&=E^{-1}\xi^{A_{2n-1}}_{[1,k-1]}\xi^{A_{2n-1}}_{[1,k-j]} \xi^{A_{2n-1}}_{[1,k-1]} \xi^{A_{2n-1}}_{[2n-k+1, 2n-1]} \xi^{A_{2n-1}}_{[2n-k+j, 2n-1]} \xi^{A_{2n-1}}_{[2n-k+1, 2n-1]}E\\
&=E^{-1}\xi^{A_{2n-1}}_{[j,k-1]}{\xi}_{[2n-k+1,2n-j]}^{A_{2n-1}} E, \\
& \quad\mbox{ by Remark \ref{re:Abenderknuth10} \eqref{bkschutz},  Remark \ref{re:dualgenerator}     \eqref{tildeevacBKlast} in type $A_{2n-1}$}\\
&=E^{-1}q_{[j,k-1]}q_{[2n-k+1,2n-j]} E, \mbox{ by Remark \ref{re:Abenderknuth10}  \eqref{bkschutz} in type $A_{2n-1}$}\\
&=E^{-1}q_{[j,k-1]}\tilde q_{[j,k-1]} E, \mbox{ by \eqref{tildeevacBKlast} in type $A_{2n-1}$}\\
&=\xi^{C_n}_{[j,k-1]},
\:\mbox{ by Theorem \ref{goingandbacR1}}.
\end{align*}
\end{remark}

Thanks to Theorem \ref{cactusg}, \eqref{cactusCa} and \eqref{cactusC},
one has therefore that $\gls{BKc}$ is a quotient of $\gls{Jsp}$. The generators of $\gls{BKc}$ satisfy the cactus group $\gls{Jsp}$ relations.

\begin{theorem}\label{symplecticbk}
The  following  is a group epimorphism from $\gls{Jsp}$ to $\gls{BKc}$:
$$s_{[1,j]} \mapsto q_{[1,j]}^{C_n}, \; 1\le j<n,\qquad s_{[j,n]}\mapsto q_{[j,n]}^{C_n},\; 1\le j\le n.$$
Therefore, $\gls{BKc}$ is isomorphic to a quotient of $\gls{Jsp}$.
\end{theorem}

We next define symplectic Bender--Knuth involutions $t_i^{C_n}$, $1\le i\le 2n-1$, on straight shaped KN tableaux that in turn generate
$\gls{BKc}$.

\begin{definition}\label{def:symplecticbenderknuth} The $2n-1$ symplectic Bender--Knuth involutions $t_i^{C_n}$ on KN tableaux of straight
shape on the alphabet
$\gls{Corderedletters}$ are defined by
\begin{subequations}
\begin{align}
\label{bkbk1}
 {t}^{C_n}_i &:= {q}^{C_n}_{[1,i-1]} {q}_{[1,i]}^{C_n}{q}^{C_n}_{[1,i-1]}
{q}_{[1,i-2]}^{C_n},  1\le i\leq n-1,\\
\label{symplecticbenderknuth}
{t}^{C_n}_{n-1+i}& := {q}_{[n-i+1,n]}^{C_n}{q}_{[n-i+2,n]}^{C_n}, \qquad \;\; 1\le i\leq n.
\end{align}
\end{subequations}
\end{definition}

Thanks to the $\gls{Jsp}$ relations satisfied by the generators of $\gls{BKc}$,
\begin{align}q^{C_n}_{[j,j]} =q^{C_n}_{[1,j]} q^{C_n}_{[1,1]}
q^{C_n}_{[1,j]}=q^{C_n}_{[1,j]} t^{C_n}_1 q^{C_n}_{[1,j]}\label{eq:qxi}\end{align}

\noindent
(Definition \ref{def:bksymplectic} with $j=k-1$) computes the symplectic
crystal reflection operator $\xi_j^{C_n}$, for $1\le j\le n$, on KN tableaux (see Proposition \ref{weylaction}  \eqref{refection c}).

 \begin{remark}\label{weightBnaction}  The symplectic Bender--Knuth involutions $t_i^{C_n}$, $1\le i\le n$, act on the
 elements
 of the  set $ \gls{symptab}$ such that $\textsf{wt}(t_i^{C_n}(T))=\textsf{wt}(\xi_i^{C_n}(T))=r_i.\textsf{wt}(T)$,  $r_{i}\in B_{n}$, $1\le i\le
 n$. This induces an action of the Weyl group $\gls{bn}$ on the weights in $\mathbb{Z}^n$, although, as we shall see,  in Subsection \ref{re:nobnactionkn2}, they
 do
 not define an action of the  Weyl group $\gls{bn}$ on  the set $ \gls{symptab}$. Let $T\in \gls{symptab}$  and
 $\textsf{wt}(T)=(v_1,\dots,v_n)$
$\in \mathbb{Z}^n$, then
\begin{align*}
\textsf{wt}(t_i^{C_n}(T))&=r_i\textsf{wt}(T), \;1\le i<n,\\
\textsf{wt}({t}^{C_n}_{n} (T))
&=(v_1,\dots,-v_n)=r_n\textsf{wt}(T),\\
\textsf{wt}({t}^{C_n}_{2n-i} (T))
&=(v_1,\dots,-v_{i},\dots,v_n)=r_{n-1}\cdots r_{n-i}r_nr_{n-i}\cdots
r_{n-1}(v_1,\dots,v_n)\\
&=\textsf{wt}(t^{C_n}_{n-1}\cdots t^{C_n}_{n-i}t^{C_n}_nt^{C_n}_{n-i}\cdots t^{C_n}_{n-1}(T)), \quad
1\le i< n.\\
\end{align*}
\end{remark}

\begin{proposition}\label{prop:orcadolfin}
 The symplectic Bender--Knuth involutions $t_i^{C_n}$, $1\le i\le 2n-1$, generate $\gls{BKc}$. In particular,
 \begin{enumerate}
\item
$q^{C_n}_{[1,i]}=p^{C_n}_1p^{C_n}_2\cdots p^{C_n}_i,$ $\;1\le i< n,$ and
\item  $q^{C_n}_{[n-i,n]}=t^{C_n}_{n+i}\cdots t^{C_n}_n,$ $0\le i\le n-1,$
 \end{enumerate}
 where $p_i^{C_n}:=t^{C_n}_i\cdots t^{C_n}_2t^{C_n}_1$ is the symplectic promotion, $1\le i\le 2n-1$.
\end{proposition}

\begin{proof}
$(1)$ We show by induction on $i$ that ${q}^{C_n}_{[1,i]}={q}^{C_n}_{[1,i-1]}p_i^{C_n}$. Note that $
{q}^{C_n}_{[1,1]}=p_1^{C_n}={t}_1^{C_n}.$
Furthermore, for $i>1$, ${q}_{[1,i]}^{C_n}= {q}^{C_n}_{[1,i-1]}{t}^{C_n}_i {q}^{C_n}_{[1,i-2]} {q}_{[1,i-1]}^{C_n}$.
Assuming that for some fixed positive integer $k$, ${q}^{C_n}_{[1,j]}={q}^{C_n}_{[1,j-1]}p_i^{C_n}$ for all $j\in [1,k-1]$, our
inductive
hypothesis implies

\begin{align*}
{q}_{[1,k]}^{C_n}&= {q}^{C_n}_{[1,k-1]}{t}^{C_n}_k {q}^{C_n}_{[1,k-2]} {q}_{[1,k-1]}^{C_n}\\
                          &={q}^{C_n}_{[1,k-1]}{t}^{C_n}_k {q}^{C_n}_{[1,k-2]}{q}_{[1,k-2]}^{C_n}p_{k-1}^{C_n} \\
                          &={q}^{C_n}_{[1,k-1]}{t}^{C_n}_k p_{k-1}^{C_n} \\
                          &={q}^{C_n}_{[1,k-1]} p_{k}^{C_n}.
\end{align*}

$(2)$ We proceed by induction on $i$. As a base case,
when $i = 0$, we have ${t}^{C_n}_{n} = {q}_{[n,n]}^{C_n}$. As an inductive step, we assume the statement is true for all $j\in [0, k-1]$ for
some
fixed positive integer $k<n-1$, so
\begin{align*}
{t}^{C_n}_{n+k}= {q}_{[n-k,n]}^{C_n}{q}_{[n-(k-1),n]}^{C_n} & \Rightarrow {t}^{C_n}_{n+k} {q}_{[n-(k-1),n]}^{C_n}= {q}_{[n-k,n]}^{C_n} \\
                                                                                              &\Rightarrow {q}_{[n-k,n]}^{C_n}  =t^{C_n}_{n+k}t^{C_n}_{n+k-1}\cdots t^{C_n}_n.
\end{align*}
\end{proof}

 Henceforth, $t_i$ and $q_{[i,j]}$ will be denoted in 
 $\mathcal{BK}_{n}$ by $t^{A_{n-1}}_
i$  and
$q^{A_{n-1}}_{
[i,j]}$ to distinguish from the corresponding symplectic involutions.
By Theorem \ref{th:jnbk}, the involutions $q^{A_{2n-1}}_{[i,j]}\in \mathcal{BK}_{2n}$, $ 1\le i\le j<2n$, satisfy the cactus $J_{2n}$
relations. Consider in $\mathcal{BK}_{2n}$ the involution  $q^{A_{2n-1}}_{[1,i]}$ with its dual $\tilde
q^{A_{2n-1}}_{[1,i]}:=q^{A_{2n-1}}_{[{2n-i},2n-1]}$, for $1\le i<n$, (Remark
 \ref{re:dualgenerator}, \eqref{eq:tildeq}), and $q^{A_{2n-1}}_{[i,2n-i]}$, $1\le i\le n$. Indeed, \begin{align*}q^{A_{2n-1}}_{[1,i]}\tilde q^{A_{2n-1}}_{[1,j]}
=\tilde q^{A_{2n-1}}_{[1,j]}q^{A_{2n-1}}_{[1,i]},\; \mbox{ $1\le i, j<n$}.\end{align*}

\begin{definition}\label{def:virtualbk}The \emph{virtual symplectic Berenstein--Kirillov group} $\gls{vBK2n}$ is the subgroup of
$\mathcal{BK}_{2n}$ generated by the  $2n-1$
involutions
\begin{subequations}
\begin{align} q_{[1,i]\sqcup [{2n-i},2n-1]}^{A_{2n-1}}
&:=q^{A_{2n-1}}_{[1,i]}\tilde q^{A_{2n-1}}_{[1,i]}
=\tilde q^{A_{2n-1}}_{[1,i]}q^{A_{2n-1}}_{[1,i]}\label{bkcvtilde1}\\
 &=
q^{A_{2n-1}}_{[1,i]}q^{A_{2n-1}}_{[{2n-i},2n-1]}
=q^{A_{2n-1}}_{[{2n-i},2n-1]}q^{A_{2n-1}}_{[1,i]}, \;\; \;1\le i<n,\nonumber\\
 q^{A_{2n-1}}_{[i,2n-i]}&{},\qquad\qquad \qquad\qquad\qquad\qquad\qquad\quad\quad\quad\;\; \;\mbox{  $  1\le i\le n$}, \label{bkcvtilde2}
\end{align}
\end{subequations}
modulo the relations they satisfy when acting on semi-standard tableaux of any straight shape.
\end{definition}

By Theorem \ref{th:schutzunion}, $q_{[1,i]\sqcup [{2n-i},2n-1]}^{A_{2n-1}}$ coincides with $\xi_{[1,i]\sqcup [{2n-i},2n-1]}^{A_{2n-1}}$
on  semi-standard tableaux of any straight shape, $1\le i<n$. (In particular, in $\gls{E}(\gls{symptab})$, for any partition
$\lambda$ with at most $n$ parts.)

 Bender--Knuth involutions $t^{A_{2n-1}}_i$ and dual Bender--Knuth involutions $\tilde t^{A_{2n-1}}_{2n-j}$, for $1\le i,j<n$, in $ \mathcal{BK}_{2n}$ commute as the next lemma shows. 
\begin{lemma}\label{benderknuthcommute}  For $1\le i,j<n$,  $
    (t_i^{A_{2n-1}}\tilde t^{A_{2n-1}}_{2n-j})^2=
    1.$ 

\end{lemma}
\begin{proof} By Proposition \ref{prop:ti_si},  Remark \ref{re:dualgenerator}, \eqref{eq:tildet2}, in
${\mathcal{BK}}_{2n}$, Lemma \ref{Acact}, 2A, and Theorem \ref{th:jnbk}, it follows
\begin{align*}t_i^{A_{2n-1}}\tilde t^{A_{2n-1}}_{2n-j}&={q}^{A_{2n-1}}_{[1,i-1]} {q}^{A_{2n-1}}_{[1,i]}{q}^{A_{2n-1}}_{[1,i-1]}
{q}^{A_{2n-1}}_{[1,i-2]}
{q}^{A_{2n-1}}_{[2n-j+1,2n-1]} {q}^{A_{2n-1}}_{[2n-j,2n-1]}\nonumber\\
&\qquad\qquad \qquad\qquad\qquad\qquad
\qquad\qquad \quad {q}^{A_{2n-1}}_{[2n-j+1,2n-1]} {q}^{A_{2n-1}}_{[2n-j+2,2n-1]}\nonumber\\
&={q}^{A_{2n-1}}_{[2n-j+1,2n-1]} {q}^{A_{2n-1}}_{[2n-j,2n-1]}{q}^{A_{2n-1}}_{[2n-j+1,2n-1]} {q}^{A_{2n-1}}_{[1,i-1]}
{q}^{A_{2n-1}}_{[1,i]}{q}_{[1,i-1]} {q}^{A_{2n-1}}_{[1,i-2]} \nonumber\\
&     =\tilde t^{A_{2n-1}}_{2n-j}t_i^{A_{2n-1}},\; 1\le i, j<n.
\end{align*}
    \end{proof}

\begin{proposition} \label{prop:virtualbender-knuth} For $1\le i<n$, consider the Bender--Knuth
 involution $t_i^{A_{2n-1}}$ with its dual $\tilde t^{A_{2n-1}}_{2n-i}$ in $ \mathcal{BK}_{2n}$. The group $\gls{vBK2n}$ also has
 the $2n-1$ involution generators

\begin{align}
 t^{A_{2n-1}}_{\{i\}\sqcup\{2n-i\}}&:=t_i^{A_{2n-1}}\tilde t^{A_{2n-1}}_{2n-i}=\tilde t^{A_{2n-1}}_{2n-i}t_i^{A_{2n-1}}, 1\le
i<n,\label{bkcvtilde3}\\
\label{bkcvtilde4}
t^{A_{2n-1}}_{[n-i+1,n+i]}&:=q^{A_{2n-1}}_{[n-i+1,n+i-1]}{q}_{[n-i+2,n+i-2]}^{A_{2n-1}}\nonumber\\
& \\
&={q}_{[n-i+2,n+i-2]}^{A_{2n-1}}q^{A_{2n-1}}_{[n-i+1,n+i-1]}, 1\le i\le n.\nonumber
\end{align}

\noindent where ${q}_{[n+1,n-1]}^{A_{2n-1}}:=1$. We call them the virtual symplectic Bender--Knuth involutions.
\end{proposition}
\begin{proof} The group ${\mathcal{BK}}_{2n}$ satisfies the $J_{2n}$ relations and $\gls{vBK2n}\subseteq
\mathcal{BK}_{2n}$.
Hence,
by Lemma \ref{Acact}, $\mathit{3A.}$, and Theorem \ref{th:jnbk}, one finds 
\[ {q}_{[n-i+1,n+i-1]}^{A_{2n-1}}{q}_{[n-i+2,n+i-2]}^{A_{2n-1}} ={q}_{[n-i+2,n+i-2]}^{A_{2n-1}}{q}_{[n-i+1,n+i-1]}^{A_{2n-1}}, 1\le i\le
n.\]

The identity $t_i^{A_{2n-1}}\tilde t^{A_{2n-1}}_{2n-i}=\tilde t^{A_{2n-1}}_{2n-i}t_i^{A_{2n-1}},$ $1\le i<n$, follows from Lemma \ref{benderknuthcommute} with $i=j$.
Considering Definition \ref{def:virtualbk} \eqref{bkcvtilde1}, for $1\le i<n$,
\begin{align*} q_{[1,i]\sqcup [{2n-i},2n-1]}^{A_{2n-1}}&=q^{A_{2n-1}}_{[1,i]}
q^{A_{2n-1}}_{[{2n-i},2n-1]}\\
&=p_1^{A_{2n-1}}\cdots p_i^{A_{2n-1}}q^{A_{2n-1}}_{[
1,2n-1]}p_1^{A_{2n-1}}\cdots p_i^{A_{2n-1}}q^{A_{2n-1}}_{[1,2n-1]}, \mbox{ by   \eqref{evacBK}, \eqref{evacBK2}}   \\
&=p_1^{A_{2n-1}}\cdots p_i^{A_{2n-1}}\tilde p_{2n-1}^{A_{2n-1}}\cdots \tilde p_{2n-i}^{A_{2n-1}}, \;\mbox{ $q^{A_{2n-1}}_{[1,2n-1]}$ is an involution}\\
&=t^{A_{2n-1}}_1 { \tilde t^{A_{2n-1}}_{2n-1}}(t^{A_{2n-1}}_2 \tilde t^{A_{2n-1}}_{2n-2}{ t^{A_{2n-1}}_1 \tilde
t^{A_{2n-1}}_{2n-1}})\cdots(t^{A_{2n-1}}_i\tilde t^{A_{2n-1}}_{2n-i} \cdots \\
&
\qquad\qquad\qquad\qquad\qquad \cdots t^{A_{2n-1}}_2 \tilde t^{A_{2n-1}}_{2n-2}t^{A_{2n-1}}_1 { \tilde t^{A_{2n-1}}_{2n-1}}), \;\mbox{ by Lemma \ref{benderknuthcommute}}\\
&=t^{A_{2n-1}}_{\{1\}\sqcup\{2n-1\}}(t^{A_{2n-1}}_{\{2\}\sqcup\{2n-2\}}t^{A_{2n-1}}_{\{1\}\sqcup\{2n-1\}})\cdots(t^{A_{2n-1}}_{\{i\}\sqcup\{2n-i\}}\cdots\\
& \qquad\qquad\qquad\qquad\qquad \qquad\qquad\cdots t^{A_{2n-1}}_{\{2\}\sqcup\{2n-2\}}t^{A_{2n-1}}_{\{1\}\sqcup\{2n-1\}}), \mbox{  by \eqref{bkcvtilde3}}
\end{align*}
where $q^{A_{2n-1}}_{[1,i]}=p_1^{A_{2n-1}}\cdots p_i^{A_{2n-1}}$ with  $p_i^{A_{2n-1}}:=t^{A_{2n-1}}_i\cdots t^{A_{2n-1}}_2t^{A_{2n-1}}_1$,
and
\begin{align*}
\tilde p_{2n-i}^{A_{2n-1}}&:=q^{A_{2n-1}}_{[{1},2n-1]}p_i^{A_{2n-1}}q^{A_{2n-1}}_{[{1},2n-1]}\\
                                       &\;=\tilde t^{A_{2n-1}}_{2n-i}\cdots \tilde t^{A_{2n-1}}_{2n-2}{\tilde t^{A_{2n-1}}_{2n-1}}, 1\le i<n.
\end{align*}

On the other hand, considering Definition \ref{def:virtualbk}, \eqref{bkcvtilde2}, for $1\le i\le n$,
\begin{align*}{q}^{A_{2n-1}}_{[n-i+1,n+i-1]}&=
q^{A_{2n-1}}_{[n,n]}(q^{A_{2n-1}}_{[n,n]}{q}_{[n-1,n+1]}^{A_{2n-1}})({q}_{[n-1,n+1]}^{A_{2n-1}}
{q}_{[n-2,n+2]}^{A_{2n-1}})\cdots\\
&\qquad\qquad\cdots ({q}_{[n-(i-2),n+i-2]}^{A_{2n-1}}{q}_{[n-i+1,n+i-1]}^{A_{2n-1}}), \; \mbox{ $q^{A_{2n-1}}_{[n-i, n+i]}$ is an involution}\\
&=t^{A_{2n-1}}_{[n,n+1]}t^{A_{2n-1}}_{[n-1,n+2]}t^{A_{2n-1}}_{[n-2,n+3]} \cdots t^{A_{2n-1}}_{[n-i+2,n+i-1]}t^{A_{2n-1}}_{[n-i+1,n+i]},\;  \mbox{ by  \eqref{bkcvtilde4}}.
\end{align*}
\end{proof}

\begin{remark} \label{re:weights} If $T$ is an $A_{2n-1}$ semi-standard tableau,
$$\textsf{wt}(t^{A_{2n-1}}_{\{i\}\sqcup\{2n-i\}}(T))=r_ir_{2n-i}.\textsf{wt}(T),$$ where $r_i=(i,\;i+1)$ and $r_{2n-i}=(2n-i,\;\,2n-i+1)$ are simple
transpositions in $\mathfrak{S}_{2n}$, for $1\le i<n$, and
$ \textsf{wt}(t_{[n-i+1,n+i]}(T))=(n-i+1,\;\,n+i)\textsf{wt}(T)$, where $(n-i+1,\;\,n+i)$ is the transposition of $\mathfrak{S}_{2n}$ that swaps
$n-i+1$ and $n+i$, for $1\le i\le n$.The virtual symplectic Bender--Knuth involutions $t^{A_{2n-1}}_{\{i\}\sqcup\{2n-i\}}$, $1\le i< n$, and
$t^{A_{2n-1}}_{[n,n+1]}$ act on  the elements
 in  $ \textsf{SSYT}(\lambda, 2n)$,  inducing an
 action of the Weyl group $\gls{bn}$, realized as $ \langle(i,\;i+1)(2n-i,\;\,2n-i+1),\; (n, n+1):\; 1\le i<n\rangle$, on  the weights  in $\mathbb{Z}^{2n}$.
\end{remark}

Thanks to Theorem \ref{cactusj2n}, we have that
$\gls{vBK2n}$ is a quotient of the virtual symplectic cactus $\gls{vJ2n}$.
The generators \eqref{bkcvtilde1} and \eqref{bkcvtilde2} of the group $\gls{vBK2n}$ satisfy the relations of the cactus group $\gls{vJ2n}$ or
equivalently those of the cactus group $\gls{Jsp}$.
\begin{theorem}\label{virtualsymplecticbk}
The  following  is a group epimorphism from $\gls{vJ2n}$ to $\gls{vBK2n}$:
\begin{align*}
\tilde s_{[1,j]\sqcup [2n-j,2n-1]}  &\mapsto q_{[1,j]\sqcup [2n-j,2n-1]}^{A_{2n-1}}, \; 1\le j<n,\\
\tilde s_{[j,2n-j]}&\mapsto
 q_{[j,2n-j]}^{A_{2n-1}},\; \;\qquad\quad  1\le j\le n.\end{align*} 
 The group $\gls{vBK2n}$ is isomorphic to a quotient of $\gls{vJ2n}$, and via the isomorphism between
$\gls{Jsp}$ and $\gls{vJ2n}$ defined via $s_{[1,j]} \mapsto \tilde s_{[1,j]\sqcup [2n-j,2n-1]}$, $1\le j<n$,
and $ s_{[j,n]}\mapsto \tilde s_{[j,2n-j]}$, $1\le j\le n$, is also isomorphic to a quotient of $\gls{Jsp}$.
\end{theorem}

Because the action of $\gls{vJ2n}$ on the set $\gls{ssytv}$ preserves the subset \linebreak $\gls{E}(\gls{symptab})$, {see  Remark  \ref{re:preserve}},
we now relate the virtual symplectic and symplectic Bender--Knuth involutions by embedding the crystal $\gls{symptab}$ into the crystal
$\gls{ssytv}$.
\begin{theorem}\label{thm:virtualbender-knuth2}
 The symplectic Bender--Knuth involutions $ t_i^{C_n}$, $1\le i\le 2n-1$, in $ \gls{BKc}$ can be realized by the virtual symplectic
 Bender--Knuth involutions
 $t_i^{A_{2n-i}}\tilde t^{A_{2n-i}}_{2n-i}$, $1\le i<n$, and $t^{A_{2n-1}}_{[n-i+1,n+i]}$, $1\le i\le n$,  in $\gls{vBK2n}$,
 and \emph{vice-versa},
\begin{align}
\label{isobk1}{t}_i^{C_n} &=E^{-1}t_i^{A_{2n-1}}{\tilde t_{\overline{i+1}}}^{A_{2n-1}}E=E^{-1}t_i^{A_{2n-1}}{\tilde
t_{2n-i}}^{A_{2n-1}}E,&1\le i<n,\\
\label{isobk2} t^{C_n}_{n+i-1} &= E^{-1}t^{A_{2n-1}}_{[n-i+1,n+i]}E,&1\le i\le n.
 \end{align}

\noindent In particular, the map $\gls{BKc} \rightarrow \gls{vBK2n}$ defined on generators by 

\begin{align}
\label{map:isobkvirtual}
t^{C_{n}}_{i}& \mapsto t^{A_{2n-1}}_{i}\tilde{t}^{A_{2n-i}}_{2n-i},&1\le i<n,\\
\label{map:isobkvirtualb}
t^{C_n}_{n+i-1}& \mapsto t^{A_{2n-1}}_{[n-i+1,n+i]},&1\le i\le n,
\end{align}

\noindent
is an isomorphism of groups. 

\end{theorem}

\begin{proof} 
For $1\le i<n$,
 \begin{align*}{t}^{C_n}_i& = {q}^{C_n}_{[1,i-1]} {q}_{[1,i]}^{C_n}{q}^{C_n}_{[1,i-1]} {q}_{[1,i-2]}^{C_n}, \,\mbox{by Definition \ref{def:symplecticbenderknuth}}\\
 &=E^{-1}(\xi^{A_{2n-1}}_{[1,i-1]}  \xi^{A_{2n-1}}_{[\overline{i},\bar 2]}) EE^{-1}(\xi^{A_{2n-1}}_{[1,i]}
 \xi^{A_{2n-1}}_{[\overline{i+1},\bar 2]}) EE^{-1}(\xi^{A_{2n-1}}_{[1,i-1]}   \xi^{A_{2n-1}}_{[\overline{i},\bar 2]})
 E\\
 &\qquad \qquad \qquad \qquad \qquad \qquad \qquad \qquad \qquad \quad\, E^{-1}(\xi^{A_{2n-1}}_{[1,i-2]}   \xi^{A_{2n-1}}_{[\overline{i-1},\bar 2]}) E, \\
 & \qquad \qquad  \qquad \qquad \qquad \qquad \qquad \mbox{ by Definition \ref{def:bksymplectic} and Theorem \ref{goingandbacR1} }\\
 &=E^{-1}(\xi^{A_{2n-1}}_{[1,i-1]}\xi^{A_{2n-1}}_{[1,i]}\xi^{A_{2n-1}}_{[1,i-1]}\xi^{A_{2n-1}}_{[1,i-2]}
 \xi^{A_{2n-1}}_{[\overline{i},\bar
 2]} \xi^{A_{2n-1}}_{[\overline{i+1},\bar 2]}\xi^{A_{2n-1}}_{[\overline{i},\bar 2]}\xi^{A_{2n-1}}_{[\overline{i-1},\bar 2]}) E\\
 &=E^{-1}(t_i^{A_{2n-1}}   \tilde t_{\overline{i+1}}^{A_{2n-1}})E.
 \end{align*}

By Definition \ref{def:symplecticbenderknuth}, 
 $t^{C_n}_{n+i-1}= {q}_{[n-i+1,n]}^{C_n}{q}_{[n-i+2,n]}^{C_n}$, for $ 2\le i\leq n$, and 
$t_n^{C_n}= q_{[n,n]}^{C_n} =\xi_n^{C_n}$. Then the result follows from Theorem \ref{goingandbacR2}
and Proposition \ref{prop:virtualbender-knuth}.

Note that the map defined by (\ref{map:isobkvirtual}) and (\ref{map:isobkvirtualb}) is an isomorphism as a consequence, since the description (\ref{isobk1}), (\ref{isobk2}) of the generators of \gls{BKc} in terms of those of \gls{vBK2n} implies that both sets of generators will satisfy precisely the same relations.
\end{proof}

\subsection{Symplectic Bender--Knuth involutions and the character of a
KN tableau crystal}\label{re:nobnactionkn2}

In the $C_2$ Weyl group ${B}_2=\langle r_1,r_2: r_i^2=1, (r_1r_2)^4=1\rangle$ with long element $r_2r_1r_2r_1$, the $C_2$ symplectic Bender--Knuth
involutions are $t^{C_2}_1=\xi_1^{C_2}$, $t^{C_2}_2=\xi_2^{C_2}$, $t^{C_2}_3=\xi_{[1,2]}^{C_2}\xi_{2}^{C_2}=\xi_2^{C_2}\xi_{[1,2]}^{C_2}$;
therefore, in this case, $t^{C_2}_1$ and $t^{C_2}_2$ define an action of the Weyl group $B_2$ on $\textsf{KN}(\lambda,2)$. However, in general,
for $n\ge 3$, the symplectic Bender--Knuth involutions $t_1^{C_n},\dots,t_n^{C_n}$ do not define an action of the Weyl group $\gls{bn}$ on the
set $\gls{symptab}$. Recall that the first $n-1$ generators  of the Weyl group $\gls{bn}$ satisfy the braid
relations \eqref{hyper3} of $\mathfrak{S}_n$, but we claim that, in general, $t^{C_n}_it^{C_n}_{i+1}t^{C_n}_i\neq
t^{C_n}_{i+1}t^{C_n}_it^{C_n}_{i+1}$,
i.e., $(t^{C_n}_it^{C_n}_{i+1})^3 \neq 1$ for $1\le i<n$. To show this inequality, note that by Theorem \ref{thm:virtualbender-knuth2}, it is
enough to consider the virtual symplectic Bender--Knuth
involutions and the corresponding virtual inequality
\begin{align}\label{ineqlion}t_i^{A_{2n-1}}{\tilde
    t_{2n-i}}^{A_{2n-1}}t_{i+1}^{A_{2n-1}}{\tilde
    t_{2n-(i+1)}}^{A_{2n-1}}t_i^{A_{2n-1}}{\tilde
    t_{2n-i}}^{A_{2n-1}}\neq  t_{i+1}^{A_{2n-1}}{\tilde
    t_{2n-(i+1)}}^{A_{2n-1}}t_i^{A_{2n-1}}{\tilde
    t_{2n-i}}^{A_{2n-1}}t_{i+1}^{A_{2n-1}}{\tilde
    t_{2n-(i+1)}}^{A_{2n-1}}.
    \end{align}

\noindent From Proposition \ref{prop:ti_si} and Remark \ref{re:dualgenerator},
$t_i^{A_{2n-1}}{\tilde t_{2n-(i+1)}}^{A_{2n-1}}=
{\tilde
t_{2n-(i+1)}}^{A_{2n-1}}t_i^{A_{2n-1}},$ for $1\le i<n$. If we had equality in \eqref{ineqlion}, then
$${\tilde t_{2n-i}}^{A_{2n-1}}{\tilde t_{2n-(i+1)}}^{A_{2n-1}}{\tilde
t_{2n-i}}^{A_{2n-1}}t_i^{A_{2n-1}}t_{i+1}^{A_{2n-1}}t_i^{A_{2n-1}}={\tilde
    t_{2n-(i+1)}}^{A_{2n-1}}{\tilde
    t_{2n-i}}^{A_{2n-1}}{\tilde
    t_{2n-(i+1)}}^{A_{2n-1}}t_{i+1}^{A_{2n-1}}t_i^{A_{2n-1}}t_{i+1}^{A_{2n-1}}$$
  $$\Leftrightarrow ({\tilde t_{2n-(i+1)}}^{A_{2n-1}}
    {\tilde t_{2n-i}}^{A_{2n-1}})^{3}=(t_i^{A_{2n-1}}t_{i+1}^{A_{2n-1}})^{3}.$$
Applying this identity to the $A_{9}$ tableau $E(T)=(P^+,P^-)$ in the virtualization Example \ref{ex:virt} would imply that

\begin{align}
({\tilde t_{\bar 6}}^{A_{11}}{\tilde t_{\bar 5}}^{A_{11}})^3(E(T))&=(t_4^{A_{11}}t_5^{A_{11}})^3 (E(T))  \nonumber\\
\label{tilde} &\Leftrightarrow  \\
 (P^+,({\tilde t_{\bar 6}}^{A_{11}}{\tilde t_{\bar 5}}^{A_{11}})^3 (P^-)) &=((t_4^{A_{11}}t_5^{A_{11}})^3 (P^+),P^-), \nonumber
\end{align}

\noindent but this is impossible, as $(t_2^{A_{9}}t_3^{A_{9}})^3 (P^+)\neq P^+$.  Note that the LHS of \eqref{tilde} follows from $\tilde t_{2n-i}(P^+,P^-)=\textsf{evac}\, t_i\textsf{evac}(P^+,P^-)$ and  Remark \ref{re:tildet}.

Though the symplectic Bender--Knuth involutions do not define an action of $\gls{bn}$ on $\gls{symptab}$, they can be used to show that the
character of the crystal $\gls{symptab}$ is a symmetric Laurent polynomial with respect to the action of $\gls{bn}$. Let
$\mathcal{E}:=\mathbb{Z}[x_1^\pm,\dots,x_n^\pm]$ be the ring of Laurent polynomials in $n$ variables over $\mathbb{Z}$, and let
$\mathcal{E}^{\gls{bn}}=\{f\in \mathcal{E}: r_i.f=f, \, r_i\in \gls{bn},\, 1\le i\le n\}$, where $ r_i.x^\alpha:=x^{r_i.\alpha}$, for
$x^\alpha:=x_1^{\alpha_1}\cdots x_n^{\alpha_n}$,  $\alpha\in\mathbb{Z}^n$ and $r_i\in \gls{bn}$, be the subring of symmetric Laurent
polynomials.

The character of $\gls{symptab}$ is the symplectic Schur function $sp_\lambda(x)$ in the sequence of variables $x=(x_1,\dots,x_n)$. Thanks
to
Remark \ref{weightBnaction}, $\textsf{wt}(t_i^{C_n}.b)=r_i.\textsf{wt}(b)$ for any $b\in \gls{symptab}$ and $1\le i\le n$.  Therefore, since
$t_i^{C_n}$, $1\le i\le n$, is an involution on the set $\gls{symptab}$, we obtain a proof that $sp_\lambda(x)$ is a symmetric Laurent
polynomial:

\begin{align*}
sp_{\lambda}(x)&=\sum_{b\in\gls{symptab}}x^{\textsf{wt}(b)}
=\sum_{b\in\gls{symptab}}x^{\textsf{wt}(t_i^{C_n}b)},\,1\le
i\le n,\\
&=\sum_{b\in\gls{symptab}}x^{r_i.\textsf{wt}(b)}=sp_\lambda(r_i.x),\,1\le
i\le n.
\end{align*}

\subsection{Relations for the symplectic Berenstein--Kirillov group}

Thanks to Theorem \ref{symplecticbk} and Theorem
\ref{virtualsymplecticbk}, we now provide the following relations for $\gls{BKc}$ equivalently $\gls{vBK2n}$.
  The relations \eqref{eq:symplecticbkrelationspecial} and \eqref{eq:symplecticbkrelationspecialbn} below  are the only ones known for
  $\gls{BKc}$,  equivalently   $\gls{vBK2n}$, which do not follow from the cactus group $\gls{Jsp}$ relations equivalently
  the
  virtual cactus group  $\gls{vJ2n}$
  relations (see also Remark \ref{re:braid}).

\begin{proposition} \label{prop:bkCrelations}  The symplectic Bender--Knuth involutions $ t_i^{C_n}
$, $i=1,\dots, 2n-1$,
satisfy the following relations:
\label{rels:bkc}
\begin{enumerate}
\item $(t_i^{C_n})^2=1$, $i=1,\dots, 2n-1$.
\item $({t}^{C_n}_{n+i-1}{t}^{C_n}_{n+j-1})^2=1$,  $1\le i,j\le n$.
\item  $(t^{C_n}_it^{C_n}_j)^2=1$, $|i-j|>1$, $1\le i, j<n$.
\item $(t_i^{C_n}t_{n+j-1}^{C_n})^2=1$, $i<n-j$. 
\item $(t^{C_n}_iq_{[j,k-1]}^{C_n})^2=1,$ $i+1<j<k\le n$.
\label{rels:bkc5}
\item $(t^{C_n}_iq_{[j,n]}^{C_n})^2=1,$ $i+1<j\le n$.
\item $(t^{C_n}_{n+i-1}q_{[j,n]}^{C_n})^2=1,$ $1\le i, j\le n$.
\label{rels:bkc7}
\item $(t^{C_n}_{n+i-1}q_{[j,k-1]}^{C_n})^2=1,$ $n-i+1<j<k\le n$.
\item  $(t_1^{C_n}t_2^{C_n})^6=1$, $ n\ge 3$. \label{eq:symplecticbkrelationspecial}
\item $(t^{C_n}_{n-1}\cdots t^{C_n}_2t^{C_n}_1t^{C_n}_2\cdots t^{C_n}_{n-1}t^{C_n}_n)^4=1$.  \label{eq:symplecticbkrelationspecialbn}
\end{enumerate}
The virtual symplectic Bender--Knuth involutions  $t_i^{A_{2n-i}}\tilde t^{A_{2n-i}}_{2n-i}=E{t}_i^{C_n}E^{-1}$, $1\le i<n$, and
$t^{A_{2n-1}}_{[n-i+1,n+i]}=Et^{C_{n}}_{n-i+1}E^{-1}$, $1\le i\le n$,
in $\gls{vBK2n}$ satisfy the same relations as those of $\gls{BKc}$ by replacing $t^{C_n}_i$ by
$t_i^{A_{2n-i}}\tilde t^{A_{2n-i}}_{2n-i}$, $1\le i<n$,
and $t^{C_n}_{n+i-1}$ by $t^{A_{2n-1}}_{[n-i+1,n+i]}$, $1\le i\le n$.

\end{proposition}

\begin{proof} Recall Definition \ref{def:bksymplectic}, Definition \ref{def:symplecticbenderknuth},   Theorem \ref{symplecticbk},  Theorem \ref{virtualsymplecticbk} and Theorem \ref{thm:virtualbender-knuth2}.

$(1)$  $(t_1^{C_n})^2=(q_{[1,1]}^{C_n})^2=1$,  $(t_n^{C_n})^2=(q_{[1,n]}^{C_n})^2=1$. For $ 2\le i\leq n$, \begin{align*}
    ({t}^{C_n}_{n-1+i} )^2=
({q}_{[n-i+1,n]}^{C_n}{q}_{[n-i+2,n]}^{C_n})^2=1\end{align*} is equivalent to the $\gls{Jsp}$  relation $\mathit{3C. (i)}$ of Lemma \ref{symplecticact}.

For $2\le i\leq n-1,$ $$({t}^{C_n}_i)^2 = {q}_{[1,i-1]}^{C_n} {q}_{[1,i]}^{C_n}{q}_{[1,i-1]}^{C_n} {q}_{[1,i-2]}^{C_n}{q}_{[1,i-1]}^{C_n}
{q}_{[1,i]}^{C_n}{q}_{[1,i-1]}^{C_n} {q}_{[1,i-2]}^{C_n}=1$$ follows from the cactus $\gls{Jsp}$  relation $\mathit{3C. (ii)}$  of Lemma \ref{symplecticact}, and
the observations that
\begin{align*}q_{[1,i-1]}^{C_n} {q}_{[1,i-2]}^{C_n}&=q_{[2,i-1]}^{C_n} {q}_{[1,i-1]}^{C_n},\\
q_{[1,i-1]}^{C_n} {q}_{[2,i-1]}^{C_n}&=q_{[1,i-2]}^{C_n}
{q}_{[1,i-1]}^{C_n},\\
q_{[1,i]}^{C_n} {q}_{[2,i-1]}^{C_n}&=q_{[2,i-1]}^{C_n} {q}_{[1,i]}^{C_n}.
\end{align*}

$(2)$ Let $i\ne j$. From  Lemma \ref{symplecticact} $\mathit{3C. (i)}$,
\begin{align*}{t}^{C_n}_{n+i-1}{t}^{C_n}_{n+j-1}&={q}_{[n-i+1,n]}^{C_n}{q}_{[n-i+2,n]}^{C_n}{q}_{[n-j+1,n]}^{C_n}{q}_{[n-j+2,n]}^{C_n}\\
&={q}_{[n-j+1,n]}^{C_n}{q}_{[n-j+2,n]}^{C_n}{q}_{[n-i+1,n]}^{C_n}{q}_{[n-i+2,n]}^{C_n}\\
&={t}^{C_n}_{n+j-1}{t}^{C_n}_{n+i-1}.
\end{align*}

$(3)$ Recall Theorem \ref{thm:virtualbender-knuth2}, equation \eqref{eq:bkrelations1}, Remark \ref{re:dualgenerator} \eqref{eq:bkrelations1tilde}, and Lemma \ref{benderknuthcommute}. Then
\begin{align*}
(t_i^{C_n}t_j^{C_n})^2 & =(E^{-1}t^{A_{2n-1}}_i\tilde t^{A_{2n-1}}_{2n-i}EE^{-1}t^{A_{2n-1}}_j\tilde t^{A_{2n-1}}_{2n-j}E)^2 \\
                                     & =E^{-1}(t^{A_{2n-1}}_it^{A_{2n-1}}_j)^2(\tilde t^{A_{2n-1}}_{2n-i}\tilde t^{A_{2n-1}}_{2n-j})^2E=1,
                                     \;\text{for} \,|i-j|>1, \, 1\le i,j\le  n-1.
\end{align*}

$(4)$
For $i<n-j$, due to the $\gls{Jsp}$ relation $2C$.
\begin{align}
t_i^{C_n}t_{n+j-1}^{C_n}&= q^{C_n}_{[1,i-1]}q^{C_n}_{[1,i]}q^{C_n}_{[1,i-1]}q^{C_n}_{[1,i-2]}q^{C_n}_{[n-j+1,n]}q^{C_n}_{[n-j+2,n]}\nonumber\\
&= q^{C_n}_{[n-j+1,n]}q^{C_n}_{[n-j+2,n]}q^{C_n}_{[1,i-1]}q^{C_n}_{[1,i]}q^{C_n}_{[1,i-1]}q^{C_n}_{[1,i-2]}\nonumber \\
&=t_{n+j-1}^{C_n}t_i^{C_n}.\nonumber
\end{align}

$(5)$ Recall Theorem \ref{goingandbacR1},  Theorem \ref{thm:virtualbender-knuth2}, Definition \ref{def:bksymplectic}, \eqref{eq:bkrelationextra}  and Remark \ref{re:dualgenerator} \eqref{eq:bkrelationextratilde}.
We first observe the following forms of \eqref{eq:bkrelationextra} respectively \eqref{eq:bkrelationextratilde} in type $A_{2n-1}$.  For  $i+1<j<k\le n$,
\begin{align}\label{eq:generalize59}
(t^{A_{2n-1}}_iq^{A_{2n-1}}_{[2n-k+1,2n-j]})^2&=1, 
\\
\label{eq:generalize69}
    (\tilde t^{A_{2n-1}}_{2n-i}q^{A_{2n-1}}_{[j,k-1]})^2&=1. 
\end{align}
Since $i+1<j<n<2n-k+1<2n-j+1<2n-i< 2n$,
\begin{align*}t^{A_{2n-1}}_iq^{A_{2n-1}}_{[2n-k+1,2n-j]}&=q^{A_{2n-1}}_{[2n-k+1,2n-j]}t^{A_{2n-1}}_i,\mbox{ by \eqref{eq:bkrelationextra} with $k:=2n-j+1$ and }   \\
&\mbox{  $j:=2n-k+1$, 
for $i+1<j:=2n-k+1<k:=2n-j+1$},\nonumber\\
    \end{align*}
 \noindent which proves \eqref{eq:generalize59}. 
Thus, the identity \eqref{eq:generalize69} is the dual version of \eqref{eq:generalize59} with $k:=2n-j+1$
 and $j:=2n-k+1$.

Henceforth, for $i+1<j<k\le n$,
\begin{align}(t^{C_n}_iq_{[j,k-1]}^{C_n})^2&=(E^{-1}t_i^{A_{2n-1}}{\tilde
t_{2n-i}}^{A_{2n-1}}EE^{-1}{\xi}_{[j,k-1]}^{A_{2n-1}}{\xi}_{[2n-k+1,2n-j]}^{A_{2n-1}}E)^2,\nonumber \\ &\qquad\mbox{ \; \;\;by    Theorem \ref{thm:virtualbender-knuth2} and  Remark     \ref{re:C16(2)}}\nonumber\\
&=E^{-1}(t_i^{A_{2n-1}}{\xi}_{[j,k-1]}^{A_{2n-1}}{\tilde
t_{2n-i}}^{A_{2n-1}}{\xi}_{[2n-k+1,2n-j]}^{A_{2n-1}})^2E,\; \mbox{ by \eqref{eq:generalize69}}\nonumber\\
&=E^{-1}t_i^{A_{2n-1}}{\xi}_{[j,k-1]}^{A_{2n-1}}{\tilde
t_{2n-i}}^{A_{2n-1}}{\xi}_{[2n-k+1,2n-j]}^{A_{2n-1}} \nonumber\\
&\qquad \qquad \qquad \qquad t_i^{A_{2n-1}}{\xi}_{[j,k-1]}^{A_{2n-1}}{\tilde
t_{2n-i}}^{A_{2n-1}}{\xi}_{[2n-k+1,2n-j]}^{A_{2n-1}}E \nonumber\\
& =E^{-1}t_i^{A_{2n-1}}{\xi}_{[j,k-1]}^{A_{2n-1}}t_i^{A_{2n-1}}{\tilde
t_{2n-i}}^{A_{2n-1}}{\xi}_{[2n-k+1,2n-j]}^{A_{2n-1}} \nonumber\\
&\qquad \qquad \qquad \qquad{\xi}_{[j,k-1]}^{A_{2n-1}}{\tilde
t_{2n-i}}^{A_{2n-1}}{\xi}_{[2n-k+1,2n-j]}^{A_{2n-1}}E\nonumber\\
&\qquad \mbox{ by \eqref{eq:generalize59}, Lemma \ref{benderknuthcommute}}\nonumber\\
&=E^{-1}t_i^{A_{2n-1}}{\xi}_{[j,k-1]}^{A_{2n-1}}t_i^{A_{2n-1}}{\xi}_{[j,k-1]}^{A_{2n-1}}{\tilde
t_{2n-i}}^{A_{2n-1}}{\xi}_{[2n-k+1,2n-j]}^{A_{2n-1}}\nonumber\\
&\qquad \qquad \qquad \qquad \qquad \qquad{\tilde t_{2n-i}}^{A_{2n-1}}{\xi}_{[2n-k+1,2n-j]}^{A_{2n-1}}E, \nonumber\\
& \qquad \mbox{ by  Lemma \ref{Acact}, $2.A$, $j\le k-1<n<2n-k+1 $, and \eqref{eq:generalize69}} \nonumber\\
&=E^{-1}(t_i^{A_{2n-1}}{\xi}_{[j,k-1]}^{A_{2n-1}})^2({\tilde t_{2n-i}}^{A_{2n-1}}{\xi}_{[2n-k+1,2n-j]}^{A_{2n-1}})^2E \nonumber\\
&=1, \mbox{ by \eqref{eq:bkrelationextra}, \eqref{eq:bkrelationextratilde}.}\nonumber
\end{align}

$(6)$  For $i+1<j\le n$, equivalently $2n-i-1>2n-j\ge
n>j-1$,
\begin{align}t^{C_n}_iq_{[j,n]}^{C_n}&=E^{-1}t_i^{A_{2n-1}}{\tilde t_{2n-i}}^{A_{2n-1}}{\xi}_{[j,2n-j]}^{A_{2n-1}}E, \;\;\text{ by Theorem \ref{goingandbacR2}}\nonumber\\
&=E^{-1}t_i^{A_{2n-1}}{\xi}_{[j,2n-j]}^{A_{2n-1}}{\tilde t_{2n-i}}^{A_{2n-1}}E,\mbox{ by \eqref{eq:bkrelationextratilde} with $k:=2n-j+1$}\nonumber\\
&\qquad\mbox{ and  $2n-i-1>2n-j>j-1$ }
\nonumber\\
&=E^{-1}{\xi}_{[j,2n-j]}^{A_{2n-1}}t_i^{A_{2n-1}}{\tilde t_{2n-i}}^{A_{2n-1}}E,\;\; \mbox{by \eqref{eq:bkrelationextra} with $k:=2n-j+1$}\nonumber\\
&\qquad\mbox{ and  \;$2n-j+1>j>i+1$}\nonumber\\
&=q_{[j,n]}^{C_n}t^{C_n}_i.\nonumber
\end{align}

$(7)$  We will prove $(t^{C_n}_{n+i-1}q_{[j,n]}^{C_n})^2=1,$ $1\le i, j\le n$.
Recall  Definition \ref{def:symplecticbenderknuth},  
  Theorem \ref{goingandbacR2}, and
Theorem \ref{thm:virtualbender-knuth2}   
  \eqref{isobk2}, 
\begin{align*}
t^{C_n}_{n+i-1}q_{[j,n]}^{C_n}&
=E^{-1}t^{A_{2n-1}}_{[n-i+1,n+i]}\xi_{[j,2n-j]}^{A_{2n-1}}E
\\
&=E^{-1} {\xi}_{[n-(i-1),n+(i-1)]}^{A_{2n-1}}{\xi}_{[n-(i-2),n+(i-2)]}^{A_{2n-1}}\xi_{[j,2n-j]}^{A_{2n-1}}E,
\;\mbox{ by \eqref{bkcvtilde4} }\\
& =E^{-1}\xi_{[j,2n-j]}^{A_{2n-1}}{\xi}_{[n-(i-1),n+(i-1)]}^{A_{2n-1}}{\xi}_{[n-(i-2),n+(i-2)]}^{A_{2n-1}}E,\;\\ ~&\qquad \mbox{ by Theorem
\ref{virtualsymplecticbk} and Lemma 
\ref{symmetrieslemma} \eqref{cactisymm1} }\nonumber\\
&=q_{[j,n]}^{C_n}t^{C_n}_{n+i-1}.
\end{align*}

$(8)$ We now prove  $ (t^{C_n}_{n+i-1}q_{[j,k-1]}^{C_n})^2=1,$ for $n-i+1<j<k\le n$. It follows from Theorem \ref{goingandbacR1}, Theorem \ref{thm:virtualbender-knuth2}, and Remark \ref{re:C16(2)},
\begin{align}
t^{C_n}_{n+i-1}q_{[j,k-1]}^{C_n}&=E^{-1}
{\xi}_{[n-(i-1),n+(i-1)]}^{A_{2n-1}}{\xi}_{[n-(i-2),n+(i-2)]}^{A_{2n-1}}\xi_{[j,k-1]}^{A_{2n-1}}\xi_{[2n-k+1, 2n-j]}^{A_{2n-1}}E\nonumber\\
&=E^{-1}
{\xi}_{[n-(i-1),n+(i-1)]}^{A_{2n-1}}\xi_{[2n-k+1,
2n-j]}^{A_{2n-1}}{\xi}_{[n-(i-2),n+(i-2)]}^{A_{2n-1}}\xi_{[2n-k+1, 2n-j]}^{A_{2n-1}}E,\nonumber\\
&\qquad\;
\mbox{ 
Lemma 
\ref{Acact} $\mathit{3A.}$ with $n-i+1<j<k\le n$}\nonumber\\
&= E^{-1} \xi_{[j
k-1]}^{A_{2n-1}}{\xi}_{[n-(i-1),n+(i-1)]}^{A_{2n-1}}\xi^{A_{2n-1}}_{[j,k-1]}{\xi}_{[n-(i-2),n+(i-2)]}^{A_{2n-1}}E, \nonumber\\
&\qquad\;
\mbox{ 
Lemma 
\ref{Acact} $\mathit{3A.}$ with $n+i-1>2n-j>2n-k\ge n$}\nonumber\\
&=E^{-1}\xi_{[j,k-1]}^{A_{2n-1}}\xi_{[2n-k+1,
2n-j]}^{A_{2n-1}}{\xi}_{[n-(i-1),n+(i-1)]}^{A_{2n-1}}{\xi}_{[n-(i-2),n+(i-2)]}^{A_{2n-1}}E,\; \nonumber\\
&\quad\quad\;\;
\mbox{ by Lemma 
\ref{Acact} $\mathit{3A.}$
}
\nonumber\\
&=q_{[j,k-1]}^{C_n}t^{C_n}_{n+i-1}.\nonumber
\end{align}

$(9)$ Recall  Theorem \ref{thm:virtualbender-knuth2}    \eqref{isobk1} and Remark \ref{re:dualgenerator}. Then, for $n\ge 3$,
\begin{align*}(t_1^{C_n}t_2^{C_n})^6&=E^{-1}(t_1^{A_{2n-1}}\tilde t_{\bar 2}^{A_{2n-1}}t_2^{A_{2n-1}}\tilde t_{\bar
3}^{A_{2n-1}})^6E\nonumber\\
&=E^{-1}(t_1^{A_{2n-1}}t_2^{A_{2n-1}}\tilde t_{\bar 2}^{A_{2n-1}}\tilde t_{\bar 3}^{A_{2n-1}})^6E, \;\mbox{ by Lemma \ref{benderknuthcommute}} \nonumber\\
&=E^{-1}(t_1^{A_{2n-1}}t_2^{A_{2n-1}})^6(\tilde t_{\bar 2}^{A_{2n-1}}\tilde t_{\bar 3}^{A_{2n-1}})^6E,\;\mbox{ by Lemma \ref{benderknuthcommute}}\nonumber\\
&=E^{-1}(\tilde t_{\bar 2}^{A_{2n-1}}\tilde t_{\bar 3}^{A_{2n-1}})^6E, \mbox{ by \eqref{eq:bkrelationspecial}}\\
&=E^{-1}(\tilde t_{2n-1}^{A_{2n-1}}\tilde t_{2n-2}^{A_{2n-1}})^6E
=1, \mbox{ by \eqref{eq:bkrelationspecialtilde}}.
\end{align*}

 $(10)$  From Proposition \ref{weylaction} \eqref{refection c}, we know that $(\xi_{n-1}^{C_n}\xi_{n}^{C_n})^4=1$. We prove that $(t^{C_n}_{n-1}\cdots t^{C_n}_2t^{C_n}_1t^{C_n}_2\cdots t^{C_n}_{n-1}t^{C_n}_n)^4=1$ is equivalent to $(\xi_{n-1}^{C_n}\xi_{n}^{C_n})^4=1$. 
Recall  Proposition
\ref{prop:orcadolfin}. Then, from Proposition
\ref{prop:orcadolfin} one has the following

\begin{align}
(\xi_{n-1}^{C_n}\xi_{n}^{C_n})^4&=(q_{[1,n-1]}^{C_n}t_1^{C_n}q^{C_n}_{[1,n-1]}t_n^{C_n})^4, \mbox{  by \eqref{eq:qxi}}\nonumber\\
\label{eq:turtle}&=(q_{[1,n-2]}^{C_n}p_{n-1}^{C_n}t_1^{C_n}q^{C_n}_{[1,n-2]}p_{n-1}^{C_n}t_n^{C_n})^4. \quad 
\end{align}

From  \cite[Proposition 1.4, (a)]{bk95} and mimicking its proof in conjunction with relation 
\eqref{rels:bkc5}
, we may write
\begin{align}q_{[1,n-1]}^{C_n}&=q_{[1,n-2]}^{C_n}p_{n-1}^{C_n}\nonumber\\
&=(p_{n-1}^{C_n})^{-1}q_{[1,n-2]}^{C_n}.\label{mimicbks}\end{align}
Therefore, using \eqref{mimicbks},
\begin{align*}
(\xi_{n-1}^{C_n}\xi_{n}^{C_n})^4&=(q_{[1,n-2]}^{C_n}p_{n-1}^{C_n}t_1^{C_n}(p_{n-1}^{C_n})^{-1}q^{C_n}_{[1,n-2]}t_n^{C_n})^4, \mbox{ by \eqref{eq:turtle}}\\
&=(q_{[1,n-2]}^{C_n}p_{n-1}^{C_n}t_1^{C_n}(p_{n-1}^{C_n})^{-1}t_n^{C_n}q^{C_n}_{[1,n-2]})^4, \mbox{ by relation \eqref{rels:bkc7}}\\
&=q_{[1,n-2]}^{C_n}(p_{n-1}^{C_n}t_1^{C_n}(p_{n-1}^{C_n})^{-1}t_n^{C_n})^4q^{C_n}_{[1,n-2]}\\
&=q_{[1,n-2]}^{C_n}(p_{n-1}^{C_n}t_1^{C_n}(p_{n}^{C_n})^{-1})^4q^{C_n}_{[1,n-2]}.\end{align*}

Since $(q_{[1,n-2]}^{C_n})^2=1$, we get \begin{align}(\xi_{n-1}^{C_n}\xi_{n}^{C_n})^4=1\Leftrightarrow (p_{n-1}^{C_n}t_1^{C_n}(p_{n}^{C_n})^{-1})^4=(t^{C_n}_{n-1}\cdots
t^{C_n}_2t^{C_n}_1t^{C_n}_2\cdots
t^{C_n}_{n})^4=1, \label{equivalentbraidb}\end{align} 
\noindent \mbox{  where $p_0^{C_n}:=1$.} In particular, for $n=2$,
$(t_1^{C_n}t_2^{C_n})^4=1$.
\end{proof}

\begin{remark}\label{re:braid}
\noindent $1.$ The relation \eqref{eq:symplecticbkrelationspecial},  Proposition \ref{prop:bkCrelations}, in $\gls{BKc}$, respectively
$$(t_1^{A_{2n-1}}t_2^{A_{2n-1}})^6(\tilde t_{ 2n-1}^{A_{2n-1}}\tilde t_{2n-2}^{A_{2n-1}})^6=1\;\text{ in $\gls{vBK2n}$},$$
is equivalent to the braid relations of $\gls{bn}$, $(\xi^{C_n}_i\xi^{C_n}_{i+1})^3=1$, for $1\le i<n-1$, respectively
$$(\xi^{A_{2n-1}}_i\xi^{A_{2n-1}}_{2n-i}\xi^{A_{2n-1}}_{i+1}\xi^{A_{2n-1}}_{2n-i-1})^3=(\xi^{A_{2n-1}}_i\xi^{A_{2n-1}}_{i+1})^3(\xi^{A_{2n-1}}_{2n-i}\xi^{A_{2n-1}}_{2n-i-1})^3=1,$$
the braid relations of $\mathfrak{S}_{2n}$, for $1\le i<n-1$, \cite[Proposition 1.4, (d)]{bk95}.

\medskip 
\noindent $2.$ The identity $(\xi_{n-1}^{C_n}\xi_{n}^{C_n})^4=1$ in the group $\gls{BKc}$ translates to the isomorphic group $\gls{vBK2n}$ as
    \begin{align}(\xi_{n-1}^{A_{2n-1}}\xi_{n+1}^{A_{2n-1}} \xi_{n}^{A_{2n-1}})^4=1.\label{relationten}\end{align} Thus relation \eqref{eq:symplecticbkrelationspecialbn},  Proposition \ref{prop:bkCrelations}, in
    $\gls{BKc}$  translates to $\gls{vBK2n}$ as
\begin{align}(t^{A_{2n-1}}_{n-1}\cdots t^{A_{2n-1}}_{1}\tilde t_{n+1}^{A_{2n-1}}\cdots \tilde t_{ 2n-1}^{A_{2n-1}}
\xi_{n}^{A_{2n-1}})^4=1.\label{bkbnaction}\end{align}
Recall  in Corollary \ref{virtWeylaction} we have seen that involutions
$
\xi^{A_{2n-1}}_i\xi^{A_{2n-1}}_{2n-i},\xi^{A_{2n-1}}_n, \,1\le i\le n-1,$ define an action of $B_n$ on the embedded crystal
$\gls{E}(\gls{symptab})$ in $\textsf{SSYT}(\lambda^A,2n)$.
Hence, as \eqref{equivalentbraidb} shows, the relation \eqref{eq:symplecticbkrelationspecialbn} in $\gls{BKc}$ respectively \eqref{bkbnaction} in $\gls{vBK2n}$ is equivalent to the braid relation $(\xi_{n-1}^{C_n}\xi_{n}^{C_n})^4=1$ of $B_n$ respectively the braid relation \eqref{relationten} of $B_n$ realized in    $\mathfrak{S}_{2n}$.
\end{remark}

\section{Open questions and final remarks}

Similarly to $\gls{BK}$, it remains to establish whether or not $\gls{BKc}$ satisfies additional relations besides those listed in
Proposition
\ref{prop:bkCrelations}.
Chmutov, Glick and Pylyavskyy \cite{cgp16} have determined relationships between subsets of relations in the groups $\gls{BKn}$ and
$\gls{Jn}$
which yield a presentation for the cactus group $\gls{Jn}$ in terms of Bender--Knuth generators. Rodrigues \cite{ro20, ro21, ro23} has also
introduced a shifted Berenstein--Kirillov group with many parallels with the original $\gls{BK}$ group. Following Halacheva she has defined
a
cactus group action of $\gls{Jn}$ via partial shifted Sch\"utzenberger--Lusztig involutions (partial shifted reversal) on the
 Gillespie--Levinson--Purbhoo shifted tableau crystal \cite{glp}. On the other hand, with the shifted tableau switching she has defined
shifted
Bender--Knuth involutions, and following Chmutov, Glick and Pylyavskyy she has obtained a presentation for the cactus group $\gls{Jn}$ in terms
of shifted Bender--Knuth generators.
In the same vein, it is natural to seek precise relationships between subsets of relations in the two groups $\gls{vBK2n}$ and the virtual
symplectic cactus group $\tilde J_n$. It is also natural to seek a presentation of the virtual symplectic cactus group $\gls{vJ2n}$ in terms
of the virtual symplectic Bender--Knuth generators.

\vskip1cm

\section{Conflict of Interest}
The authors state that they have no conflict of interest.

\printglossary
\printglossary[type=acronym]
\printglossary[type=symbols]

\bibliographystyle{unsrt}

\end{document}